\documentclass[11pt,letterpaper]{amsart}
\usepackage[T1]{fontenc} 

\usepackage{amsmath,amsthm,amsfonts,amssymb,bm,graphicx,mathrsfs, dsfont, mathtools, mathdots}

\usepackage{tikz-cd}
\usepackage{adjustbox}

\usepackage{comment}
\usepackage%[backref=page]
{hyperref}

\hypersetup{colorlinks=true,citecolor=blue,linkcolor=blue,urlcolor=blue,
pdfstartview=FitH }

\theoremstyle{plain}
\newtheorem{lemma}{Lemma}
\newtheorem{theorem}[lemma]{Theorem}

\newtheorem{cor}[lemma]{Corollary}
\newtheorem{prop}[lemma]{Proposition}

\numberwithin{lemma}{section}
\numberwithin{equation}{section}

\theoremstyle{definition}

\newtheorem{hyp}{Hypothesis}

\theoremstyle{remark}
\newtheorem{remark}{Remark}[section]

\mathtoolsset{showonlyrefs=true}

\newcommand{\op}{{\rm op}}

\newcommand{\dd}{{\rm d}}
\makeatletter
\def\Ddots{\mathinner{\mkern1mu\raise\p@
\vbox{\kern7\p@\hbox{.}}\mkern2mu
\raise4\p@\hbox{.}\mkern2mu\raise7\p@\hbox{.}\mkern1mu}}
\makeatother

\begin{document}
\author{Edgar Assing}
\author{Valentin Blomer}
  
\address{Mathematisches Institut, Endenicher Allee 60, 53115 Bonn}
\email{blomer@math.uni-bonn.de}
\email{assing@math.uni-bonn.de}

\title{The density conjecture for principal congruence subgroups}

%\thanks{The first author was supported in part by the DFG-SNF lead agency program grant BL 915/2-2. The second author acknowledges support of the Max Planck Institute for Mathematics.}

\thanks{This research   was supported in part by the DFG-SNF Lead Agency Program grant BL 915/2-2 and Germany's Excellence Strategy grant EXC-2047/1 - 390685813. The first author was partially supported by ERC Advanced Grant 101054336.}

\begin{abstract}   We prove the spherical density conjecture of Sarnak for the principal   subgroup of ${\rm SL}_n( \Bbb{Z})$  of squarefree level and discuss various arithmetic applications. The ingredients include new bounds for local Whittaker functions and Kloosterman sums. 
\end{abstract}


\setcounter{tocdepth}{2}  \maketitle 

\maketitle

\section{Introduction}

\subsection{The density hypothesis}

Perhaps the most fundamental open conjecture in the theory of automorphic forms, in some aspects comparable to the Riemann hypothesis in the theory of $L$-functions,   is the generalized Ramanujan conjecture: unitary cuspidal automorphic representations for the group ${\rm GL}(n)$ are tempered. Even for $n=2$ it appears  completely out of reach. See \cite{Sa4, BB} for surveys.  Many applications, however, do not need the full strength of the Ramanujan conjecture, it suffices to quantify the possible exceptions and show that ``strong'' violations to the Ramanujan conjecture appear ``rarely'' (in a sense to be made precise below).  Such a quantitative approximation towards the Ramanujan conjecture is called a \emph{density theorem}. This is a familiar concept from the theory of $L$-functions, the most famous application  being the Bombieri--Vinogradov theorem. It can be proved by a density theorem for Dirichlet $L$-functions (see \cite[Section 10]{IK}), and its arithmetic content is (a precise version of) the statement that primes are equidistributed in residue  classes to large moduli; see \cite[Section 17]{IK}.   The Bombieri--Vinogradov theorem is one of the cornerstones in analytic number theory and a good substitute for the Riemann hypothesis in many cases. 

We   turn to the automorphic set-up. 
%\marginpar{\tiny{Maybe this needs to be corrected. We might need to be more precise with the adelic vs.\ classical picture. Throughout the paper there is some confusion between ${\rm SL}(n)$, ${\rm GL}(n)$, adelic etc. I think we only need the adelic version to make sure that we have local representations Hecke eigenvalues etc. \\ I have changed this slightly trying to avoid automorphic representations at this point. Does this make sense?}} 
%%For a congruence subgroup $\Gamma$ of ${\rm PGL}_n(\Bbb{Z})$
Let $\mathcal{F}$ be a finite family of %automorphic forms for ${\rm SL}_n(\Bbb{R})$. A typical example for such a family is the set of all 
  automorphic forms contributing to the spectral decomposition of $L^2(\Gamma \backslash {\rm SL}_n(\Bbb{R})/{\rm SO}_n(\Bbb{R}))$ for some congruence lattice $\Gamma \subset {\rm SL}_n(\Bbb{R})$. %We call this family $\mathcal{F}_{\Gamma}$. 
For $\varpi\in\mathcal{F}$  and %\marginpar{\tiny{Would it be a good solution  to make a forward reference that they correspond to adelic representations whose local components come with Langlands parameters? \\ Probably this is the best solution. I have tried to implement this. \\ \\ Somehow $\sigma_{\varpi}(v)$ mutated to $\sigma_{\varpi,v}$. I can change this back.}}
 an unramified place $v$ of $\Bbb{Q}$   we have a set of Langlands parameters $\{\mu_{\varpi,v}(j) \mid 1 \leq j \leq n\}$. We normalize these such that $\varpi$ is tempered at $v$ when all $\mu_{\varpi,v}(j)$ are purely imaginary and write  $\sigma_{\varpi,v} =  \max_j |\Re \mu_{\varpi,v}(j)|$.  For $\sigma \geq 0$ we write 
\begin{equation}\label{def11}
  N_v(\sigma, \mathcal{F}) = \#\{ \varpi \in \mathcal{F} \mid  \sigma_{\varpi,v} \geq \sigma\}
  \end{equation}%\sum_{\substack{  \varpi \in \mathcal{F}\\ \sigma_{\varpi,v} \geq \sigma}} 1$$
 for the number of automorphic forms in $\mathcal{F}$ \emph{violating the Ramanujan conjecture  at $v$ by $\sigma$}. The trivial bound is $N_v(\sigma, \mathcal{F}) \leq N_v(0, \mathcal{F}) = \#\mathcal{F}$, and if $\mathcal{F}$ contains the trivial representation, then $N_v(\frac{1}{2}(n-1), \mathcal{F})= 1$. 

In his 1990 ICM address \cite{Sa1}, see also \cite{SX}, Sarnak  considered this scenario  in the context of groups $G$ of real rank $1$, the  principal congruence subgroup %$\Gamma(q) = \text{ker}({\rm SL}_n(\Bbb{Z}) \rightarrow {\rm SL}_n(\Bbb{Z}/q\Bbb{Z}))$  
and $v = \infty$, and coined the phrase  density hypothesis for a  linear interpolation between these two trivial bounds for the tempered spectrum and the trivial representation.
 See  \cite{Sa2, PSa} for interesting recent applications  of this hypothesis in the context of Golden Gates and quantum computing.  

For the  rest of this paper, we consider the 
 (finite) family $\mathcal{F} = \mathcal{F}_{\Gamma(q)}(M)$ of cusp forms $\varpi$ for the principal congruence subgroup    
 \begin{displaymath}
 \begin{split}
   \Gamma(q)& = \text{ker}({\rm SL}_n(\Bbb{Z}) \rightarrow {\rm SL}_n(\Bbb{Z}/q\Bbb{Z})) \\
   &=   \left(\begin{matrix} 1 + q\Bbb{Z}& q\Bbb{Z}&  & \cdots & q \Bbb{Z}\\ q\Bbb{Z} & 1 + q\Bbb{Z} &  & \cdots& q\Bbb{Z}\\\vdots  &  & \ddots& &\vdots\\ 
	q\Bbb{Z} &  \cdots && 1 + q\Bbb{Z} &q\Bbb{Z} \\
	q \Bbb{Z} &     \cdots && q\Bbb{Z} & 1 + q\Bbb{Z}
	\end{matrix}\right) \cap {\rm SL}_n(\Bbb{Q})
 \end{split}
 \end{displaymath}
  with bounded spectral parameter $\| \mu_{\varpi,\infty}\| \leq M$. In this case, the spherical density conjecture reads 
\begin{equation*}%\label{nv}
N_v(\sigma, \mathcal{F}) \ll_{v, \varepsilon, M, n}[{\rm SL}_n(\Bbb{Z}) : \Gamma(q)]^{1 - \frac{2\sigma}{n-1} + \varepsilon}
\end{equation*}
where we think of $O_M([{\rm SL}_n(\Bbb{Z}) : \Gamma(q)])$ as the cardinality of $\mathcal{F}$ (by an appropriate version of Weyl's law). For applications it is desirable to have polynomial dependence in $M$. For $n=2$, this conjecture and archimedean variations thereof are well understood and proved in many typical situations and in different ways; see e.g.\ \cite{Hum, Hux, Iw, Sa7, FHMM}.  

\medskip 

For $n > 2$ the problem is completely open, even in special cases. Only very recently, the second author \cite{Bl} established (a stronger version of) the density conjecture for the Hecke congruence subgroup $\Gamma_0(p) \subseteq {\rm SL}_n(\Bbb{Z})$ for a large prime $p$,  using the Kuznetsov formula and a careful analysis of Kloosterman sums. An archimedean analogue was subsequently shown in \cite{Ja}.  Other than that we have no results for $n > 3$, neither for archimedean families, nor in the level aspect for any type of congruence subgroup. The authors of \cite{GK} argue ``\emph{that it seems that the problem is harder for principal congruence subgroups, and easier for group[s] that are far from being normal.}''  While it is too early to formalize this and also depends to some extent on the viewpoint, the principal congruence subgroup is certainly a very natural scenario and also interesting for applications. In this paper we solve the problem for the principal congruence subgroup of squarefree level \emph{in arbitrary dimension}. %This is supposedly not only a very hard case, but also an attractive one for applications.  

\begin{theorem}\label{thm1} Let $M, \varepsilon > 0$. Let $q \in \Bbb{N}$ be squarefree and let $\mathcal{F}=\mathcal{F}_{\Gamma(q)}(M)$ be the set of cuspidal automorphic forms for $\Gamma(q) \subseteq {\rm SL}_n(\Bbb{Z})$ with archimedean spectral parameter $\| \mu\| \leq M$. Fix a place $v$ of $\Bbb{Q}$, and if $v = p$ is finite, assume that $p\nmid q$. 
There exists a constant $K$ depending only on $n$, such that 
$$\mathcal{N}_v(\sigma, \mathcal{F}) \ll_{v, \varepsilon, n} M^K  [{\rm SL}_n(\Bbb{Z}) : \Gamma(q)]^{1 - \frac{2\sigma}{n-1} + \varepsilon}$$
for $\sigma \geq 0$. 
\end{theorem}

We will extend Theorem \ref{thm1} to the entire spectrum (not only the cuspidal one) in Theorem \ref{eisen} below. However, our method begins naturally with the cuspidal spectrum and in this way provides, as in \cite{Bl},  the option to improve the exponent \emph{beyond} the interpolation between the tempered and spectrum and the trivial representation (which is not cuspidal). See \cite[Proposition 4]{BM} for a special case in this direction. 
This is the beauty of the Kuznetsov formula, since the spectral side involves no residual spectrum. By the Luo-Rudnick-Sarnak  bounds \cite{LRS}, Theorem \ref{thm1} by itself is nevertheless non-trivial for  $\sigma < 1/2 - 1/(n^2 + 1)$.

We rephrase Theorem \ref{thm1} in a representation theoretical way. A spherical (i.e.\ unramified) representation $\pi_v$ of ${\rm GL}_n(\Bbb{Q}_v)$ is determined by a set of Langlands parameters $\{\mu_{\pi_v}(j) \mid 1 \leq j \leq n\}$. Again we normalize so that $\pi_v$ is tempered when all $\mu_{\pi_v}(j)$ are purely imaginary. We say that $\pi_v$ appears in $\mathcal{F}$ when there is $\varpi\in \mathcal{F}$ with $\mu_{\varpi,v} = \mu_{\pi_v}$ and write $m_{\mathcal{F}}(\pi_v)$ for the multiplicity of $\pi_v$ in $\mathcal{F}$, i.e.\  $m_{\mathcal{F}}(\pi_v) = \# \{ \varpi\in\mathcal{F}\colon \mu_{\varpi,v} = \mu_{\pi_v}\}$ (see also \eqref{dim-m} below).  If $\mathcal{F}=\mathcal{F}_{\Gamma(q)}(M)$, it is suggestive to write $m_{\mathcal{F}} = m_{\Gamma(q)}$. 

An immediate corollary of Theorem~\ref{thm1} is the following upper bound on these multiplicities. This proves  Conjecture 1 of Sarnak-Xue \cite{SX}  when generalized to arbitrary rank and an arbitrary place~$v$. 

\begin{cor} Under the same assumptions as in Theorem \ref{thm1} and with the same notation we have for an irreducible unramified representation $\pi_v$ of ${\rm GL}_n(\Bbb{Q}_v)$ the bound %\marginpar{\tiny{Do we want the real part in the exponent?}}
$$m_{\Gamma(q)}(\pi_v) \ll_{v, n, \varepsilon} \| \mu_{\pi_{\infty}} \|^K  [{\rm SL}_n(\Bbb{Z}) : \Gamma(q)]^{1 - \frac{2\sigma_{\pi_v}}{n-1} + \varepsilon}$$
where $\| . \|$ denotes the max-norm and $\sigma_{\pi_v} = \| \Re \mu_{\pi_v} \| =\max_j|\Re \mu_{\pi_v}(j)|$. 
\end{cor}

The restriction that $q$ be squarefree is used at two independent technical (but important) points in the argument: in the computation of Kloosterman sums in Lemma \ref{technical} and in the computation of local Whittaker functions for supercuspidal representations in Section \ref{sec6}. %Both of these, however,  do not appear to be fundamental and can likely be removed with 
It might be possible to remove this assumption with 
more work and somewhat different arguments, cf.\ the remarks at the end of Subsection \ref{12} and after Lemma \ref{technical}. 
 
We   prove several different versions of Theorem \ref{thm1} that we now discuss. 

The condition $\sigma_{\varpi,v}(v) \geq \sigma$ in the definition of $\mathcal{N}_v(\sigma, \mathcal{F})$ will be detected by a weight $T^{2\sigma_{\varpi,v}}$ for a parameter $T > 1$ in the trace formula. For several applications (including the ones in this paper) it is easier to work directly with this weight rather than the quantity $\mathcal{N}_v(\sigma, \mathcal{F})$; see Corollary \ref{cor-kuz} below. It is also possible to obtain a version of Theorem~\ref{thm1} on the level of cuspidal automorphic representations,  rather than cuspidal automorphic forms.  This is discussed in Remark~\ref{rem:aut_rep_pers} below.
 
Also for applications it is useful to include non-cuspidal automorphic representations into the bound to be able to handle a full spectral decomposition. By Langlands' classification of Eisenstein series, this is essentially an inductive process based on Theorem \ref{thm1}. This more general bound will be stated and proved in Section \ref{sec-eis} below, and we refer to Theorem \ref{eisen}. 

The prototype of our density result comes from an application of the Kuznetsov formula and therefore  involves unipotent periods (Fourier coefficients) as weights; see Proposition \ref{density-kuz}. This might be useful for some applications where these weights come up naturally. 
Removing them to obtain the classical unweighted density conjecture is a non-trivial task in the present situation. This would be simple for the group $\Gamma_0(q)$ by newform theory, but in absence of multiplicity one theorems and newform theory for general subgroups like $\Gamma(q)$ one has to obtain good lower bounds/asymptotics for local Whittaker functions on average over an orthonormal basis in each given representation. This is to some extent a problem of local representation theory and should have applications elsewhere. We discuss this in more detail below; see Theorem \ref{thm15} in the next subsection.  %It is only at this point that we use the squarefreeness of $q$.  This restriction   can be removed using slightly different techniques, as we discuss below. 

Finally we mention that there is a slightly different way to measure the failure of being tempered which avoids local Langlands parameters; cf.\ \cite{SX}. For a local representation $\pi$ let $p(\pi)$ be the infimum over $p \geq 2$ such that $K$-finite matrix coefficients of $\pi$ are in $L^p(Z(\Bbb{Q}_v)\backslash{\rm GL}_n(\Bbb{Q}_v))$. Then $\pi$ is tempered if $p(\pi) = 2$, and instead of $\sigma_{\varpi, v}$ one can use a function containing $p(\pi)$, e.g.\  $1 - 2/p(\pi)$ in the definition \eqref{def11} of $\mathcal{N}_v$. One can express $p(\pi)$ as a function of $\Re \mu_i$; see \cite[Lemma 3.2]{GGN}. In principle our method could also handle this alternative formulation of the density conjecture by using more general weight parameters in \eqref{weight}.

\subsection{The methods}\label{12} All approaches to automorphic density conjectures use some sort of trace formula which translates the spectral problem into an arithmetic problem. If we were to apply the pretrace formula to prove Theorem \ref{thm1}, we would essentially arrive at the counting problem described in Theorem \ref{thm2} below. This is the route that Sarnak-Xue \cite{SX} used for $n=2$. However, in higher rank we do not know how to solve this problem independently. Instead, %, so we turn tables and derive it as a consequence of the density conjecture. Instead, 
we use the Kuznetsov formula   to derive Theorem \ref{thm1}; which also leads to a counting problem, but a rather different one. The general framework of the Kuznetsov formula (see Section \ref{sec5}) was already developed in \cite{Bl}, but a trace formula is only as powerful as its input on the geometric and/or the spectral side. Here we need to develop new tools on both sides, and our application requires \emph{sharp bounds} on both sides. 

The arithmetic work consists  in analyzing Kloosterman sums for the group $\Gamma(q)$ in higher rank. The investigation of such Kloosterman sums is, to the best of the authors' knowledge, new. It turns out that the critical Weyl element is the element $w_{\ast}$ as defined in \eqref{wast} below, %$w_{\ast} = \left(\begin{smallmatrix} && -1\\ & I_{n-2} & \\ 1 && \end{smallmatrix}\right)$ 
and here we obtain sharp upper bounds for the corresponding Kloosterman sums. The key result is Theorem \ref{thm32} below; see also Subsection \ref{sec14}. 

Being a relative trace formula with respect to two copies of unipotent upper triangular matrices, the Kuznetsov formula contains Fourier coefficients on the spectral side. It is most convenient to work with the conjugated group %\marginpar{\tiny{The new notation should be: $\Gamma(q)$ for the principal congruence subgroup and $C(q)$ for the adelic version ($C_v(q)$ locally). Everything that is twisted by the diagonal matrix gets a $\natural$ sup-script. (Also the adelic/local additive characters.) I hope I made these changes consistently.}}
$$\Gamma(q)^{\natural} = \text{diag}(q^{n-1}, \ldots, q, 1) ^{-1} \Gamma(q) \text{diag}(q^{n-1}, \ldots, q, 1), $$
the advantage being that the standard character on the unipotent group $U$ is trivial on $\Gamma(q)^{\natural}$ and $\Gamma(q)^{\natural}$ contains $U(\Bbb{Z})$. Therefore the Fourier expansion with respect to $\Gamma(q)^{\natural}$ is similar to the one for $\Gamma_0(q)$ and in contrast to $\Gamma(q)$ has no ``denominators''. 
For a member $\varpi$ of an orthonormal basis of $L^2(\Gamma(q)^{\natural}\backslash {\rm SL}_n(\Bbb{R})/{\rm SO}_n(\Bbb{R}))$  with respect to the inner product as defined in  \eqref{standinner} below
% $\langle \varpi_1, \varpi_2\rangle = \int_{\Gamma_q\backslash {\rm GL}_n(\Bbb{R}
and $y = \text{diag}(y_{n-1} \cdots  y_1, \ldots, y_2y_1, y_1, 1)$, we denote by
$$ \mathcal{W}_{\varpi}(y) =  \int_{U(\Bbb{Z}) \backslash U(\Bbb{R})} \varpi(xy) \theta^{-1}(x) \dd x$$ 
its  ``first'' Fourier coefficient (with the standard notation to be recalled and introduced in the next section). For cuspidal $\varpi$ with spectral parameter $\mu=\mu_{\varpi,\infty}$ the Fourier coefficient  $\mathcal{W}_{\varpi}(y)$ is a multiple of the standard Whittaker function $W_{\mu}(y_1, \ldots, y_{n-1})$  as defined in \eqref{wmu} below, and we write
$$ \mathcal{W}_{\varpi}(y) = A_{\varpi}(1) W_{\mu}(y_1, \ldots, y_{n-1}).$$
It may well happen that $A_{\varpi}(1) = 0$. An important ingredient in the proof, which may have applications elsewhere, is the following essentially sharp lower bound of independent interest, which shows that in a certain average sense the first Fourier coefficient is at least as big as expected. 

\begin{theorem}\label{thm15} Let $q$ be squarefree. For an irreducible cuspidal automorphic representation $\pi$ on $X_q^{\Bbb{A}}$ (as defined in \eqref{227a} below) let $V_{\pi}\subseteq L^2(\Gamma(q)^{\natural}\backslash {\rm SL}_n(\Bbb{R})/{\rm SO}_n(\Bbb{R}))$ be the (finite dimensional) space of cusp forms whose adelic lifts generate $\pi$. For any $\varepsilon > 0$ we have $$\frac{1}{\dim V_{\pi}}\sum_{\varpi \in \text{\rm ONB}(V_{\pi})} |A_{\varpi}(1)|^2 \gg (\| \mu_{\pi_{\infty}} \| q)^{-\varepsilon} \frac{ q^{n(n-1)(n-2)/6}}{q^{n^2 - 1}} ,$$
where ${\rm ONB}$ refers to an orthonormal of $V_{\pi}$. 
\end{theorem}
Here the denominator is essentially the index $[ {\rm SL}_n(\Bbb{Z}) : \Gamma(q)]$ and the numerator is the index $[\Gamma(q)^{\natural} \cap U(\Bbb{Q}) : U(\Bbb{Z})]$. The proof of Theorem \ref{thm15} depends ultimately on classification results for $p$-adic representations, and for computations it is useful to note that the principal congruence subgroup is not only normal, but also admits an Iwahori factorization with respect to all standard parabolic subgroups (so it is ``good'' in the sense of \cite[Section 23]{NV}). %Suppose for instance that $q = p$ is a prime. 
Locally, we need to understand the average
\begin{equation}\label{int}
	\sum_{v\in {\rm ONB}(\pi_p^{\Gamma(q)^{\natural}})} \int^{\text{st}}_{U(\Bbb{Q}_p)} \theta^{-1}(x) \langle \pi(x) v, v\rangle \dd x,
\end{equation}
where ${\rm ONB}(\pi_p^{\Gamma(q)^{\natural}})$ is any orthonormal basis of $\Gamma(q)^{\natural}$-fixed vectors in $\pi_p$ and $\int^{\text{st}}$ denotes the stable integral as defined in \cite[Definition~2.1]{LM}. Suppose for instance that $\pi_p$ is a depth-zero supercuspidal representation. Any such representation can be obtained from a cuspidal representation $\pi_0$ of ${\rm GL}_n(\Bbb{Z}/p\Bbb{Z})$. Then the integral \eqref{int} depends only on the character of $\pi_0$ at elements in $ U(\Bbb{Z}/p\Bbb{Z})$, %, which in turn is constructed from a character of $\Bbb{Z}/p\Bbb{Z}$, 
so one can compute \eqref{int} quite explicitly. For $q = p^k$ with $k > 1$, the bottleneck is the treatment of supercuspidal representations of higher depth. It should be possible to treat these using more general constructions via induction from open compact subgroups. The methods in \cite{NV} provide robust  tools in the archimedean (microlocal) case, and it should be possible to translate them to the $p$-adic setting; see also \cite[Section 2.5]{Ne}. We hope to return to this in the future.

\subsection{Two arithmetic counting problems}\label{appl} As observed by Sarnak, closely related to the density conjecture for the group $\Gamma(q) \subseteq {\rm SL}_n(\Bbb{Z})$ are two very concrete arithmetic applications that are in some sense the analogue of the distribution of primes in arithmetic progressions for density theorems of $L$-functions. Both of them only need the case $v = \infty$ in Theorem \ref{thm1}. 

The first one is a rather innocent looking counting  problem:  Count the number of matrices $\gamma \in \Gamma(q)$ of norm $\leq T$. As long as $q$ is fixed and $T$ tends to infinity, this has been established in much broader context in \cite{DRS}. The problem gets much harder if we are aiming for a bound that is uniform in $q$. By a volume argument, we expect something of the size $T^{n(n-1)}/q^{n^2-1}$. In analogy with primes $p \leq x$ in arithmetic progressions modulo $q$, where the Riemann hypothesis predicts a main term of size (roughly) $x/q$ and  an error term of size (roughly) $x^{1/2}$, one may optimistically hope to obtain a correction term of size $T^{n(n-1)/2}$ for the matrix counting problem. This was conjectured by Sarnak-Xue \cite{SX} in rank 1 and proved by an essentially elementary technique, which then implied the density conjecture in rank 1. In higher rank, this counting problem becomes surprisingly hard and has remained completely open until now (even for $n=3$, say).  Here we turn tables and try to derive it as a consequence of the density conjecture.\footnote{This leads to an interesting logical sequence which in rather simplified form looks as follows: in order to solve the matrix  counting problem, we apply the pretrace formula to translate it into a spectral problem, add artificially unipotent periods, apply the Kuznetsov formula backwards, and obtain a very different arithmetic problem of counting double cosets in the definition of Kloosterman sums that we are able to solve in a best possible way.}

The second application is a beautiful lifting property. As is well-known, the map ${\rm SL}_n(\Bbb{Z}) \rightarrow {\rm SL}_n(\Bbb{Z}/q\Bbb{Z})$ is surjective. Given some $g \in {\rm SL}_n(\Bbb{Z}/q\Bbb{Z})$, how small is the smallest preimage? This is the automorphic analogue of finding the smallest prime in an arithmetic progression. 
Since the term $T^{n(n-1)}/q^{n^2-1}$ in Theorem \ref{thm2} comes from the volume and cannot be improved, by the pigeonhole principle we cannot ask for anything better than a preimage of norm $q^{1 + 1/n}$. But even this is not true in general. Already for $n = 2$ and $q$ a large power of 2, Sarnak \cite{Sa2} constructed a matrix $g$ whose smallest preimage is as big as $q^{8/3}$. However, generically, i.e.\ for almost all $g \in {\rm SL}_n(\Bbb{Z}/q\Bbb{Z})$, the pigeonhole principle should give the correct answer. This is Sarnak's \emph{optimal lifting property}. See \cite{GK, KL, JK} from more further discussion. 

A crucial input for both applications is a density theorem, and interestingly a spectral gap (even a numerically strong one) alone does not suffice.  For an illuminating discussion of the various connections and logical implications of the density conjecture and the associated counting problems for co-compact subgroups see \cite{GK}.

 In the present case of a non-cocompact congruence subgroup, one has to treat Eisenstein series in the pre-trace formula (see Section \ref{sec8}), and to this end we need some information on the $L^2$-growth of Eisenstein series. Our results are conditional on the assumption that the norms of all relevant truncated (adelic) Eisenstein series for ${\rm GL}_n(\Bbb{A})$ are essentially bounded in the $q$-aspect and of at most polynomial growth in the spectral aspect:
\begin{equation}\label{11}
\| \Lambda^T E(.,\phi,\lambda)\|^2    \ll_{T, \varepsilon} q^{\varepsilon} | \mu_{\sigma_{\infty,\lambda}}|^{O(1)}.
\end{equation}
For the notation and the precise conditions we refer to Section \ref{sec8} where it is formulated as Hypothesis \ref{hyp_A}. We discuss this hypothesis as well as recent developments to remove it after stating the following two theorems.

\begin{theorem}\label{thm2} Assume Hypothesis \ref{hyp_A}. Let $q$ be squarefree, $T \geq 1$, $\varepsilon > 0$.  The number of matrices $\gamma \in \Gamma(q) \subseteq {\rm SL}_n(\Bbb{Z})$ of norm not exceeding $T$ is  
$$\ll_{\varepsilon, n} (Tq)^{\varepsilon}\Big( \frac{T^{n(n-1)}}{q^{n^2 - 1}} + T^{n(n-1)/2}\Big).$$
\end{theorem}
 
As explained above, the key point of this bound is that the second order term features square-root cancellation in the $T$-aspect (and hence is  ``of Ramanujan quality''). With slightly more work, this could be turned into an asymptotic formula with a main term of size $T^{n(n-1)}/q^{n^2-1}$ and an error term of form $T^{n(n-1) - \delta}$ for some $\delta > 0$. This has been obtained in \cite[Corollary 5.2]{GN}. It is easy to see (cf.\ \cite[Section 2.6]{GK}) that a lower bound is given by $T^{n(n-1)}/q^{n^2-1} + (T/q)^{n(n-1)/2} + 1$. % with a main term of size $

\begin{theorem}\label{thm3} Assume Hypothesis \ref{hyp_A}. For any $\varepsilon > 0$ there exists $\delta > 0$ with the following property. Let $q$ be squarefree. Then the   number of matrices $\gamma \in   {\rm SL}_n(\Bbb{Z}/q\Bbb{Z})$ without a lift of norm $\leq q^{1 + 1/n + \varepsilon}$ is at most $O(q^{n^2 - 1 - \delta})$. \end{theorem}

In other words,  since  $| {\rm SL}_n(\Bbb{Z}/q\Bbb{Z})| \asymp q^{n^2 - 1 }$, ``almost all''  matrices $\gamma \in   {\rm SL}_n(\Bbb{Z}/q\Bbb{Z})$ have the optimal lifting property.

By Langlands'   inner product formula  (also referred to as the  Maa{\ss}-Selberg relations), the left hand side of \eqref{11}   can be   expressed in terms of quotients of Rankin-Selberg $L$-functions  $L(s, \pi_1 \times \pi_2)$ for $\Re s \geq 1$ for cusp forms $\pi_j$ on ${\rm GL}_{n_j}(\Bbb{A})$ with $n_1 + n_2 \leq n$. To show that for residual Eisenstein series no values of $s$ inside the critical strip contribute requires non-trivial cancellation of terms in the Maa{\ss}-Selberg relations; see \cite{Mi} for $n=3$ and      \cite{JK} for the details in the general case (cf.\ also \cite{GL}). In particular, in this way we see that Hypothesis \ref{hyp_A} follows from the following weak zero-free region (and hence a fortiori from GRH):    whenever  $n_1+n_2\leq n$ there exists $A > 0$ such that we have
\begin{equation}\label{zero}
	L(1 + it, \pi_1 \times \tilde{\pi}_2) \gg c_{\text{fin}}(\pi_1\times\tilde{\pi}_2)^{-\varepsilon} \big((1 + |t|)c_{\infty}(\pi_1 \times \tilde{\pi}_2)\big)^{-A} \end{equation}
for all cuspidal automorphic representations $\pi_i$ of ${\rm GL}_{n_i}(\Bbb{A})$ with $c_{\text{fin}}(\pi_i)\mid q^{n_i}$. In particular, by functoriality of ${\rm GL}(2) \times {\rm GL}(2)$ Rankin-Selberg convolution \cite{Ra} together with the standard zero-free region \cite[Theorem 5.10]{Iw}, Siegel's bound \cite{Si} and  the non-existence of Siegel zeros for standard $L$-functions on ${\rm GL}(2)$ \cite{HL}, standard $L$-functions on ${\rm GL}(3)$ \cite{Ba} and Rankin-Selberg $L$-functions on ${\rm GL}(2) \times {\rm GL}(2)$ \cite{RW}, this is known for $n \leq 4$ (for which  Theorems \ref{thm2} and \ref{thm3} become unconditional). See \cite{Br} for general zero-free regions (which however is not of the strength required in \eqref{zero}).

 An inspection of  the proof of Lemma \ref{local} shows that we only need a certain averaged version of Hypothesis \ref{hyp_A}. After this paper was submitted, Jana and Kamber managed to prove this, cf.\ \cite[Section 1.3]{JK}. In particular, Theorems \ref{thm2} and \ref{thm3} are now unconditional.

\subsection{Some remarks on the numerology}\label{sec14} We end the introduction with a quick synopsis of the exponents appearing in the various statements presented so far. %Some of the following notation will be introduced only later, but is suggestive enough to make the meaning 
 The spectral side of the Kuznetsov formula (Lemma \ref{kuz-formula}) for the group $\Gamma(q)^{\natural}$ features an expression of the type
 \begin{equation}\label{lhs}
 \int_{  \mu_{\varpi} \ll 1} |A_{\varpi}(1)|^2 T^{2\sigma_{\varpi,v}} \dd \varpi
 \end{equation}
(see Section \ref{sec25} for the notation). As long as 
\begin{equation}\label{aslongas}
T \ll q^{n+1},
\end{equation}
 only two Weyl group elements contribute to the Kloosterman side: the trivial Weyl element and $w_*$, as defined in \eqref{wast} below.  The Kloosterman sums corresponding to the other Weyl elements vanish\footnote{For the group $\Gamma_0(q)$ one can derive the original density conjecture (not the improved version of \cite{Bl}) without any analysis of Kloosterman sums, but it is interesting to note that this is not possible here -- the analysis of $w_{\ast}$ is absolutely crucial.}, since their moduli are divisible by such high $q$-powers that they are outside the support of test function on the Kloosterman side (Lemma \ref{lem41}). 
The contribution of the trivial Weyl element is of size  $\mathcal{N}_q = q^{n(n-1)(n-2)/6}$ as defined in \eqref{vqnq}. For $w_{\ast}$, the admissible moduli $c \in \Bbb{Z}^{n-1}$ satisfy $q^n \mid c_j \ll T \ll q^{n+1}$, so that they are all of the form $c_j= q^n\gamma_j$ with $\gamma_j \ll q$. For simplicity let us assume $(\gamma_j, q) = 1$. We factor the Kloosterman sums to moduli $c$ into one with moduli $(q^n, \ldots, q^n)$ which is bounded in Theorem \ref{thm32} by $\mathcal{N}_qq^{(n-1)^2}$. This is sharp. The remaining unramified Kloosterman sum with moduli $\gamma$ is estimated trivially. In this way  we conclude that the contribution of $w_*$ is $\ll \mathcal{N}_q T^{n-1}q^{1-n^2}$. By \eqref{aslongas} this is not  more than the trivial Weyl element (which is the reason why \eqref{aslongas} is the critical bottleneck). We conclude that   \eqref{lhs} is bounded by $\mathcal{N}_q$, provided that \eqref{aslongas} holds. Now Theorem \ref{thm15} implies that the weight $|A_{\varpi}(1)|^2$ can essentially be dropped at the cost of a factor $\mathcal{N}_q /\mathcal{V}_q$, as defined in \eqref{vqnq}.  Dividing both sides by this factor, we conclude
  \begin{equation*}%\label{lhs}
 \int_{  \mu_{\varpi} \ll 1}  T^{2\sigma_{\varpi,v}} \dd \varpi \ll \mathcal{V}_q^{1+o(1)} ,  
 \end{equation*}
 provided that \eqref{aslongas} holds.  Of course, on the left hand side we can replace $\Gamma(q)^{\natural}$ with $\Gamma(q)$, and the right hand side is essentially $[{\rm SL}_n(\Bbb{Z}) : \Gamma(q)]$, so this bound is best possible, the only restriction being \eqref{aslongas} which limits the size of the   penalty $T^{2\sigma_{\varpi,v}}$ for large values of $\sigma_{\varpi,v}$. Theorem \ref{thm1} is now an easy consequence (see Section \ref{sec6}). Any numerical improvement in Theorem \ref{thm1} needs a relaxed version of \eqref{aslongas}, but as remarked above, even in the most optimistic scenario this is rather small (and has no influence on the applications discussed in Subsection \ref{appl}).     
  
 \medskip
 
\textbf{A roadmap:} Section \ref{sec2} contains some basic notation. For the reader's convenience we also establish a careful dictionary between the classical and the adelic framework. Section \ref{sec3} is technical in nature and refines \cite[Section 5]{Bl}; this is needed to have polynomial control on $M$ in Theorem \ref{thm1} and all subsequent results. Section \ref{sec4} is the key   input on the arithmetic side of the Kuznetsov formula: a careful analysis of   Kloosterman sums for the principal congruence subgroup. Section \ref{sec5} sets up the Kuznetsov formula, while Section \ref{sec6} provides the key input on the spectral side of   the Kuznetsov formula: lower bounds for local Whittaker functions. In Section \ref{sec-eis} we add the non-cuspidal spectrum to the picture, while the final Section \ref{sec8} deals with applications. 

%\marginpar{\tiny{To be completed later.}}

\section{Basic notation}\label{sec2}

\subsection{Basic conventions}\label{sec21} We regard $n$ as (arbitrary,  but) fixed. For convenience we assume $n \geq 3$, since all statements in the introduction are known for $n=2$. The phrase ``absolute constant'' refers to a constant that may depend on $n$. In this terminology, the letter $K$ denotes a sufficiently large absolute constant, not necessarily the same on each occurrence. %, while the letter $\kappa$ denotes a sufficiently small positive constant, not necessarily the same at each occurrence. 

We write $a \mid b^{\infty}$ to mean that $a$ has only prime factors occurring in $b$. Similarly $(a, b^{\infty})=\sup_{n\in \mathbb{N}} (a,b^n)$ is the greatest divisor of $a$ whose prime divisors divide $b$. 
We denote by $\| . \|$ the maximum norm on vectors, the maximum norm on matrices and the sup-norm on functions. %We denote by $\| . \|_F$ the Frobenius norm on matrices. 
We denote by $I_d$  the identity matrix of dimension $d$.

% We use the notation of \cite[Section 2]{Bl}.
Let $%\mathfrak{u}, 
U \subseteq {\rm GL}_n$ be the subgroup of %nilpotent resp.\ 
unipotent upper triangular matrices. %\marginpar{\tiny{I needed the symbol $N$ for the unipotent radical of parabolic subgroups. Because I could not think of anything better I changed $N$ for nilpotent upper triangular matrices to $\mathfrak{u}$. Where are the nilpotent matrices used?}} 
 Let $V \subseteq {\rm GL}_n$ be the group of   diagonal matrices with entries $\pm 1$. Let $T \subseteq {\rm GL}_n$ be the diagonal torus, $\tilde{T}(\Bbb{R}) \subseteq {\rm GL}_n(\Bbb{R})$ the set of diagonal matrices with positive entries and let ${\rm Z}$ denote the center of ${\rm GL}_n$. We write ${Z}^{+} \cong  \Bbb{R}_{>0}$ for the subgroup of diagonal scalar matrices  with positive entries and  $\mathcal{H} = U(\Bbb{R}) \tilde{T}(\Bbb{R})/{Z}^+$. Write $T_0=T\cap {\rm SL}_n$ for the diagonal torus with determinant one.%\marginpar{\tiny{I changed $\mathcal{H}$ by modding out the center. Is that what was intended? Otherwise I don't understand \eqref{standinner} since $\tilde{T}(\Bbb{R})$ is $n$-dimensional as defined above but the measure $d^{\ast}y$ seems $n-1$ dimensional.}}

We will write $P=MN$ for a standard parabolic subgroup and its Levi decomposition. In particular there is a partition $n=n_1+\ldots+n_k$ so that %\marginpar{\tiny{Should this be ${\rm GL}_{n_1}\times \ldots \times {\rm GL}_{n_k}$? \\ Fixed.}} 
$M\cong {\rm GL}_{n_1}\times \ldots\times {\rm GL}_{n_k}$ and $N\subseteq U$. Further let $P^{\op}=M^{\op}N^{\op}$ be the standard parabolic subgroup conjugate to $P^{\top}$.

Let $W$ be the Weyl group. For our purpose it is important to choose representatives with determinant 1 (since the principal congruence group $\Gamma(q)$ has no elements with determinant $-1$ for $q > 2$), so we identify $W$ with  permutation matrices where potentially the entry in first row (say) is $-1$. 
We identify a permutation matrix $w = (w_{ij}) \in W$ with the  permutation $ i \mapsto j$ for $w_{ij} = \pm 1$. For $w \in W$ we define 
\begin{equation}\label{Uw}
U_w = w^{-1} U^{\top} w \cap U.
\end{equation} Important Weyl elements for us are the special element $w_{\ast}$ and the long element $w_l$ defined by
\begin{equation}\label{wast}
w_{\ast} = \left(\begin{smallmatrix} && -1\\ & I_{n-2} &\\ 1 && \end{smallmatrix}\right)\quad  \text{ and } \quad w_{l} = \left(\begin{smallmatrix} && \pm 1\\ & \iddots &\\ 1 && \end{smallmatrix}\right).
\end{equation}
%where generally $I_d$ denotes the identity matrix of dimension $d$.  

\subsection{Embeddings, coordinates and measures}

The Haar measure on $U(\Bbb{R})$ is given by    $\dd x = \prod_{1 \leq i < j \leq n} \dd x_{ij}$. We write  $\dd x$ also for the induced measure on the subgroup $U_{w}(\Bbb{R})$. 

We embed   $y = (y_1, \ldots, y_{n-1}) \in \Bbb{G}_m^{n-1}$ into $T$ as 
\begin{equation}\label{iota}
\iota(y) = \text{diag}(y_{n-1} \cdots  y_1, \ldots, y_2y_1, y_1, 1) = \text{diag}\big((y_1 \cdots y_{n-j})_{1 \leq j \leq n}\big)
\end{equation}
with the empty product defined as 1. We multiply two elements in $y, y' \in \Bbb{G}_m^{n-1}$ componentwise, written $y \cdot y'$, so that $\iota$ is a homomorphism. 
 The Iwasawa decomposition gives a unique decomposition $g = x y k \alpha \in {\rm GL}_n(\Bbb{R})$ with $x \in U(\Bbb{R})$, $y \in \tilde{T}(\Bbb{R})$, $k \in {\rm O}_n(\Bbb{R})$, $\alpha \in {Z}^+$.  We write ${\rm y}(g) = \iota^{-1}y \in \Bbb{R}_{>0}^{n-1}$ for $(n-1)$-tuple of  Iwasawa $y$-coordinates.  In particular, for $g = \text{diag}(y_1, \ldots, y_n) \in \tilde{T}(\Bbb{R})$ we have 
\begin{equation}\label{y}
  {\rm y}(g) = \Big(\frac{y_{n-1}}{y_n}, \ldots, \frac{y_1}{y_2}\Big) \in \Bbb{R}_{>0}^{n-1}.
  \end{equation}
 For $w \in W$, $y \in \Bbb{R}_{>0}^{n-1}$ we write 
 \begin{equation}\label{wy}
   {\rm y}(w\iota(y)^{-1} w^{-1}) =: {}^wy  := ({}^wy_1, \ldots, {}^wy_{n-1})
   \end{equation}
    for the Iwasawa $y$-coordinates of $w\iota(y)^{-1} w^{-1}$.  

For $\alpha\in \Bbb{C}^{n-1}$, $y \in \Bbb{R}_{>0}^{n-1}$ we write $y^{\alpha} = y_1^{\alpha_1} \cdots y_{n-1}^{\alpha_{n-1}}\in \Bbb{C}$. 
Let 
\begin{equation}\label{eta}
\eta = (\eta_1, \ldots, \eta_{n-1}) =   \Big(\frac{1}{2}j(n-j)\Big)_{1 \leq j \leq n-1} \in \Bbb{Q}^{n-1}. 
\end{equation}
We define a measure on $\Bbb{R}_{>0}^{n-1}$   by 
\begin{equation}\label{measure}
\dd^{\ast}y = y^{-2\eta} \frac{\dd y_1}{y_1} \cdots \frac{\dd y_{n-1}}{y_{n-1}}
\end{equation}
 and correspondingly an inner product   by 
\begin{equation}\label{inner}
\langle f, g \rangle = \int_{\Bbb{R}_{>0}^{n-1}} f(y) \bar{g}(y) \dd^{\ast}y.
\end{equation}
We denote the push forward of  $\dd^{\ast}y$ to $\tilde{T}(\Bbb{R})$ by $\iota$ also by $\dd^{\ast}y$. 
%Then $\dd x \, \dd^{\ast}y$ is a left ${\rm GL}_n(\Bbb{R})$ invariant measure on $\mathcal{H}$. 

 We define a different embedding of $\Bbb{R}_{>0}^{n-1}$ into $T(\Bbb{R})$  by 
 \begin{equation}\label{cast}
 c = (c_1, \ldots, c_{n-1}) \mapsto c^{\ast} = \text{diag}\Big(\frac{1}{c_{n-1}}, \frac{c_{n-1}}{c_{n-2}}, \ldots, \frac{c_2}{c_1}, c_1\Big).
 \end{equation}

\subsection{Subgroups and characters}
 For  $N   \in \Bbb{Z}^{n-1}$ we define a character $\theta_N : U(\Bbb{R})/U(\Bbb{Z}) \rightarrow S^1$ by 
\begin{equation}\label{character}
   \theta_N(x) = e(N_{n-1} x_{12} + \ldots + N_{1} x_{n-1, n})% = \theta_{(1, \ldots, 1)}(N x N^{-1})
.\end{equation}
 For $v \in V$ we write $\theta_N^v(x) = \theta_N(v^{-1} x v)$. % If $N = (1, \ldots, 1)$, we drop it from the notation of the character. 

Let $q \in \Bbb{N}$ and $\Gamma(q) %= \text{ker}({\rm SL}_n(\Bbb{Z}) \rightarrow {\rm SL}_n(\Bbb{Z}/q\Bbb{Z})) 
\subseteq {\rm SL}_n(\Bbb{Z})$  be the principal congruence subgroup. Let 
\begin{equation}\label{dq}
   D_q =  \text{diag}(q^{n-1}, \ldots, q, 1)
\end{equation} 
and write 
\begin{equation}\label{gq}
	\Gamma(q)^{\natural} = D_q^{-1}\Gamma(q) D_q  = \left(\begin{matrix} 1 + q\Bbb{Z}& \Bbb{Z}& \frac{1}{q}\Bbb{Z} & \cdots & \frac{1}{q^{n-2}} \Bbb{Z}\\ q^2\Bbb{Z} & 1 + q\Bbb{Z} & \Bbb{Z} & \cdots& \frac{1}{q^{n-3}}\Bbb{Z}\\\vdots  & \ddots & \ddots&\ddots&\vdots\\ 
	q^{n-1}\Bbb{Z} &  \cdots && 1 + q\Bbb{Z} &\Bbb{Z} \\
	q^n \Bbb{Z} & q^{n-1}\Bbb{Z} &  \cdots & q^2\Bbb{Z} & 1 + q\Bbb{Z}
	\end{matrix}\right) \cap {\rm SL}_n(\Bbb{Q}).
\end{equation}
In contrast to $\Gamma(q)$, the conjugated group $\Gamma(q)^{\natural}$ has the advantage   that 
\begin{equation}\label{adv}
   U(\Bbb{Z}) \subseteq \Gamma(q)^{\natural} \quad \text{and} \quad \theta_N \text{ is trivial on } \Gamma(q)^{\natural}\cap U(\Bbb{R}).
\end{equation}
In particular, $U(\Bbb{Z})$ acts on $\Gamma(q)^{\natural}$ both from the left and the right. Let
\begin{equation}\label{vqnq}
\begin{split}
&	\mathcal{V}_q = [{\rm SL}_n(\Bbb{Z}) : \Gamma(q)]  = q^{n^2 - 1 + o(1)},\\
	& \mathcal{N}_q = [\Gamma(q)^{\natural} \cap U(\Bbb{Q}) : U(\Bbb{Z})] = q^{n(n-1)(n-2)/6}.
	\end{split}
\end{equation}
These are the quantities appearing in Theorem \ref{thm15}. 

\subsection{Whittaker functions}
For $\mu = (\mu_1, \ldots, \mu_n) \in \Bbb{C}^{n}$  satisfying
  \begin{equation}\label{unit}
     \mu_1 + \ldots + \mu_n = 0, \quad \{\mu_1, \ldots, \mu_n\} = \{-\bar{\mu}_1, \ldots ,- \bar{\mu}_n\}
     \end{equation}
we define   $\nu \in \Bbb{C}^{n-1}$ by
 $$    \nu_j = \frac{1}{n} (\mu_j - \mu_{j+1}), \quad 1 \leq j \leq n-1.$$
We denote by  ${W}_{\mu} : \Bbb{R}_{>0}^{n-1} \rightarrow \Bbb{C}$   the normalized Whittaker function 
\begin{equation}\label{wmu}
{W}_{\mu}(y) = \int_{U(\Bbb{R})} I_{\nu}\big( %\Big(\begin{smallmatrix}&& 1\\ & \Ddots &\\ 1 &&\end{smallmatrix}\Big)
w_l x \iota(y) \big) \theta_{(1, \ldots, 1)}(x)^{-1} \dd x
\end{equation}
with the power function
\begin{displaymath}
  I_{\nu}(g) = \prod_{i, j = 1}^{n-1} |{\rm y}(g)_{i}|^{b_{ij} (\frac{1}{n}+ \nu_j)}, \quad b_{ij} = \begin{cases} ij, & i+j \leq n,\\ (n-i)(n-j), & i+j \geq n. \end{cases}
\end{displaymath}  
This differs from the completed Whittaker function by a factor
\begin{equation}\label{norm}
G(\mu) :=   \prod_{1 \leq j \leq k \leq n-1} \Gamma_{\Bbb{R}} (1 + n(\nu_j+\ldots + \nu_k)), \quad \Gamma_{\Bbb{R}}(s) = \Gamma(s/2)\pi^{-s/2}.
\end{equation}

\subsection{Automorphic forms}\label{sec25}

In the interest of a compact notation, for a congruence subgroup $\Gamma\subset {\rm SL}_n(\Bbb{R})$ we denote by 
\begin{equation}\label{spec}
\int_{\Gamma} \mathcal{V}_{\varpi}\dd \varpi
\end{equation}
a spectral decomposition of $L^2(\Gamma\backslash {\rm SL}_n(\Bbb{R})/{\rm SO}_n(\Bbb{R}))$ respecting the unramified Hecke algebra. Note that the spaces $\mathcal{V}_{\varpi}$ are one-dimensional, and we use $\varpi$ to denote a suitably normalized generator of $\mathcal{V}_{\varpi}$.  
The  standard   inner product is
\begin{equation}\label{standinner}
\langle \varpi_1, \varpi_2 \rangle = \int_{\Gamma \backslash \mathcal{H}} \varpi_1(xy) \overline{\varpi_2(xy)} \dd x\, \dd^{\ast}y.   
\end{equation}
% $\langle f, g \rangle = \int_{\Gamma_q\backslash \mathcal{H}} f(xy)  \bar{g}(xy)  \dd x\, \dd^{\ast}y$. 
The relevant spectral decomposition is a special case of Langlands' general theory; see e.g.\ \cite{Ar} for a convenient summary in adelic language. Each element $\varpi$ in the decomposition above has spectral parameter $\mu_{\varpi} = (\mu_1, \ldots, \mu_n)$ related to the eigenvalues of invariant different operators. This spectral parameter agrees with the local archimedean Langlands parameter $\mu_{\pi_{\infty}}$ of the associated automorphic representation $\pi=\otimes_v\pi_v$ and satisfies \eqref{unit}.

We now specialize to the group $\Gamma = \Gamma(q)^{\natural}$, defined in \eqref{gq}. Set $$X_q=\Gamma(q)^{\natural}\backslash {\rm SL}_n(\Bbb{R})/{\rm SO}_n(\Bbb{R}).$$ For the adelisation process described below we will need the following decomposition of the space $L^2(X_q)$.  Let $\Gamma_1(q)\subset {\rm SL}_n(\Bbb{Z})$ be the set of matrices %$\gamma \equiv t \text{ mod }q$ for some $t\in T(\Bbb{Z}/q\Bbb{Z})$. 
that are diagonal modulo $q$. 
Put $\Gamma_1(q)^{\natural}=D_q^{-1}\Gamma_1(q)D_q$. Given $\boldsymbol{\chi}=(\chi_1, \ldots,\chi_{n-1})$, where $\chi_1,\ldots,\chi_{n-1}$ are Dirichlet characters modulo $q$, we set 
\begin{multline}
	L^2(X_q,\boldsymbol{\chi}) = \{ f\in L^2(X_q)\colon f(\gamma g) = \chi_1(\gamma_{11})\cdots \chi_{n-1}(\gamma_{n-1,n-1})\cdot f(g)\\  \text{ for all }g\in {\rm SL}_n(\Bbb{R}), \, \gamma = (\gamma_{ij})\in \Gamma_1(q)^{\natural}\},
\end{multline}
We have the spectral refinement
\begin{equation}
	L^2(X_q) = \bigoplus_{\boldsymbol{\chi}\in \widehat{T_0}(\Bbb{Z}/q\Bbb{Z})} L^2(X_q,\boldsymbol{\chi}).\nonumber
\end{equation}

For $\varpi \in L^2(X_q)$ and $N \in \Bbb{Z}^{n-1}$ we denote by
$$ \mathcal{W}_{\varpi}(y; N) =  \int_{U(\Bbb{Z}) \backslash U(\Bbb{R})} \varpi(xy) \theta_N^{-1}(x) \dd x$$ 
its $N$-th Fourier coefficient. This is a multiple of the Whittaker function $W_{\mu}$ as in \eqref{wmu} associated with the archimedean Langlands parameters $\mu = \mu_{\varpi}$, and using \eqref{eta} we decompose it as
$$\mathcal{W}_{\varpi}(y; N)  = \frac{A_{\varpi}(N)}{N^{2\eta}} W_{\mu}(N \cdot {\rm y}(y))$$
for some constant $A_{\varpi}(N) \in \Bbb{C}$. If $\lambda_{\varpi}(m)$ denotes the Hecke eigenvalue with respect to the Hecke operator $T_m$ for $m\in \Bbb{N}$, $(m, q) = 1$, then 
\begin{equation}\label{hecke}
   A_{\varpi}((m, 1, \ldots, 1)) = \lambda_{\varpi}(m) A_{\varpi}((1, \ldots, 1)).
   \end{equation}
For a prime $p \nmid q$ and $\nu > n$ we have
\begin{equation}\label{finite}
\max_{0 \leq j \leq n-1} | \lambda_{\varpi}(p^{\nu-j})| \geq (2 p^{\sigma_{\varpi,p}})^{1-n}   p^{\nu \sigma_{\varpi,p}}
\end{equation}
by \cite[Lemma 4]{Bl}.

\subsection{Adelising Automorphic Forms}\label{sec:adelising}

\subsubsection{The adelic space}

Let $\Bbb{A}$ denote the adele ring of $\Bbb{Q}$. Let $\psi_{\Bbb{A}} = \prod_v \psi_v$ be the usual adelic character on $\Bbb{Q}\backslash \Bbb{A}$, where $\psi_{\infty}=e(\cdot)$. We lift $\psi_{\Bbb{A}}$ to a character $$\boldsymbol{\psi}_{\Bbb{A}}\colon U(\Bbb{Q})\backslash U(\Bbb{A})\to\Bbb{C}^{\times}, \quad x\mapsto \psi(x_{12}+\ldots+x_{n-1,n}),$$ so that the archimedean part is $\boldsymbol{\psi}_{\Bbb{Q}_{\infty}}=\theta_{(1,\ldots,1)}$.

Given a primitive Dirichlet character $\xi$ we write $\omega_{\xi}\colon \Bbb{Q}^{\times}\backslash \Bbb{A}^{\times}\to S^1$ for the unique Hecke character so that the completed Dirichlet $L$-function associated to $\xi$ agrees with the $L$-function associated to $\omega_{\xi}$ by Tate. If $\xi$ is a non-primitive Dirichlet character, then we write $\omega_{\xi}$ for the Hecke character associated to the underlying primitive character (sometimes called the idelic lift of $\xi$). In either case we have the factorization $\omega_{\xi} = \prod_v \omega_{\xi,v}$.

We set $K_{\infty}={\rm SO}_n(\mathbb{R})$, $K_p={\rm GL}_n(\Bbb{Z}_p)$, $K_{{\rm fin}} = \prod_{p} K_p$ and $K = K_{\infty}\times K_{{\rm fin}}$. Further, we define $K(q)\subset K_{{\rm fin}}$ by requiring $$\Gamma(q) = {\rm GL}_n(\Bbb{Q})\cap ({Z}^+{\rm SL}_n(\Bbb{R})\times K(q)).$$ The local components $K_p(q)$ of $K(q)$ are simply the local principal congruence subgroups  
\begin{equation}\label{cpq}
   K_p(q)=\ker(K_p\to {\rm GL}_n(\Bbb{Z}_p/q\Bbb{Z}_p)).
   \end{equation} Also define $K(q)^{\natural}=D_q^{-1}K(q)D_q$, where we embed $D_q$ diagonally in ${\rm GL}_n(\Bbb{A}_{{\rm fin}})$. This is the adelic version of $\Gamma(q)^{\natural}$. In order to apply strong approximation we enlarge $K(q)$ and $K(q)^{\natural}$ slightly. Put $K_1(q) = T(\widehat{\Bbb{Z}})K(q)$ and $K_1(q)^{\natural}=D_q^{-1}K_1(q)D_q$. For this open compact subgroup the map $\det \colon K_1(q)^{\natural}\to \widehat{\Bbb{Z}}^{\times}$ is surjective. By strong approximation we get
%\begin{equation}
	${\rm GL}_n(\Bbb{A}) = {\rm GL}_n(\Bbb{Q})\cdot ({Z}^+{\rm SL}_n(\Bbb{R})\times K_1(q)^{\natural}).$ % \nonumber
%\end{equation}
We define the adelic space 
\begin{equation}\label{227a} 
X_q^{\Bbb{A}} = {\rm GL}_n(\Bbb{Q})\backslash {\rm GL}_n(\Bbb{A})^1/K_{\infty}K(q)^{\natural}.
\end{equation} The inner product is given by
\begin{equation}
\langle F,G\rangle = \int_{{\rm GL}_n(\Bbb{Q})\backslash {\rm GL}_n(\Bbb{A})^1} F(g)\overline{G(g)}\dd g, \label{eq:innerproduct_adelic}
\end{equation}
where we use the Tamagawa measure, which satisfies %\begin{equation}
${\rm Vol}({\rm GL}_n(\Bbb{Q})\backslash {\rm GL}_n(\Bbb{A})^1,\dd g)=1.$ % \nonumber
%\end{equation}
Note that we use ${\rm GL}_n(\Bbb{A})^1$, which is given by those $g\in {\rm GL}_n(\Bbb{A})$ with $|\det(g)|_{\Bbb{A}}=1$, in order to treat possibly different central characters (all of which will be  trivial on ${Z}^+$) simultaneously.

We can decompose the space $L^2(X_q^{\Bbb{A}})$ similarly as the space $L^2(X_q)$ using the bigger group $K_1(q)^{\natural}$. To do so, given an  $n$-tuple $\boldsymbol{\zeta} = (\zeta_1,\ldots, \zeta_n)$ of Dirichlet characters modulo $q$, we define $\boldsymbol{\zeta}\colon K_1(q)^{\natural}\to S^1$ by
\begin{equation}
	\boldsymbol{\zeta}(D_q^{-1}kD_q) = \prod_{p\mid q} \omega_{\zeta_{1}, p}((k_p)_{11})\cdots \omega_{\zeta_{n},p}((k_p)_{nn}).\nonumber
\end{equation}
Define $L^2(X_q^{\Bbb{A}},\boldsymbol{\zeta}) = \{F \in L^2(X_q^{\Bbb{A}})\colon F(gk) = \boldsymbol{\zeta}(k)F(g)\text{ for all } g\in {\rm GL}_n(\Bbb{A})^1,\, k\in K_1(q)^{\natural}\}.$ Note that we have the decomposition $L^2(X_q^{\Bbb{A}}) = \bigoplus_{\boldsymbol{\zeta}}L^2(X_q^{\Bbb{A}},\boldsymbol{\zeta})$. Furthermore, given $F\in L^2(X_q^{\Bbb{A}},\boldsymbol{\zeta})$ the central character of $F$ is determined by $\boldsymbol{\zeta}$. Indeed we have
\begin{equation*}
	F(zg) = \Big(\prod_{i=1}^n \omega_{\zeta_i}(a)\Big) F(g), \quad g\in {\rm GL}_n(\Bbb{A})^1, \quad z=\text{diag}(a,\ldots,a), \quad \vert a\vert_{\Bbb{A}}=1. \nonumber
\end{equation*}

\subsubsection{Adelization and de-adelization}

The de-adelisation map $\mathcal{D}\colon L^2(X_q^{\Bbb{A}}) \to L^2(X_q)$ is given by $[\mathcal{D}F](x) = F(x)$, where we view $x\in {\rm SL}_n(\Bbb{R})\subset {\rm GL}_n(\Bbb{A})^1$. Sometimes we will write $\mathcal{D}F = \varpi_F$. Suppose $F\in X_q^{\Bbb{A}}(\boldsymbol{\zeta})$, then one checks that $\varpi_F(\gamma x) = [\zeta_1\zeta_n^{-1}](\gamma_{11})\cdots[\zeta_{n-1}\zeta_n^{-1}](\gamma_{n-1,n-1})\cdot  \varpi_F(x)$ for all $\gamma\in \Gamma_1(q)^{\natural}$. Therefore the de-adelisation map restricts to
\begin{equation*}
	\mathcal{D}\colon L^2(X_q^{\Bbb{A}},\boldsymbol{\zeta}) \to L^2(X_q,\boldsymbol{\chi}_{\zeta}), \nonumber
\end{equation*} 
where $\boldsymbol{\chi}_{\zeta} = (\zeta_1\zeta_n^{-1},\ldots,\zeta_{n-1}\zeta_n^{-1} )$. 

We come to the adelisation procedure. Given $\varpi\in L^2(X_q,\boldsymbol{\chi})$ and a Dirichlet character $\xi$ modulo $q$ we define%\marginpar{\tiny{This is all a bit cumbersome at the moment. But it gave me some headache so that I wrote things out in detail. We can shorten this. Maybe start adelic and only-de-adelise? Important is that one needs to work modulo twists, because twists give rise to the same classical object.}}
\begin{equation}
	[\mathcal{A}^{(\xi)}]\varpi(g) = \boldsymbol{\zeta}_{\chi,\xi}(k)\varpi(g_{\infty}) \text{ for }g=\gamma zg_{\infty}k \in {\rm GL}_n(\Bbb{Q}) {Z}^+{\rm SL}_n(\Bbb{R})K_1(q)^{\natural}, \nonumber
\end{equation}
for $\boldsymbol{\zeta}_{\chi,\xi} = (\chi_1\xi,\ldots,\chi_{n-1}\xi,\xi)$. This is a slight modification of the standard adelisation procedure for ${\rm GL}_n$ as described in \cite[Section~13.4]{GH}. We call $F_{\varpi}^{(\xi)} =\mathcal{A}^{(\xi)}\varpi$ the $\xi$-adelic lift of $\varpi$. Note that the choice of the character $\xi$ is the only redundancy in the lifting procedure. One easily verifies that $F_{\varpi}^{(\xi)}$ is well-defined, satisfies $F_{\varpi}^{(\xi)}(gk) = \boldsymbol{\zeta}_{\chi,\xi}(k)F_{\varpi}(g)$ for all $g\in {\rm GL}_n(\Bbb{A})$ and all $k\in K_1(q)^{\natural}$ and that the central character of $F_{\varpi}$ is $\omega_{\chi_1}\cdots \omega_{\chi_{n-1}}\cdot \omega_{\xi}^n$. 

Clearly we have $\mathcal{D}\circ\mathcal{A}^{(\xi)}\varpi = \varpi$. One also sees that the redundancy in the lifting procedure is manifested by twisting with $\omega_{\xi}$ on the adelic side. This is reflected in the following diagram:\\
\adjustbox{scale=1,center}{\begin{tikzcd}
	L^2(X_q,\boldsymbol{\chi}) \arrow[d, "\mathcal{A}^{(1)}"'] \arrow[r, "=" description] & L^2(X_q,\boldsymbol{\chi}) \arrow[d, "\mathcal{A}^{(\xi)}"'] \\
	{L^2(X_q^{\mathbb{A}},\boldsymbol{\zeta}_{\chi,1})} \arrow[r, "\omega_{\xi}\otimes"]  & {L^2(X_q^{\mathbb{A}},\boldsymbol{\zeta}_{\chi,\xi})}       
\end{tikzcd}}
Recall that the twisting operation $\omega_{\xi}\otimes$ is defined by $F\mapsto \omega_{\xi}(\det(\cdot))F$.

If $\varpi\in L^2(X_q,\boldsymbol{\chi})$ is cuspidal, then so is the adelisation $F_{\varpi}^{(\xi)}$. A special case of this computation shows that 
\begin{equation*}
	\mathcal{W}_{F_{\varpi}^{(\xi)}}(y) := \int_{U(\Bbb{Q})\backslash U(\Bbb{A})}F_{\varpi}^{(\xi)}(uy)\boldsymbol{\psi}_{\Bbb{A}}(u)^{-1}\dd u = \mathcal{W}_{\varpi}(y) = A_{\varpi}(1)W_{\mu_{\varpi}}({\tt y}(y)) 
\end{equation*} 
for $y \in \tilde{T}(\Bbb{R}) \subseteq {\rm GL}_n(\Bbb{A})$. 
%where we embed $y=\text{diag}(y_{n-1}\cdots y_1,\ldots, y_2y_1,y_1,1)\in {\rm GL}_n(\Bbb{R})\subset {\rm GL}_n(\Bbb{A})$. 

\subsubsection{Spectral decomposition}
Suppose that $\varpi$ is a cuspidal joint eigenfunction of the ring of invariant differential operators and the unramified Hecke algebra, then $F_{\varpi}^{(\text{triv})}$ generates a cuspidal automorphic representation $\pi_{\varpi}$ and $F_{\varpi}^{(\xi)}$ generates $\omega_{\xi}\otimes\pi_{\varpi}$. Given a cuspidal automorphic representation $\pi$ we abuse notation and write 
\begin{equation}\label{pimidxq}
\pi\mid L^2(X_q)
\end{equation} if there is $\varpi\in L^2(X_q,\boldsymbol{\chi})\subset L^2(X_q)$ with $\omega_{\xi}\otimes \pi_{\varpi}\cong \pi$ for some Dirichlet character $\xi$ modulo $q$. In order to exploit the language of (adelic) automorphic representations, in particular the inherited factorisation properties, we want to construct a spectral decomposition for $L^2(X_q)$ parametrized by automorphic representations $\pi\mid L^2(X_q)$. To do so it is important to keep the twist redundancy in mind. Let us make this precise. Each irreducible unitary automorphic representation $\pi$ of ${\rm GL}_n(\Bbb{A})^1$ has a finite dimensional subspace $\pi^{K(q)^{\natural}}$ of right $K(q)^{\natural}$-invariant elements that are spherical at infinity. If $\pi^{K(q)^{\natural}}\neq \{0\}$, then $\pi\mid L^2(X_q)$. In general we write $\text{\rm ONB}(\pi^{K(q)^{\natural}})$ for an orthonormal basis of this space and we put $$V_{\pi} = \mathcal{D}(\pi^{K(q)^{\natural}})\subset L^2(X_q).$$ This space $V_{\pi}$ is precisely the linear subspace of $L^2(X_q)$ associated to a generic irreducible automorphic representation that is studied in Theorem~\ref{thm15}. (Note that it is finite-dimensional even though $\pi$ itself is infinite-dimensional.) We write the cuspidal part of spectral decomposition of $L^2(X_q)$ as  \begin{equation}\label{sum-prime}
	L^2_{\rm cusp}(X_q) = \sideset{}{'}\bigoplus_{\pi\mid L^2(X_q)} V_{\pi}, 
\end{equation}
where $\sideset{}{'}\bigoplus$ stands for the direct sum of the corresponding automorphic representations modulo character twists. It now becomes clear why we introduced different versions of the adelisation process. Indeed, in order to go from $V_{\pi}$ to $\pi^{K(q)^{\natural}}$ we need to combine all of them as follows. Write $V_{\pi}(\boldsymbol{\chi}) = V_{\pi}\cap L^2(X_q,\boldsymbol{\chi})$ and $\pi^{(K_1(q)^{\natural},\boldsymbol{\zeta})} = \pi\cap L^2(X_q^{\Bbb{A}},\boldsymbol{\zeta})$. If $V_{\pi}(\boldsymbol{\chi})\neq \{ 0\}$, then there is a unique adelisation map $\mathcal{A}\colon V_{\pi} \to \pi^{K(q)^{\natural}}$ given fibrewise by $\mathcal{A}^{(\xi)}\colon V_{\pi}(\boldsymbol{\chi}) \to \pi^{(K_1(q)^{\natural},\boldsymbol{\zeta}_{\chi,\xi})}$ for a suitable Dirichlet characters $\xi$ modulo $q$ depending on $\boldsymbol{\chi}$.

\subsubsection{Inner products}

Finally we need to investigate how the inner products behave under adelisation. This is standard, but since the exact  $q$-dependence is absolutely crucial, we give some details. First note that if $\boldsymbol{\chi}_1\neq \boldsymbol{\chi}_2$, then
\begin{equation*}
	\langle F_{\varpi_1}^{(\xi)},F_{\varpi_2}^{(\xi)}\rangle = 0 =\langle \varpi_1,\varpi_2\rangle \text{ for }\varpi_1\in L^2(X_q,\boldsymbol{\chi}_1) \text{ and }\varpi_2\in L^2(X_q,\boldsymbol{\chi}_2).\nonumber
\end{equation*}
Thus we assume $\boldsymbol{\chi}_1= \boldsymbol{\chi}_2$. Using strong approximation for $K_1(q)^{\natural}$ one obtains
\begin{align}
	\langle F_{\varpi_1}^{(\xi)},F_{\varpi_2}^{(\xi)}\rangle &= \text{Vol}(K_{\infty}K_1(q)^{\natural})\cdot \int_{{\rm GL}_n(\Bbb{Q})\backslash{\rm GL}_n(\Bbb{A})^1/K_{\infty}K_1(q)^{\natural}} F_{\varpi_1}^{(\xi)}(x)\overline{F_{\varpi_2}^{(\xi)}(x)}\dd x \nonumber \\
	&= C_{\rm meas}\text{Vol}(K_{\infty}K_1(q)^{\natural})\cdot \int_{\Gamma_1(q)^{\natural}\backslash\mathcal{H}}\varpi_1(xy)\overline{\varpi_2(xy)}\dd x\dd^{\ast}y \nonumber\\
	&= C_{\rm meas}\text{Vol}(K_{\infty}K_1(q)^{\natural})[\Gamma_1(q)^{\natural}\colon \Gamma(q)^{\natural}]^{-1}\cdot \langle \varpi_1,\varpi_2\rangle\nonumber 
\end{align}
for some proportionality constant $C_{\text{meas}}$.  
We can evaluate both sides with $\varpi_1=\varpi_2\equiv 1$ leading to
\begin{displaymath}
\begin{split}
	C_{\rm meas}\text{Vol}(K_{\infty}K_1(q)^{\natural})[\Gamma_1(q)^{\natural}\colon \Gamma(q)^{\natural}]^{-1} &= \text{Vol}(\Gamma(q)^{\natural}\backslash\mathcal{H})^{-1} \\
	&= [{\rm SL}_n(\Bbb{Z})\colon\Gamma(q)]^{-1}\text{Vol}({\rm SL}_n(\Bbb{Z})\backslash\mathcal{H})^{-1}.\nonumber
	\end{split}
\end{displaymath}
In summary we have
\begin{equation}\label{strongapprox}
	\langle F_{\varpi_1}^{(\xi)},F_{\varpi_2}^{(\xi)}\rangle = n\cdot\Big( \prod_{l=2}^n \frac{\pi^{l/2}}{\Gamma(l/2)\zeta(l)}\Big) \cdot \frac{\langle \varpi_1,\varpi_2\rangle}{\mathcal{V}_q}, \quad \varpi_j\in L^2(X_q,\boldsymbol{\chi}_j). % \text{ and }\varpi_2\in L^2(X_q,\boldsymbol{\chi}_2).
\end{equation}

Let $\pi\mid L^2(X_q)$ be a $\boldsymbol{\psi}_{\Bbb{A}}$-generic (irreducible) cuspidal representation realised in $L^2_{{\rm cusp}}({\rm GL}_n(\Bbb{Q})\backslash {\rm GL}_n(\Bbb{A})^1)$. We realise the contragredient $\tilde{\pi}$ as $\{\overline{F}\colon F\in \pi\}$. The inner product defined in \eqref{eq:innerproduct_adelic} restricts to an invariant inner product on $\pi$ which we denote by $(\cdot,\cdot)_{\pi}$. %\marginpar{\tiny{Didn't we define the exact same thing already on p.9}}
In particular $\Vert F\Vert_{L^2}^2 = (F,F)_{\pi}$. By Flath's factorization theorem we can write $\pi=\otimes_v \pi_v$ and also the Whittaker model ${\rm Wh}(\pi) = \bigotimes_v {\rm Wh}(\pi_v)$ factors. Note that for $\varpi\in V_{\pi}$ we have $\sigma_{\varpi,v}=\sigma_{\pi_v}$ for all unramified places $v$. This always makes sense since $\sigma_{\pi_v}$ is independent under unitary unramified twists. 

We fix any invariant inner product $\langle \cdot,\cdot \rangle_{{\rm Wh}(\pi_v)}$ on the local Whittaker model ${\rm Wh}(\pi_v)$. For $W_v\in {\rm Wh}(\pi_v)$ (whose $L^2$-norm is nonzero) we define 
\begin{equation*}
	I_v(W_v) = \int_{U(\Bbb{Q}_v)}^{\rm st} \frac{\langle \pi_v(u)W_v,W_v\rangle_{{\rm Wh}(\pi_v)}}{\langle W_v,W_v\rangle_{{\rm Wh}(\pi_v)}} \boldsymbol{\psi}_{\Bbb{Q}_v}(u)^{-1}\dd u.
\end{equation*}
This is independent of the choice of the invariant inner product and automatically normalises $W_v$. Suppose now that $F\in \pi$ corresponds to a pure tensor in $\otimes_v\pi_v$. In this case the Whittaker coefficient factors as $\mathcal{W}_F = \bigotimes_v W_{F,v}$. Fix a finite set $S$ of places containing $\infty$ and all places where $F$ is unramified. According to \cite[Lemma~4.4]{LM} and the Rankin-Selberg %\marginpar{\tiny{Did we define $\Delta$ somewhere? \\ Definition added.}} 
method we have
\begin{equation*}
	\frac{\vert\mathcal{W}_F(1)\vert^2}{\langle F,F\rangle} = \lim_{s\to 1}\frac{\Delta^S_{{\rm GL}_n}(s)}{L^S(s,\pi\times\tilde{\pi})} \prod_{v\in S} I_v(W_{F,v}),
\end{equation*}
where $\Delta^S_{{\rm GL}_n}(s) = \prod_{j=1}^n \zeta^S(s+j-1)$. We apply this formula to $F=\mathcal{A}(\varpi)$ with $\varpi\in V_{\pi}$. In this case we can take $S=\{\infty\}\cup \{ p\mid q\}$ and get
\begin{equation}
	\frac{\vert A_{\varpi}(1)\vert^2}{\langle \varpi,\varpi\rangle} = \frac{C(\pi)}{\mathcal{V}_q} \cdot \prod_{p\mid q} I_p(W_{F,p}), \label{eq:single_Fourier_coeff}
\end{equation}
for
\begin{equation}
	C(\pi) = n\cdot\left( \prod_{l=2}^n \frac{\pi^{l/2}}{\Gamma(l/2)\zeta(l)}\right)\cdot \lim_{s\to 1} \frac{\Delta_{{\rm GL}_n}^S(s)}{L^S(s,\pi\times \tilde{\pi})}\cdot \langle W_{\mu_{\varpi}},W_{\mu_{\varpi}}\rangle_{{\rm Wh}(\pi_{\infty})}^{-1}, \nonumber
\end{equation}
where we recall \eqref{strongapprox}. 
Equation \eqref{eq:single_Fourier_coeff} will be crucial in the proof of Theorem~\ref{thm15}. We will also need a lower bound for $C(\pi)$, which can be obtained as follows. First note that 
\begin{equation}
	C(\pi)\gg  q^{-\varepsilon}\cdot \vert L(1,\pi,{\rm Ad})\vert^{-1}\cdot \langle W_{\mu_{\varpi}},W_{\mu_{\varpi}}\rangle_{{\rm Wh}(\pi_{\infty})}^{-1}. \nonumber
\end{equation}
The remaining archimedean inner product in the Whittaker model is $\asymp 1$ by Stade's formula (see the remark after \eqref{stade} below). Finally the adjoint $L$-function can be bounded from above using \cite{Li} and we obtain
\begin{equation}
	C(\pi) \gg (\Vert \mu_{\pi}\Vert q)^{-\varepsilon}. \label{eq:lower_bound_cpi}
\end{equation}
(Note that we only need upper bounds for $C(\pi)$, lower bounds are in general not available unconditionally for $n \geq 3$.)

Our discussion was focused on adelising the cuspidal part of $L^2(X_q)$. Essentially the same discussion applies to the $L^2$-space of $$ X(q)= \Gamma(q)\backslash {\rm SL}_n(\Bbb{R})/{\rm SO}_n(\Bbb{R})$$ as well. We can now give a new description of $m_{\Gamma(q)}(\pi_v)$ in terms of automorphic representations. Indeed, for an irreducible unramified representation $\eta$ of ${\rm GL}_n(\Bbb{Q}_v)$, we have  \begin{equation}\label{dim-m}
	m_{\Gamma(q)}(\eta_v) = \sideset{}{'}\sum_{\substack{\pi=\otimes_{v'}\pi_{v'}\mid L^2(X(q)) \\ \Vert \mu_{\pi_{\infty}}\Vert \leq M, \,  \pi_v\cong \eta_v}} \dim(V_{\pi}).  
	\end{equation}
The non-cuspidal part of $L^2(X(q))$ can be adelised and de-adelised similarly.
% \eqref{dim-m}

\subsection{Spherical functions and spherical transform on ${\rm SL}_n(\Bbb{R})$} The content of this subsection will only be used in Section \ref{sec8}. In this subsection, we write $G = {\rm SL}_n(\Bbb{R})$, $K = {\rm SO}_n(\Bbb{R})$ and $A \subseteq {\rm SL}_n(\Bbb{R})$ for the group of diagonal matrices with positive entries and determinant 1,  and we decompose $G = KAU(\Bbb{R}) = U(\Bbb{R})AK$. Accordingly, we have the Iwasawa projections\footnote{The double use of $A$ will not lead to confusion.} $A, H : G \rightarrow \mathfrak{a}$ such that $g \in K \exp(H(g)) U(\Bbb{R}) = U(\Bbb{R}) \exp(A(g)) K$. We denote by $\dd u, \dd a, \dd k, \dd g$ the usual Haar measures. As usual, we write $\rho = (\frac{n-1}{2}, \frac{n-3}{2}, \ldots, \frac{3-n}{2}, \frac{1-n}{2}) \in \mathfrak{a}^{\ast}$ and for $\mu \in \mathfrak{a}^{\ast}_{\Bbb{C}}$ we define the spherical function (cf.\ \cite[p.\ 418 \& p.\ 435]{He} with $\mu$ in place of $i\lambda$)
\begin{equation}\label{spher-func}
\phi_{\mu}(g) = \int_K e^{(-\rho + \mu)H(gk)} \dd k = \int_K e^{(\rho + \mu)A(kg)} \dd k.
\end{equation} %where $dk$ is the probability Haar measure on $K$. 
For  a smooth, compactly supported, bi-$K$-invariant function $f : G \rightarrow \Bbb{C}$ we define the spherical transform (cf.\ \cite[p.\ 449]{He})
$$\tilde{f}(\mu) = \int_G f(g) \phi_{-\mu}(g) \dd g.$$
It is well-known (\cite[p.\ 450]{He}) that this is the composition of the Abel transfrom 
$$\mathcal{A}_f(a) = e^{\rho(\log a)} \int_{U(\Bbb{R})} f(au) \dd u, \quad a \in A,$$
and the Fourier transform %\marginpar{\tiny{check signs and give a reference}}
$$\tilde{f}(\mu) = \int_A \mathcal{A}_f(a) e^{-\mu \log a} \dd a.$$
 
The spherical transform comes up in the pretrace formula \cite{Se}: for a smooth, compactly supported, bi-$K$-invariant function $f : G \rightarrow \Bbb{C}$ and $z, w\in \Gamma(q)\backslash {\rm SL}_n(\Bbb{R})/{\rm SO}_n(\Bbb{R}) = \Gamma(q) \backslash \mathcal{H}$ the spectral expansion of the automorphic kernel on the left hand side (cf.\ \cite[(2.5)]{Se}) of the following display together with the uniqueness principle (cf.\ \cite[(1.8)]{Se}) is
\begin{equation}\label{pretrace}
   \sum_{\gamma \in \Gamma(q)} f(w^{-1} \gamma z) = \int_{\Gamma(q)} \tilde{f}(\mu_{\varpi}) \varpi(z) \overline{\varpi(w)} \dd \varpi.
\end{equation}
%where we abuse notation and write $\varpi$ for the normalized generator of $V_{\varpi}$ from \eqref{spec}.
%Note that on the right hand side we switched directly to the conjugated group $\Gamma_q$. 

We will apply this with the following test function. 
Let $T \geq 2$ and $f_0 : \Bbb{R}_{>0} \rightarrow [0, 1]$ be a fixed smooth, non-negative function with $f_0(x) = 1$ for $x \in [0, n]$ and $f_0(x) = 0$ for $x > 2n$. Let $\| g \|_F  = \text{tr}(g^{\top}g)^{1/2} \geq \sqrt{n}$ denote  the bi-$K$-invariant Frobenius norm, and define 
\begin{equation}\label{deff}
f(g) = f_0\Big(\frac{\| g \|_F}{T}\Big), \quad g \in G. 
\end{equation}
 Note that $\| g \| \leq \| g \|_F \leq  n \| g \|$ for $g \in G$. For $a = \text{diag} (a_1, \ldots, a_{n-1}, (a_1 \cdots a_{n-1}) ^{-1} ) \in A \subseteq G$ it is easy to see that
$$\partial_{a_1}^{i_1} \cdots\partial_{a_{n-1}}^{i_{n-1}}   \mathcal{A}_f(a) \ll T^{n(n-1)/2} T^{-i_1 - \ldots - i_{n-1}} \delta_{\| a \| \leq 2 n^2 T} $$
for $i_1, \ldots, i_{n-1} \in \Bbb{N}_0$. For $\mu$ satisfying \eqref{unit} and $a \in A$ with $\| a \| \ll T$ we have  $e^{-\mu \log a} \ll T^{n \| \Re \mu\|}$, and so by partial integration we obtain
% If $f_0$ is a smooth function with support in $\| g \| \leq T$ for some $T \geq 1$, then Then $\mathcal{A}_f \ll T^{n(n-1)/2}$ has support in $\| a \| \leq T$, and it follows 
%from the classical Paley-Wiener theorem that
%and so 
\begin{equation}\label{sph}
\tilde{f}(\mu) \ll_{\varepsilon, B} (1 + \| \mu \|)^{-B} T^{n(\frac{n-1}{2} + \| \Re \mu \|)+\varepsilon}
\end{equation}
for any fixed constants $B, \varepsilon > 0$. On the other hand, since $\phi_{\pm \rho}(g)= 1$, we have  \begin{equation}\label{sph-rho}
\tilde{f}(\rho) =  \int_G f(g) \dd g 
\gg  T^{n(n-1)}. 
\end{equation}

\section{Bounds for Whittaker functions}\label{sec3}

In this section we quote the  analysis of Whittaker functions $W_{\mu}$ from \cite[Section 5]{Bl} and make sure that all bounds  are polynomial in $\mu$. The main result is Lemma \ref{whit} below that states in a precise form that a large value of $\sigma_{\pi}(\infty)$ is reflected in the growth of the Whittaker function near the origin. This is the archimedean analogue of \eqref{finite}. We always assume that $\| \Re \mu \| < 1/2$. As in \cite{Bl}, we follow Stade \cite{St, St2} and renormalize the Whittaker function \eqref{wmu} by
 \begin{equation}\label{normalize}
 W_{\mu}^{\ast}(y) =   \pi^{(n-1)n(n+1)/12} y^{-\eta/2} W_{2\mu}\Big((\sqrt{y_1}/\pi, \ldots, \sqrt{y_{n-1}}/\pi)\Big).
 \end{equation}
 The inclusion of the normalization factor \eqref{norm} (which is independent of $y$) is now very useful, because it ensures that the Whittaker function is  on a polynomial scale in $\mu$: by   \cite[Theorem 2]{BHM} we have 
\begin{equation}\label{bhmbound}
W^{\ast}_{\mu}(y ) \ll  \Big((1+\| \mu \|) \Big( \| y \|  + \frac{1}{\min y_j}\Big)\Big)^K  \exp\Big(-   \frac{\| y \| }{K(1 + \| \mu \|)}\Big).
\end{equation}
Since $W^{\ast}_{\mu}(y )$ is holomorphic in $\mu$ in $\| \Re \mu \|< 1/2$, by Cauchy's formula the same bound holds for the derivatives with respect to $\mu$:
\begin{equation}\label{gradient}
\nabla_{\mu} W^{\ast}_{\mu}(y ) \ll  \Big((1+\| \mu \|) \Big( \| y \|  + \frac{1}{\min y_j}\Big)\Big)^K  \exp\Big(-   \frac{\| y \| }{K(1 + \| \mu \|)}\Big).
\end{equation}
Stade's important formula \cite{St2} reads
\begin{equation}\label{stade}
\begin{split}
 & \int_{\Bbb{R}_{\geq 0}^{n-1}} |{W}^{\ast}_{\mu}(y)|^2 \det(\iota(y))^s d^{\ast}y = \frac{C \pi^{\ell(s, \mu)}}{\Gamma(ns)}  \frac{\prod_{j, k = 1}^n \Gamma( s + \mu_j - \mu_k )}{  G(2\mu) G(-2\mu)}
  \end{split}
\end{equation}
with $G$ as in \eqref{norm} and ${\rm d}^{\ast}y$ as in \eqref{measure}, where $C \in \Bbb{R}_{>0}$ and $\ell$ is a linear form with real coefficients, and the left hand side converges absolutely for $\Re s >  \max_{j, k} \Re (\mu_j - \mu_k) $.  The key point of our normalization here is that the exponential behaviour in $\mu$ in the numerator and denominator cancels, and in fixed vertical strips the right hand side is on a polynomial scale in $\mu$.  In particular, we can now justify \eqref{eq:lower_bound_cpi} since (recall \eqref{normalize})
$$ \langle W_{\mu},W_{\mu}\rangle  =  \int_{\Bbb{R}_{\geq 0}^{n-1}} |{W}_{\mu}(y)|^2 \det(\iota(y))  d^{\ast}y \asymp\frac{\prod_{j, k = 1}^n \Gamma(\frac{1}{2} (1 + \mu_j - \mu_k) )}{  G(\mu) G(-\mu)}  \asymp 1.$$

For   $\mu$ satisfying \eqref{unit} and $\| \Re \mu \| < 1/2$ we define
$\mathcal{I}_n(\mu)$ to be the set of $y \in \Bbb{R}_{>0}^{n-1}$ such that
\begin{equation}\label{in}
\begin{split}
&    (2+ \| \mu\|)^{-K}\leq  \min_j y_j  \leq  \max_j y_j \leq  (2 + \|\mu \|)^K, \\
& |W_{\tilde{\mu}}^{\ast}(y)| \geq  (2+ \| \mu\|)^{-K} \quad \text{for} \quad |\tilde{\mu} - \mu| \leq   (2+ \| \mu\|)^{-K}.
\end{split}
\end{equation}
Here we recall our basic convention on $K$ from Section \ref{sec21}. Choosing $s$ (real and) sufficiently large, it is easy to conclude from \eqref{bhmbound}, \eqref{stade} and \eqref{gradient} that
\begin{equation}\label{lowerwhit}
\int_{\mathcal{I}(\mu)} |W_{\mu}^{\ast}(y)|^2 \dd^{\ast}y \gg (2 + \| \mu \|)^{-K}.
\end{equation}
%In particular, there is a dyadic box $ \mathcal{I} = [A, 2A] \subseteq [K^{-1} (1+ \| \mu\|)^{-K},  K (1 + |\mu \|)^2]^{n-1}$ with $A \in \Bbb{R}_{> 0}^{n-1}$ such that 

We now consider the Mellin transform
$$ \widehat{W}^{\ast}_{\mu}(s) = \int_{\Bbb{R}^{n-1}_{>0}} W^{\ast}_{\mu}(y) y^s \frac{\dd y_1}{y_1} \cdots \frac{\dd y_{n-1}}{y_{n-1}}$$
which by \cite{FG} is meromorphic in $s \in \Bbb{C}^{n-1}$. 
For $\Re s_j \geq K$ it follows from \eqref{bhmbound} and from \cite[Theorem 3.1]{St} that
\begin{equation}\label{mellinbound}
 \widehat{W}^{\ast}_{\mu}(s) \ll  \min\Big((1+\| \mu \|)^K,  \exp\big(K \|\mu \| -   \| s \|\big)\Big).
 \end{equation}
By \cite[(2.1)]{FG}, this bound (possibly with a different value of $K$) remains true for all $s$ with distance $\geq \varepsilon$ from poles. It follows from \cite[Section 3]{St} that 
 $$\widehat{W}^{\dagger}_{\mu}(s) := \widehat{W}^{\ast}_{\mu}(s) \prod_{j=1}^n (s_1 + \mu_j)$$
is holomorphic for $\Re s_1 > \sigma_{\pi}(\infty) - 1$ and $\Re s_2, \ldots, \Re s_{n-1} \geq K$, and
\begin{equation}\label{res}
\widehat{W}^{\dagger}_{\mu}(-\mu_j, s_2, \ldots, s_{n-1}) = \widehat{W}^{\ast}_{\mu^{(j)}}(s^{(j)}) \frac{G(\mu^{(j)})}{G(\mu)}\prod_{\substack{1 \leq k \leq n\\ k \not= j}} \Gamma(1 + \mu_k - \mu_j) 
\end{equation}
  where 
  \begin{displaymath}
  \begin{split}
      &\textstyle s^{(j)} =  (s_2, \ldots, s_{n-1}) + (\frac{n-2}{n-1}, \ldots, \frac{1}{n-1})\mu_j , \\
      & \mu^{(j)} = (\mu_1, \ldots, \mu_{j-1} , \mu_{j+1}, \ldots, \mu_{n-1}) +  \frac{\mu_j}{n-1} \cdot \textbf{1}.
      \end{split}
      \end{displaymath}
 For $\beta \in \Bbb{C}$ let $\mathcal{D}_{\beta} = -y\partial_{y} + \beta$. This is a commutative family of differential operators that under Mellin transformation correspond to multiplication with $s + \beta$. Assume (without loss of generality by \eqref{unit} and Weyl group symmetry) that $\Re (-\mu_1) = \sigma_{\pi}(\infty)$,  and   let $$\widehat{{\tt W}}_{\mu}(s) := \frac{\widehat{W}^{\dagger}_{\mu}(s) }{s_1 + \mu_1} = \widehat{W}^{\ast}_{\mu}(s) \prod_{j=2}^n (s_1 + \mu_j). $$
Taking inverse Mellin transforms, we obtain
\begin{equation}\label{ttw}
{\tt W}_{\mu}(y) = \mathcal{D}_{\mu_2} \cdots \mathcal{D}_{\mu_n} W_{\mu}^{\ast}(y)
\end{equation}
where the differential operators are applied to the first variable. By Mellin inversion,  \eqref{res} and \eqref{mellinbound}, a contour shift past the simple pole at $s_1 = - \mu_1$ while keeping $\Re s_2 = \ldots = \Re s_{n-1} = K$ fixed shows
\begin{equation}\label{mellin-asymp}
y_1^j \partial_{y_1}^j{\tt W}_{\mu}(y) =  \mu_1^jy_1^{\mu_1}W^{\ast\ast}_{\mu}(y_2, \ldots, y_{n-1}) + O \big((1+\| \mu \|)^{K} y_1^{\Re\mu_1 + 1/2} (y_2\cdots y_{n-1})^{-K}\big) 
\end{equation}
% \int_{( -\Re \mu_1 - 1/2)} \int_{(\sigma_2)} \cdots \int_{(\sigma_{n-1})}\widehat{{\tt W}}_{\mu}(s) (-s)^jy^{-s}\frac{ds}{(2\pi i)^{n-1}}$$ %O_{y_2, \ldots, y_{n-1}, \mu}\big(y_1^{\Re \mu_1 + 1/2}\big)$$
for   $j \in \{0, 1\}$ where
$$W^{\ast\ast}_{\mu}(y_2, \ldots, y_{n-1}) =  W^{\ast}_{\mu^{(1)}}(y_2, \ldots, y_{n-1}) \prod_{j=2}^{n-1} y_j^{\frac{n-j}{n-1}\mu_1}\Big(\frac{G(\mu^{(j)})}{G(\mu)} \prod_{k=2}^n\Gamma(1 + \mu_k - \mu_1)\Big).$$
%By \eqref{mellinbound}, the multiple integral is $\ll \| \mu \|^{K} y_1^{\Re\mu_1 + 1/2} (y_2\cdots y_{n-1})^{-K}$. 
Note that the last big parenthesis is $\gg ( 1+ \| \mu \|)^{-K}$. 
Let us assume that $(y_2, \ldots, y_{n-1}) \in \mathcal{I}_{n-1}(\mu^{(1)})$ as defined in \eqref{in} and that $y_1 \asymp 1/Z^2$ for some $Z > ( 2+ \| \mu \|)^{K^2}$ where $K$ is so large that \eqref{mellin-asymp} is an asymptotic formula for $j=0$. To transform this asymptotic formula of ${\tt W}_{\mu}$ into a statement for $W^{\ast}_{\mu}$ by means of \eqref{ttw}, we quote \cite[Lemma 5]{Bl}.
\begin{lemma}\label{new} Let $ \alpha \geq 0$, $c_0, c_1, c_2 > 0$,  $\beta \in \Bbb{C}$.  Let $I = [a, b]\subseteq  ( 0, 1)$ be an interval with $(1+c_0)a \leq b \leq 2a$ and $w : I \rightarrow \Bbb{C}$ a smooth function satisfying 
 \begin{equation*}%\label{diff}
    |\mathcal{D}_{\beta} w(y) | \geq c_1 y^{-\alpha}, \quad | \partial_y (\mathcal{D}_{\beta} w)(y) | \leq c_2 \| \mathcal{D}_{\beta} w \| y^{-1}
    \end{equation*} for $y \in I$. Then there exist  constants $c_0',  c_1', c_2'> 0$ depending only on $c_0, c_1, c_2,   \beta$  (but not on $a, b$) and an interval $I' = [a', b']  \subseteq I$ with  $(b' - a') \geq c_0'(b-a)$ such that %such that $a \leq a' < b' \leq b$ and 
    \begin{equation*}%\label{diff1}
    |w(y)| \geq c_1' y^{-\alpha}, \quad |w'(y)| \leq c'_2 \| w|_{I'} \| y^{-1}
    \end{equation*}
     for $y \in I'$. 
  \end{lemma}

An inspection of the proof shows immediately that  $$c_0' \geq \frac{c_0}{K(c_2 + 1 + |\beta|)}, \quad c_1' \geq \frac{c_1}{K(c_2 + 1 + |\beta|)}, \quad c_2' \geq K(c_2 + 1 + |\beta|).$$ Applying this repeatedly with $\beta = \mu_2, \ldots, \mu_n$ and recalling \eqref{lowerwhit}, we see that
$$|W_{\mu}^{\ast}(y)| \gg y_1^{-\sigma_{\pi}(\infty)} (2 + \| \mu \|)^{-K}$$
for $y_1 \in [\gamma_1/Z^2, \gamma_2/Z^2]$ with $1/2 \leq \gamma_1 < \gamma_2 \leq 1$, $\gamma_2 - \gamma_1 \geq (2 + \| \mu \|)^{-K}$ and  $(y_2, \ldots, y_{n-1}) \in \mathcal{I}_{n-1}(\mu^{(1)})$. Note that by definition of $\mathcal{I}_{n-1}(\mu^{(1)})$, the bound in the previous display continues to hold for $\tilde{\mu}$ in a small ball of radius $(2 + \| \mu\|)^{-K}$ about $\mu$ (possibly after changing the constants in \eqref{in}). %We can now complete the argument as in \cite[Section 5]{Bl} to obtain the following lemma.
Thus we can cover the set of $\| \mu \| \leq M$ with $M^K$ balls and for each one pick a (measurable) function $E_j^{\ast\ast} :  [M^{-K}, M^K]^{n-2} \rightarrow [0, 1]$   such that $\sum_j |\langle E^{\ast\ast}_j, W^{\ast\ast}_{\mu} \rangle|^2 \gg M^{-K}$ for $\| \mu \| \leq K $, the inner product being restricted to the last $n-2$ coordinates.  Next define $$E^{\ast}_j(y_1, \ldots, y_n) = \delta_{\gamma_1 \leq y_1 \leq \gamma_2} E^{\ast\ast}_j(y_2, \ldots, y_{n-1}), $$   
so that
  $$\sum_j \Bigl|\int_{\Bbb{R}_{>0}^{n-1}}   E^{\ast}_j(Z^2 y_1, y_2, \ldots, y_{n-1}) \overline{W^{\ast}_{\mu}(y)}   \frac{dy_1}{y_1} \cdots \frac{dy_{n-1}}{y_{n-1}} \Bigr|^2 \gg Z^{4\sigma_{\pi}(\infty)} M^{-K}.$$ Finally changing variables $y_j \leftarrow y_j^{1/2}/ \pi$ as in \eqref{normalize}, the left hand side equals up to numerical constants
  \begin{displaymath}
\begin{split}
 \sum_j  \Bigl|\int_{\Bbb{R}_{>0}^{n-1}} &y^{-\eta} E_j^{\ast}(Z^2\pi^2 y_1^2, \pi^2 y_2^2, \ldots, \pi^2 y_{n-1}^2) \overline{W_{2\mu}(y)}   \frac{dy_1}{y_1} \cdots \frac{dy_{n-1}}{y_{n-1}} \Bigr|^2 \\
 &= Z^{-2\eta_1} \sum_j  \bigl|\langle E^{(Z, 1, \ldots, 1)}_j, W_{2\mu}\rangle \big|^2
\end{split}
\end{displaymath}
upon defining 
$E_j(y_1, \ldots, y_{n-1}) = y^{\eta} E_j^{\ast}(\pi^2 y_1^2, \pi^2 y_2^2, \ldots, \pi^2 y_{n-1}^2)$. Re-normalizing $\mu$ and $\sigma_{\pi}(\infty)$ by division by 2, we obtain the following:

\begin{lemma}\label{whit} Let $M \geq 2$ and $Z >  M^{K^2}$. 
% {\rm (a)} 
 There exists $r \in \Bbb{N}$, $r \leq M^K$ and a   collection of (measurable) functions $E_1, \ldots, E_r :  [M^{-K}, M^K]^{n-1} \rightarrow [0, 1]$  such that  $$\sum_{j=1}^r |\langle E_j^{(Z, 1,\ldots, 1)}, W_{\mu}\rangle|^2 \gg  Z^{2\eta_1 +2 \sigma_{\pi}(\infty)} M^{-K}$$ for $\mu$  satisfying  $\| \mu \| \leq M$, $\| \Re \mu \| < 1/2 $ and $\eta$ as in \eqref{eta}. 
  \end{lemma}

\section{Kloosterman sums}\label{sec4}

We start by compiling the definition and basic properties of Kloosterman sums, following \cite{Fr}  in the case of the group ${\rm SL}_n(\Bbb{Z})$ that generalizes readily to the congruence subgroups $\Gamma(q)^{\natural}$ defined in \eqref{gq}; see also \cite[Section 4]{Bl} and recall \eqref{adv}.   The Bruhat decomposition gives ${\rm GL}_n( \Bbb{Q}) = \bigcup_{w\in W} G_w(\Bbb{Q})$ with $G_w := U  T  w U_w $ as a disjoint union. 
Let $N, M \in \Bbb{Z}^{n-1}$, $c \in \mathbb{N}^{n-1}$, $w \in W$, $v \in V$. Then  provided that
\begin{equation}\label{compat}
  \theta_M\big(c^{\ast} w x w^{-1} (c^{\ast})^{-1}\big) = \theta_N^v(x)
\end{equation}
for all $x \in w^{-1}U(\Bbb{Q}) w \cap U(\Bbb{Q})$, the   Kloosterman sum
\begin{equation}\label{klo}
S^v_{q, w}(M, N, c) = \sum_{x  c^{\ast} w y \in U(\Bbb{Z})\backslash G_w(\Bbb{Q}) \cap \Gamma(q)^{\natural}/U_w(\Bbb{Z}) } \theta_M(x )\theta_N^v(y)
\end{equation}
with $U_w$ as in \eqref{Uw} is well-defined; see \cite[Proposition 1.3]{Fr}. If \eqref{compat} is not met, we define $S^v_{q, w}(M, N, c) = 0$. If $v = \text{id}$, we drop it from the notation. By \cite[p.\ 175]{Fr}, the Kloosterman sum is non-zero only if $w$ is of the form (possibly up to sign in the first row)  \begin{equation}\label{w}
w = \left( \begin{smallmatrix} &  & & I_{d_1}\\  & & I_{d_2} &\\ & \Ddots & &\\ I_{d_r} & && \end{smallmatrix}\right)
\end{equation}
where $d_1 + \ldots  +d_r = n$. The condition  \eqref{compat} is equivalent to
\begin{equation}\label{comp-equiv}
 M_{n-i} \frac{c_{n-i+1} c_{n-i-1}}{c_{n-i}^2} =  \frac{v_{w(i) + 1}}{v_{w(i)}} N_{n-w(i)}
 \end{equation}
for all   $1 \leq i \leq n-1$ satisfying   
\begin{equation}\label{wi}
   w(i) + 1 = w(i+1)
   \end{equation} with the   convention $c_0 = c_n = 1$ and $v = \text{diag}(v_1, \ldots, v_n)$.  If $w$ is of the form \eqref{w}, then the $i$'s satisfying \eqref{wi}  are precisely the $i \not\in \{d_1, d_1 + d_2, \ldots, d_1 + d_2 + \ldots + d_{r-1}\}$.  
   
 The components $c_j$ of $c$ can be recovered from $x c^{\ast} w y$ by taking suitable  $j$-by-$j$ minors containing the last $j$ rows; see   \cite[Proposition 3.1]{Fr}. Since each of them is integral for  $x  c^{\ast} w y \in \Gamma(q)^{\natural}$ we see  that $c \in (\Bbb{Z} \setminus \{0\})^{n-1}$, and modulo $V$ we may assume $c \in \Bbb{N}^{n-1}$. 

	The sum in the definition of the Kloosterman sum $S_{q,w}^{v}(N,M,c)$ runs over the so called Kloosterman set
	\begin{equation}
		X_{q,w}(c) = \{ xc^{\ast}wy\in U(\mathbb{Z})\backslash G_{w}(\mathbb{Q})\cap \Gamma(q)^{\natural}/U_{w}(\mathbb{Z})\}
	\end{equation}
	of modulus $c$ (for the Weyl element $w$). It will be convenient to also define the corresponding $p$-adic versions
	\begin{equation}
		X_{q,w}^{(p)}(c) = \{ xc^{\ast}wy\in U(\mathbb{Z}_p)\backslash G_{w}(\mathbb{Q}_p)\cap K_p(q)^{\natural}/U_{w}(\mathbb{Z}_p)\}.\nonumber
	\end{equation}
   
 By the Chinese remainder theorem,    Kloosterman sums enjoy certain  multiplicativity properties in the moduli, cf.\  \cite[Proposition 2.4]{Fr}. However, in \cite{Fr} only Kloosterman sums for ${\rm SL}_n(\mathbb{Z})$ are considered. While the same factorization properties remain true for a large class of lattices, the situation for the principal congruence subgroup (and its conjugate $\Gamma(q)^{\sharp}$) is slightly more complicated. Therefore we will only give a weak statement in one particular case of interest:
 
 \begin{lemma}
 Suppose that $(rs, q) = 1$ and $u,v\mid q^{\infty}$. Write $d= (u,uv, \ldots, uv^{n-2})$, $c= (r,rs,\ldots,rs^{n-2})$ and denote the coordinate wise product by $dc$. Then 
 \begin{equation}\label{mult}
 	\vert S^v_{q, w_{\ast}}(M, N, dc)\vert \leq  \vert S^v_{1, w_{\ast}}(M, N', c)\vert\cdot \prod_{p\mid q} \sharp X_{p,w_{\ast}}^{(p)}((u_p,u_p\cdot (v_p\overline{s}_p),\ldots,u_p\cdot (v_p\overline{s}_p)^{n-2}) )
 \end{equation}
 with $u_p=(u,p^{\infty})$, $v_p=(v,p^{\infty})$ and a representative $\overline{s}_p$ of $s\cdot \frac{v}{v_p}$ modulo $p$. The new parameter $N'$ can be described explicitly in terms of $u$, $v$ and $N$.
 \end{lemma}
\begin{proof}
The proof relies essentially on finding an explicit bijection between the involved Kloosterman sets. By strong approximation for $U$ and $U_{w_{\ast}}$ the diagonal embedding
\begin{equation}
	\mathfrak{d}\colon X_{q,w_{\ast}}(c') \to \prod_{p\mid qc_1'\cdots c_{n-1}'} X_{q,w_{\ast}}^{(p)}(c') \label{eq:bij}
\end{equation}
is a bijection. Recall that the adelic character $\boldsymbol{\psi}_{\mathbb{A}}$ (resp.\ its restriction to $U_{w_{\ast}}$) is invariant under $U(\mathbb{Q})$ (resp.\ $U_{w_{\ast}}(\mathbb{Q})$)  and factors into local characters $\boldsymbol{\psi}_{\mathbb{A}} =\theta_{(1,\ldots,1)}\cdot \prod_p\boldsymbol{\psi}_{\mathbb{Q}_p}$. Since $X_{q,w^{\ast}}^{(p)}(c') = U(\mathbb{Z}_p)(c')^{\ast}w_{\ast}U_{w_{\ast}}(\mathbb{Z}_p)$, we see that \eqref{eq:bij} gives a factorization
\begin{equation}
	S^v_{q,w_{\ast}}(M,N,c') = \prod_{p\mid q\cdot c_1\cdots c_{n-1}} S^{(p)}_{q,w_{\ast}}(\boldsymbol{\psi}_{\mathbb{Q}_p}^{(M,1)},\boldsymbol{\psi}_{\mathbb{Q}_p}^{(N,v)},c') \nonumber
\end{equation}
into local Kloosterman sums. These are defined by
\begin{equation}
	S^{(p)}_{q,w_{\ast}}(\boldsymbol{\psi}_{\mathbb{Q}_p}^{N,v},\boldsymbol{\psi}_{\mathbb{Q}_p}^{(M,v)},c') = \sum_{xc^{\ast}w_{\ast}y\in X_{q,w_{\ast}}^{(p)}(c)}\boldsymbol{\psi}_{\mathbb{Q}_p}^{(M,1)}(x) \cdot \boldsymbol{\psi}_{\mathbb{Q}_p}^{(N,v)}(y),\nonumber
\end{equation}
with $\boldsymbol{\psi}_{\mathbb{Q}_p}^{(N,v)}(y) = \boldsymbol{\psi}_{\mathbb{Q}_p}^{(N,1)}(v^{-1}yv)$ and $\boldsymbol{\psi}_{\mathbb{Q}_p}^{(N,1)}(y) = \psi_p(-N_{n-1}y_{1,2}-\ldots-N_1y_{n-1,n})$. We estimate
\begin{equation}
	\vert S^v_{q,w_{\ast}}(M,N,cd)\vert \leq \vert \prod_{p\mid rs}S^{(p)}_{q,w_{\ast}}(\boldsymbol{\psi}_{\mathbb{Q}_p}^{(M,1)},\boldsymbol{\psi}_{\mathbb{Q}_p}^{(N,v)},cd)\vert \cdot \prod_{p\mid q} \sharp X_{q,w_{\ast}}(cd).\nonumber
\end{equation}

For $p\nmid q$ we have $K_p(q)^{\ast} = {\rm GL}_n(\mathbb{Z}_p)$, so that in particular
\begin{equation}
	X_{q,w_{\ast}}^{(p)}(dc) \cong X_{1,w_{\ast}}^{(p)}(c) \nonumber
\end{equation}
via the usual bijection. This allows us to rewrite
\begin{equation}
	\prod_{p\mid rs}S^{(p)}_{q,w_{\ast}}(\boldsymbol{\psi}_{\mathbb{Q}_p}^{(M,1)},\boldsymbol{\psi}_{\mathbb{Q}_p}^{(N,v)},cd) = \prod_{p\mid rs}S^{(p)}_{1,w_{\ast}}(\boldsymbol{\psi}_{\mathbb{Q}_p}^{(M,1)},\boldsymbol{\psi}_{\mathbb{Q}_p}^{(N',v)},c) = S_{1,w_{\ast}}^{v}(M,N',c),\nonumber
\end{equation}
for a suitable $N'\in \mathbb{N}$ depending on $N$ and $d$.

Suppose now that $p\mid q$. We claim that
\begin{equation}
	X_{q,w_{\ast}}^{(p)}(dc) \cong X_{p,w_{\ast}}^{(p)}((u_p,u_p\cdot(v_p\overline{s}_p),\ldots,u_p\cdot(v_p\overline{s}_p)^{n-2})).\nonumber 
\end{equation}	
This can be seen by defining the two torus elements
\begin{equation}
	t_1=\text{diag}(1,\tilde{s},\ldots,\tilde{s},1)\in T(1+q\mathbb{Z}_p) \text{ and }t_2=\text{diag}(\frac{u_p}{ur},1,\ldots,1)\in T(\mathbb{Z}_p) \nonumber
\end{equation}	
where $\tilde{s}s\frac{v}{v_p}=\overline{s}$ and using the bijection
\begin{equation}
	K_p(q)^{\natural}\ni k\mapsto t_1t_2kt_2^{-1} \in K_p(q)^{\natural}.\nonumber
\end{equation}
This establishes the claim and with it the desired estimate.
\end{proof} 

\begin{remark}
The proof shows explicitly that $p$-adically the main obstruction to the usual factorization properties is that $T(\mathbb{Z}_p)\cdot K_p(q)^{\sharp} \neq K_p(q)^{\sharp}$ for $p\mid q$.
\end{remark}

 By \cite[Theorem 0.3(i)]{DR} we have the trivial bound
\begin{equation}\label{triv}
|S^v_{1, w}(M, N, c) | \leq   \sharp X_{1,w}(c)   \ll (c_1 \cdot \ldots \cdot c_{n-1})^{1+\varepsilon}. 
\end{equation}
If $w = \text{id}$, then $U_w(\Bbb{Q})/U(\Bbb{Z}) = \{ I_n\}$ and $xc^{\ast} \in\Gamma(q)^{\natural}$ with $x\in U(\Bbb{Z}) \backslash U(\Bbb{Q})$  implies $c_1 = \ldots = c_{n-1} = 1$, and so 
\begin{equation}\label{trivialweyl}
 S_{q, \text{id}}^v(M, N, c) = \delta_{c = (1, \ldots, 1)} \delta_{N = M} \mathcal{N}_q
\end{equation}
using the notation \eqref{vqnq}.

For the following lemma  let $w_1 = \left(\begin{smallmatrix} & I_{n-1}\\ 1 & \end{smallmatrix}\right)$ (possibly up to sign in the first row) be the ``Voronoi element'' and $w_{\ast}$ as in \eqref{wast}.

\begin{lemma}\label{lem41} Let $x \in U(\Bbb{Q})$, $y \in U_w(\Bbb{Q})$, $c = (c_1, \ldots, c_{n-1})\in \Bbb{N}^{n-1}$.  \\
{\rm (a)} Let $w \in W \setminus \{\text{{\rm id}}, w_{\ast}, w_1\}$ be a Weyl element of the form \eqref{w}. If $x c^{\ast} w y \in \Gamma(q)^{\natural}$, then $q^{n+1} \mid c_j$ for some $j$.\\
{\rm (b)} If $x c^{\ast} w_{1} y \in \Gamma(q)^{\natural}$ and $c = (m\gamma^{n-1}, \gamma^{n-2}, \ldots, \gamma^2, \gamma)$ for some $\gamma \in \Bbb{N}$ and some $m \in \Bbb{N}$ with $(m, q) = 1$, then $q^{n+1} \mid \gamma^{n-1}$. \\
{\rm (c)}  If $x c^{\ast} w_{\ast} y \in \Gamma(q)^{\natural}$, then $q^n 
\mid c_j$ for all $j$. 
\end{lemma} 

\textbf{Remarks:} 1) For the congruence subgroup $\Gamma_0(q) \subseteq {\rm SL}_2(\Bbb{Z})$, the moduli $c$ in the Kuznetsov formula need to be divisible by $q$. More generally, for $\Gamma_0(q) \subseteq {\rm SL}_n(\Bbb{Z})$, the moduli $c_1, \ldots, c_{n-d_1}$ in the notation of \eqref{w} need to be divisible by $q$, cf.\  \cite[(4.3)]{Bl}. This lemma generalizes such statements for the group $\Gamma(q)^{\natural}$ and shows strong divisibility statements for the moduli. \\
2) The proofs   show the slightly stronger statements $q^{\min(n+2, 2n-2)} \mid c_j$ in part a) and $$(n-1) v_p(\gamma) \geq \Big(n+1 + \frac{2}{n-2}\Big)v_p(q)$$ 
in part b).\\

 \textbf{Proof.} (a) In the notation of \eqref{w} let $d' = d_r$, $d = d_{r-1}$ with $d' < n$.   We have
$$xc^{\ast} wy = \left(\begin{matrix}{\huge \text{$\ast$}}  &\cdots && {\huge \text{$\ast$}} \\ c_{d'}/c_{d'-1} &{\large \text{$\ast$}} &{\large \text{$\ast$}}& \vdots \\ & \ddots &{\large \text{$\ast$}}& \\ && c_1 &{\huge \text{$\ast$}}\end{matrix}\right) \in \Gamma(q)^{\natural}. $$
We see that $q^{n-d'+1} $ divides $c_1, c_2/c_1, \ldots c_{d'}/c_{d' - 1}$, hence  $q^{d'(n-d'+1)} \mid c_{d'}$ and $d'(n - d' + 1) \geq 2(n-1) \geq n+1$ for $1 < d' < n$ and $n \geq 3$. So it remains to consider the case $d' = 1$. Since $w \not= w_{\ast}, w_1$, we have $d +1 \leq n-2$. We have 
\begin{equation}\label{cw}
c^{\ast} w = \left(\begin{matrix}  &&  && {\huge \text{$\ast$}}  \\ &c_{d+1}/c_{d} &&& \\& & \ddots &&   \\ &&& c_2/c_1 &\\ c_1 &&&&   \end{matrix}\right)
\end{equation}
with ${\huge \text{$\ast$}}\in \text{Mat}(n-d-1, \Bbb{Q})$. 
We conclude that $xc^{\ast} wy\in \Gamma(q)^{\natural}$ implies $q^n \mid c_1$, and  taking the 2-by-2 determinant of the entries in rows $n-1$ and $n$ and columns $1$, $d+1$ we conclude $q^{n + (n-d-1)} \mid   c_2$ with $2n - d -1\geq n+2$ for $d+1 \leq n-2$.
 
 (b) We have 
 %  The case $d+1 = n-1$ implies $w = w_{\ast}$ and is excluded, so it remains to consider the case $d+1 = n$, in which case $w =  \left(\begin{matrix} & I_{n-1}\\ 1 & \end{matrix}\right)$ is the ``Voronoi element''. 
%For the Voronoi element $w = \left(\begin{matrix} & I_{n-1}\\ 1 & \end{matrix}\right)$, the moduli are $c_j = c^{n-j}$ with $c = c_{n-1}$. So we get
$$c^{\ast} w_1 = \left(\begin{smallmatrix} & \pm 1/\gamma & & \\ & & \ddots &  \\ & & & 1/\gamma\\  &&& & 1/(m\gamma) \\ m\gamma^{n-1} &&&\end{smallmatrix}\right).$$
Taking the $2$-by-$2$ determinant with rows $n-2, n$ and columns $1, n-1$ we conclude $q^n \mid \gamma^{n-2}$. So for each prime  $p \mid q$ we have $n v_p(q) \leq (n-2) v_p(\gamma)$; which implies $$(n-1)v_p(\gamma) \geq \frac{n(n-1)}{(n-2)} v_p(q) > (n+1) v_p(q).$$

(c) This follows in the same way from \eqref{cw} where now $\ast$ in the upper right corner is a 1-by-1 block.\\ %one of the middle 2-by-23 determinants, we get 
%$q^n \mid c^{n-2}$ and hence $q^{n(n-1)/(n-2)} \mid c^{n-1} = c_1$ where $n(n-1)/(n-2) = n+1 + 2/(n-2) > n+1$. 

The following technical lemma bounds the Kloosterman set for $w = w_{\ast}$ and certain moduli.
 
\begin{lemma}\label{technical} Let  $p$ be a prime,   $\alpha, \beta \in \Bbb{N}_0$, $s_p\in \mathbb{Z}_p^{\times}$ and $$c = (p^{n+\alpha}, \overline{s}_pp^{n+\alpha+\beta}, \ldots, \overline{s}_p^{n-2}p^{n + \alpha + \beta(n-2)}) \in \Bbb{Z}_p^{n-1}.$$ Then the number $C_{\alpha, \beta}$ of $(x, y) \in  U(\Bbb{Z}_p) \backslash U(\Bbb{Q}_p) \times U_{w_{\ast}}(\Bbb{Q}_p)/U_{w_{\ast}}(\Bbb{Z}_p)$ such that $xc^{\ast} w_{\ast} y \in K_p(p)^{\natural}$ is at most
$$2p^{\frac{1}{6}(n^3 + 3n^2 - 10n + 6) + 2\alpha (n-1) + (n-1)(n-2)\beta}.$$
\end{lemma}

\textbf{Remarks:} We expect that the true order of magnitude of the previous count is $p^{\frac{1}{6}(n^3 + 3n^2 - 10n + 6)+ \alpha (n-1) + \frac{1}{2}(n-1)(n-2)\beta}$, so the constant part of the exponent is sharp, but the $\alpha$ and $\beta$-dependence loses a factor of 2. This is just as much as we can afford without penalty in the proof of  Proposition \ref{density-kuz} below. The significance of the cubic polynomial $\frac{1}{6}(n^3 + 3n^2 - 10n + 6)$ is that
$$q^{\frac{1}{6}(n^3 + 3n^2 - 10n + 6)} = \mathcal{N}_q q^{(n-1)^2}.$$
A generalization to not necessarily squarefree $q$ in Proposition \ref{density-kuz}  would require  more general exponents of $p$ in the entries of $c$.\\

\textbf{Proof.} Recall that $U(\Bbb{Z}_p)$ acts on $K_p(p)^{\natural}$ both from the left and the right. As a system of representatives for  $U(\Bbb{Z}_p) \backslash U(\Bbb{Q}_p)$ we choose the set of matrices whose entries above the diagonal are in $\mathbb{Q}_p/\mathbb{Z}_p$. Similarly we restrict the relevant entries in $U_{w_{\ast}}(\Bbb{Q}_p)/U_{w_{\ast}}(\Bbb{Z}_p)$. Note that 
$$c^{\ast} = \text{{\rm diag}}(\overline{s}_p^{-n+2}p^{-n-\alpha- 
\beta(n-2)}, \underbrace{\overline{s}_pp^{\beta}, \ldots, \overline{s}_pp^{\beta}}_{n-2}, p^{n+\alpha}).$$
We claim that $x c^{\ast}w_{\ast} y \in K_p(p)^{\natural}$ implies the following:
\begin{equation}\label{a1x}
x_{1j} \in \frac{1}{p^{j-1+\alpha+\beta}}\Bbb{Z}_p, \quad j = 2, \ldots, n-1,
\end{equation}
\begin{equation}\label{a1y}
y_{1j}  \in \frac{1}{p^{j-1+\alpha}}\Bbb{Z}_p, \quad j = 2, \ldots, n-1,
\end{equation}
\begin{equation}\label{b1x}
 x_{in} \in \frac{1}{p^{n - i +\alpha}}\Bbb{Z}_p,\quad i = 2, \ldots, n-1, 
\end{equation}
\begin{equation}\label{b1y}
 y_{in} \in \frac{1}{ p^{(1+\alpha+\beta)(n - i) }}\Bbb{Z}_p,\quad i = 2, \ldots, n-1,
 \end{equation}
\begin{equation}\label{c1}
   x_{1n}, y_{1n} \in \frac{1}{p^{n+\alpha}} + \frac{1}{p^{n-1+\alpha}}\Bbb{Z}_p,
   \end{equation}
\begin{equation}\label{d}
 x_{ij} \in \frac{1}{p^{j-i-1+\alpha+\beta}} \Bbb{Z}_p, \quad 2 \leq i < j \leq n-1.
 \end{equation}
To show \eqref{a1x} -- \eqref{d},  we compute suitable entries of $x c^{\ast} w y = (A_{ij})$, say. First we see that
$$A_{i1} = p^{n+\alpha} x_{in} \quad (1 \leq i \leq n-1), \quad  \quad A_{nj} = p^{n+\alpha}  y_{1j}\quad   (2 \leq j \leq n)$$
in the first column and last row. Since $A_{i1} \in p^{i}\Bbb{Z}_p$ for $i \geq 2$ and  $A_{nj} \in p^{n-j+1}\Bbb{Z}_p$ for $j \leq n-1$ and $A_{11}, A_{nn}   \in 1 + p\Bbb{Z}_p$ by \eqref{gq}, this implies \eqref{a1y}, \eqref{b1x}, \eqref{c1}.  Next we have
$$\frac{1}{p^{j-i-1}}\Bbb{Z}_p \ni A_{ij} = \overline{s}_pp^{\beta}x_{ij} + p^{n +\alpha} x_{in} y_{1j}, \quad 1 \leq i < j \leq n-1.$$
Since we know already $x_{in} y_{1j} \in p^{n-i+j-1-2\alpha}\Bbb{Z}_p$ for $2 \leq i < j \leq n-1$ by \eqref{a1y}, \eqref{b1x}, we conclude \eqref{d}, and for $i=1$ a similar argument using \eqref{c1} shows \eqref{a1x}. To establish \eqref{b1y}, we use the last column:
$$\frac{1}{p^{n-i-1}} \Bbb{Z}_p \ni A_{in} = p^{n +\alpha}x_{in}y_{1n} + \overline{s}_pp^{\beta}y_{in} + \overline{s}_pp^{\beta} \sum_{j=i+1}^{n-1}  x_{ij}y_{jn}, \quad 2 \leq i \leq n-1.$$
By induction on $i = n-1, n-2, \ldots, 2$ we have 
$$p^{n+\alpha}x_{in}y_{1n} + \overline{s}_pp^{\beta} \sum_{j=i+1}^{n-1}  x_{ij}y_{jn} \in \frac{1}{p^{n-i+\alpha}}\Bbb{Z}_p + \frac{1}{p^{(1+\alpha+\beta)(n-i-1) + \alpha} }\Bbb{Z}_p \subseteq \frac{p^{\beta}}{p^{(1+\alpha+\beta)(n-i)}}\Bbb{Z}_p$$
and \eqref{b1y} follows for $i-1$. We have now established \eqref{a1x} -- \eqref{d}.

By choosing a suitable fundamental domain for $\mathbb{Q}_p/\mathbb{Z}_p$ in which the entries $x_{ij}, y_{ij}$ must lay, we re-write these conditions as
\begin{equation}\label{xiyi}
\begin{split}
x_{1j} &=\frac{x_{1j}'}{p^{j-1+\alpha+\beta}}, \quad j = 2, \ldots, n-1, \quad x'_{1j} \in \{0, 1, \ldots, p^{j-1+\alpha+\beta}-1\},\\
y_{1j} &= \frac{y_{1j}'}{p^{j-1+\alpha}}, \quad j = 2, \ldots, n-1, \quad y_{1j}' \in \{0, 1, \ldots, p^{j-1+\alpha}-1\},\\
  x_{in} &= \frac{x_{in}'}{p^{n - i+\alpha }} ,\quad i = 2, \ldots, n-1, \quad x_{in}' \in \{0, 1, \ldots, p^{n-i+\alpha}-1\},\\
  y_{in} & = \frac{y_{in}'}{p^{(1+\alpha+\beta)(n - i) }},\quad  i = 2, \ldots, n-1, \quad y_{in}' \in \{0, 1, \ldots, p^{(1+\alpha+\beta)(n-i)}-1\},\\
    x_{1n}  &= \frac{1 + px'_{1n}}{p^{n+\alpha}} , \quad   y_{1n} = \frac{1 + py'_{1n}}{p^{n+\alpha}}, \quad x_{1n}', y_{1n}' \in \{0, 1, \ldots, p^{n-1+\alpha}-1\},\\
    x_{ij} & = \frac{x_{ij}'}{p^{j-i-1+\alpha+\beta}}, \quad 2 \leq i < j \leq n-1, \quad x_{ij}' \in \{0, 1, \ldots, p^{j-i-1+\alpha+\beta}-1\}.
 \end{split}
 \end{equation}
 The number of choices of such $x, y$ is
 \begin{equation}\label{choices}
 \begin{split}
&  p^{\frac{1}{2}(n-1)(n-2) + (\alpha+\beta)(n-2)}   \cdot  p^{\frac{1}{2}(n-1)(n-2)+ \alpha(n-2)} \cdot p^{\frac{1}{2}(n-1)(n-2) + \alpha(n-2) }    \\
&  \quad\quad \cdot p^{\frac{1}{2}(n-1)(n-2)(1+\alpha+\beta)} \cdot p^{2(n-1+\alpha)} \cdot p^{\frac{1}{6}(n-2)(n-3)(n-4) + \frac{1}{2}(n-2)(n-3)(\alpha+\beta)}\\
= & p^{\frac{1}{6}(n^3 + 3n^2 + 2n - 12) + \alpha n(n-1) + \beta (n-1)(n-2)}.
 \end{split}
 \end{equation}
  In the new coordinates we compute again
 \begin{displaymath}
 \begin{split}
\frac{1}{p^{j-2}}\Bbb{Z}_p \ni A_{1j} & = \frac{\overline{s}_px'_{1j} + y'_{1j} (1 + px_{1n}')}{p^{j-1+\alpha}}, \quad 2 \leq j \leq n-1,\\
\frac{1}{p^{j-i-1}}\Bbb{Z}_p \ni   A_{ij} & =  \frac{\overline{s}_px'_{ij} + x'_{in}y'_{1j}}{p^{j-i-1+\alpha}} , \quad 2 \leq i < j \leq n-1,\\
\frac{1}{p^{n-i-1}}\Bbb{Z}_p \ni   A_{in} & =  \frac{x'_{in}(1 + p y'_{1n})}{p^{n-i+\alpha}} + \frac{\overline{s}_py'_{in}}{ p^{(1+\alpha)(n-i) + \beta (n-i-1)}} + \sum_{j=i+1}^{n-1}\frac{ \overline{s}_px'_{ij}y'_{jn}}{p^{n-i-1 + \alpha(n-j+1) + \beta(n-j)}}  , \\
&\quad\quad\quad\quad\quad 2 \leq i \leq n-1,\\
%\end{split}
%\end{displaymath}
%\begin{displaymath}
 %\begin{split}
\frac{1}{p^{n-2}}\Bbb{Z}_p \ni   A_{1n} &=  - \frac{\overline{s}_p^{-n+2}}{p^{n+\alpha + \beta(n-2)}} + \frac{(1 + px_{1n}' )(1 + py_{1n}') }{p^{n+\alpha}}  + \sum_{j=2}^{n-1}\frac{\overline{s}_p x'_{1j}y'_{jn}}{p^{n-1 + \alpha(n-j+1) + \beta(n-j)}},
\end{split}
\end{displaymath}
which implies the congruences
\begin{equation}\label{congr}
 \begin{split}
  &\overline{s}_px'_{1j} + y'_{1j} (1 + px_{1n}') \equiv 0 \, (\text{mod } p^{1+\alpha}), \quad 2 \leq j \leq n-1,\\
& \overline{s}_px'_{ij} + x'_{in}y'_{1j} \equiv 0 \, (\text{mod } p^{\alpha}) , \quad 2 \leq i < j \leq n-1,\\
& \overline{s}_py_{in}' + x'_{in}(1 + p y'_{1n})p^{(\alpha+\beta)(n-i-1)} + p \sum_{j=i+1}^{n-1} p^{(\alpha+\beta)(j-i-1)} \overline{s}_px'_{ij}y'_{jn}\\
&\quad\quad\quad\quad\quad \equiv 0 \, (\text{mod } p^{1 + \alpha(n-i) + \beta(n-i-1)}), \quad
  2 \leq i \leq n-1,\\
%& A_{1n} =  -1 + s^{n-2} (1 + qx_{1n}' )(1 + qy_{1n}')   + \sum_{j=i+1}^{n-1}\frac{ x'_{ij}y'_{jn}}{q^{n-i-1}r^{n-j+1}s^{n-j}},
\end{split}
\end{equation}
as well as
\begin{equation}\label{corner}
 - \frac{\overline{s}_p^{-n+2}}{p^{2+\alpha + \beta(n-2)}} + \frac{(1 + px_{1n}' )(1 + py_{1n}') }{p^{2+\alpha}}  +\frac{1}{p}\sum_{j=2}^{n-1}\frac{ \overline{s}_px'_{1j}y'_{jn}}{ p^{\alpha(n-j+1) + \beta(n-j)}} \in \Bbb{Z}_p.
 \end{equation}
We are now ready to count the number of $(x, y)$ satisfying \eqref{xiyi}, \eqref{congr} and \eqref{corner}. 

Our first count drops the condition \eqref{corner}. Fixing $y_{1j}'$, $2 \leq j \leq n$, and $x_{in}'$, $1 \leq i \leq n-1$, the first set of congruences in \eqref{congr} fixes $x_{1j}'$ modulo $p^{1+\alpha}$ for $2 \leq j \leq n-1$, the second set fixes $x_{ij}'$ modulo $p^{\alpha}$ for $2 \leq i < j \leq n-1$, and once these choices are made the last set of congruences fixes inductively $y_{in}'$ modulo $p^{1 + \alpha(n-i) + \beta(n-i-1)} $ for $i = n-1, n-2, \ldots, 2$. Recalling \eqref{choices},  we obtain
\begin{displaymath}
\begin{split}
 C_{\alpha, \beta} & \leq \frac{p^{\frac{1}{6}(n^3 + 3n^2 + 2n - 12) + \alpha n(n-1) + \beta (n-1)(n-2)}}{p^{(1+\alpha)(n-2)} \cdot p^{\frac{1}{2}(n-2)(n-3)\alpha} \cdot p^{n-2 + \frac{1}{2}(n-1)(n-2)\alpha + \frac{1}{2}(n-2)(n-3)\beta} }\\
 & = p^{\frac{1}{6}(n^3 + 3n^2 - 10n + 12) + 2(n-1)\alpha +  \frac{1}{2}(n+1)(n-2)\beta}.
 \end{split}
 \end{displaymath}
This is sufficient for the claim of the lemma provided that $\beta \geq 1$ and $n \geq 4$. 

We now consider the case $\beta = 0$, in which case  \eqref{corner} simplifies to 
\begin{equation}\label{corner1}
   p^{3\alpha} (x_{1n}' + y_{1n}' + p x_{1n}'y_{1n}'  )  + \sum_{j=2}^{n-1}  p^{\alpha(j-2)} \overline{s}_px'_{1j}y'_{jn}  \equiv 0 \, (\text{mod } p^{1+(n-1)\alpha }),
 \end{equation}
 and our task is to gain an additional factor $p$ from this congruence. 
We start with the subcase $\alpha = 0$, in which \eqref{corner1} becomes $$x_{1n}' + y_{1n}' + \sum_{j=2}^{n-1} \overline{s}_px_{1j}'y_{jn}' \equiv 0 \, (\text{mod } p),$$
while  \eqref{congr} becomes
\begin{equation*} 
 \overline{s}_px'_{1j} + y'_{1j}\equiv  \overline{s}_py_{in}' + x'_{in} \equiv 0 \, (\text{mod } p), \quad 2 \leq i, j \leq n-1.
 \end{equation*}
By the same argument as above we have
$$ C_{0, 0} \leq  \frac{p^{\frac{1}{6}(n^3 + 3n^2 + 2n - 12)}  }{p^{n-2} \cdot p^{n-2} \cdot p } = p^{\frac{1}{6}(n^3 + 3n^2 - 10n + 6)};$$
which establishes the bound of the lemma for $\alpha = \beta = 0$.

If $\beta=0$, $\alpha \geq 1$, we first count the contribution of those pairs $(x, y)$ with $x_{12} \not\equiv 0$ (mod $p$). As before, we use \eqref{congr} to fix $x_{1j}'$ (mod $p^{1+\alpha}$) for $2 \leq j \leq n-1$, then $x_{ij}'$ (mod $p^{\alpha}$) for $2 \leq i < j \leq n-1$ and then inductively $y_{in}'$ modulo $p^{1 + \alpha(n-i)} $ for $i = n-1, n-2, \ldots, 3$. But in contrast to the general case, we now use \eqref{corner1} to fix $y_{2n}'$ modulo $p^{1+(n-1)\alpha}$ (which gains a factor $p^{\alpha}$). On the other hand, if $x_{12}' \equiv 0$ (mod $p$), then by the first congruence in \eqref{congr} we must have $y_{12}' \equiv 0$ (mod $p$), so up front we save a factor $p$, since $y_{12}'$ is determined modulo $p$. We conclude that
\begin{displaymath}
\begin{split}
 C_{\alpha, 0} & \leq \frac{p^{\frac{1}{6}(n^3 + 3n^2 + 2n - 12) + \alpha n(n-1)}}{p^{(1+\alpha)(n-2)} \cdot p^{\frac{1}{2}(n-2)(n-3)\alpha} \cdot p^{n-2 + \frac{1}{2}(n-1)(n-2)\alpha  } } \Big(\frac{1}{p^{\alpha}} + \frac{1}{p}\Big)\\
 & \leq  2 p^{\frac{1}{6}(n^3 + 3n^2 - 10n + 6) + 2(n-1)\alpha }
 \end{split}
 \end{displaymath}
for $\alpha \geq 1$. 

The last remaining case is $n=3$ and $\beta \geq 1$. In this case \eqref{corner} becomes
$$ p^{\alpha}(p^{\beta} - \overline{s}_p) + p^{1+\alpha + \beta}(x_{13}' + y_{13}' + px_{13}'y_{13}') + p\overline{s}_px_{12}'y_{23}' \equiv 0 \, (\text{mod } p^{2+2\alpha + \beta}).$$
The case $\alpha = 0$ is obviously impossible, and for $\alpha \geq 1$ the previous display implies
$$x_{12}'y_{23}'  \equiv \begin{cases} 0 \, (\text{mod } p), & \alpha \geq 2,\\
1 \, (\text{mod } p), & \alpha = 1.\end{cases} $$
As before this saves a factor $2/p$ for the number of choices for $(x_{12}', y_{23}')$, and after these choices are made we use \eqref{congr} to determine $y_{12}', x_{23}'$ modulo $p^{1+\alpha}$. This suffices to establish the lemma also for $n=3$ and completes the proof.

\begin{theorem}\label{thm32} Let $N = (\ast, 1, \ldots, 1, \ast), M = (\ast, 1,\ldots, 1, \ast) \in \Bbb{N}^{n-1}$, $v\in V$ and $q \in \Bbb{N}$. Then $S^v_{q, w_{\ast}}(M, N, c)$ vanishes unless $$c = (q^nr, q^n rs, q^nrs^2, \ldots,  q^nrs^{n-2}) \quad \text{or} \quad  (q^nrs^{n-2}, q^n rs^{n-1},  \ldots,  q^nrs, q^nr)$$ for some $r, s \in \Bbb{N}$. We write $r = r_1r_2$  and $s = s_1s_2$ with $r_1s_1 \mid q^{\infty}$, $(r_2s_2, q) = 1$. If $q$ is squarefree, then with this notation we have 
$$S^v_{q, w_{\ast}}(M, N, c) \ll  \mathcal{N}_q \frac{(c_1\cdot \ldots \cdot c_{n-1})^{1+\varepsilon} }{q^{n-1}}   \Big(\frac{c_1}{q^n} \cdots \frac{c_{n-1}}{q^n}, q^{\infty}\Big).$$% \mathcal{N}_q q^{(n-1)^2+\varepsilon} (r_1^2r_2)^{n-1} (s_1^2s_2)^{\frac{1}{2}(n-1)(n-2)}.$$
\end{theorem}

\textbf{Proof.} That all entries of $c$ must be divisible by $q^n$ follows from Lemma  
\ref{lem41}(c), and the special shape of $c$ follows from \eqref{comp-equiv} since we must have $c_i^2 = c_{i-1}c_{i+1}$ for $2 \leq i \leq n-2$. The standard involution $g \mapsto w_l g^{-\top} w_l$ leaves $w_{\ast}$, $\Gamma(q)^{\natural}$,  the shape of $N$ and $M$, $U(\Bbb{Z})$ and $U_{w_{\ast}}(\Bbb{Z})$ invariant and switches  the diagonal matrices $ (q^nr, q^n rs, q^nrs^2, \ldots,  q^nrs^{n-2})^{\ast}$ and $(q^nrs^{n-2}, q^n rs^{n-1},  \ldots,  q^nrs, q^nr)^{\ast}$. Hence it suffices to establish the bound for one of the two forms of $c$, say the first one (cf.\ \cite[Theorem 3.2]{Ste}). By \eqref{mult}, we first split off the portion containing $r_2$ and $s_2$ for which we use the trivial bound \eqref{triv}. For the remaining portions we may apply Lemma \ref{technical} with $q = p$ (since $q$ is squarefree),  $r = p^{\alpha}$, $s = p^{\beta}$ to obtain
$$S^v_{q, w_{\ast}}(M, N, c) \ll    \mathcal{N}_q q^{(n-1)^2+\varepsilon} (r_1^2r_2)^{n-1} (s_1^2s_2)^{\frac{1}{2}(n-1)(n-2)},$$
 and the claim follows. 

%\noindent \textbf{Remark:} To remove the squarefree condition, we would need a version of Lemma \ref{technical} 

\section{The Kuznetsov formula}\label{sec5}

Let $E$ be a fixed compactly supported function on $\Bbb{R}_{>0}^{n-1}$, $X \in \Bbb{R}_{>0}^{n-1}$ a ``parameter'' and define  
%\begin{equation}\label{ex}
   $$E^{(X)}(y_1, \ldots, y_{n-1}) = E(X_1y_1,  \ldots, X_{n-1}y_{n-1}).$$
%\end{equation}
Next we define the right ${\rm O}_n(\Bbb{R}){Z}^+$  invariant function $F^{(X)} : {\rm GL}_n(\Bbb{R}) \rightarrow \Bbb{C}$ by 
\begin{equation}\label{F}
     F^{(X)}(xyk\alpha) = \theta(x) E^{(X)}({\rm y}(y)) %\quad E^{(X)}(z) = E(X\cdot z),
\end{equation}
for $x \in U(\Bbb{R})$, $y \in \tilde{T}(\Bbb{R})$, $k \in {\rm O}_n(\Bbb{R})$, $\alpha \in {Z}^{+} $,  $\theta = \theta_{(1, \ldots, 1)}$ as in \eqref{character} and ${\rm y}(y)$ as in \eqref{y}. Recalling \eqref{adv}, from the Fourier expansion of Poincar\'e series \cite[Theorem A]{Fr} one concludes as in \cite[Lemma 7]{Bl} the following Kuznetsov-type formula, using the notation \eqref{iota}, \eqref{y}, \eqref{wy}, \eqref{eta}, \eqref{inner}, \eqref{cast}, \eqref{spec}. 
      
\begin{lemma}\label{kuz-formula} Let $M, N \in \Bbb{N}^{n-1}$, $X \in \Bbb{R}_{> 0}^{n-1}$, $E$ a compactly supported function on $\Bbb{R}_{> 0}^{n-1}$ and define $F^{(X)}$ as in \eqref{F}. Let
$$A = \iota(X \cdot M)c^{\ast} w \iota(X \cdot N)^{-1} w^{-1} = \iota\big( X \cdot M \cdot{}^w(X \cdot N) \big)c^{\ast} \in T(\Bbb{R}).$$ 
Then
\begin{equation}\label{formula}
	\begin{split}
	&\int_{\Gamma(q)^{\natural}} \hspace{-0.25cm}\overline{A_{\varpi}(M)}  A_{\varpi}(N)  |\langle W_{\mu_{\varpi}}, E^{(X)}\rangle|^2 \dd \varpi = \sum_{w\in W} \sum_{v \in V} \sum_{c \in \Bbb{N}^{n-1}} \hspace{-0.25cm}\frac{S_{q, w}^v(M, N, c) }{c_1 \cdots c_{n-1}}  \\
	&\quad\quad\quad\frac{X^{2\eta}}{ {\rm y}(A)^{\eta}} \int_{\tilde{T}(\Bbb{R})} \int_{U_w(\Bbb{R}) } F^{(X)}( \iota(X)^{-1} A w xy ) \theta^v(-x) \overline{E( {\rm y}( y))} \dd x\, \dd^{\ast}y.
	\end{split}
\end{equation}
\end{lemma}

Note that $A_{\varpi}(M) = 0$ if $\varpi$ belongs to the residual spectrum.

We analyze the right hand side  of \eqref{formula} for a  test function $E$ with support in some rectangle $[r, R]^{n-1} \subseteq \Bbb{R}_{>0}^{n-1}$, 
$N = M = (m,1, \ldots, 1)$ for some $(m, q) = 1$ and  
\begin{equation}\label{weight}
   X = (Z, 1, \ldots, 1)
\end{equation}   
    for some $Z \geq 1$ where 
\begin{equation}\label{mz}
mZ \leq K^{-1} (1 + 1/r + R)^{-K} q^{n+1}
\end{equation}
%and a test function $E$ with support in some rectangle $[r, R]^{n-1}$. 
with the convention from Section \ref{sec21}. We introduce the notation
$$A \preccurlyeq B \quad :\Longleftrightarrow \quad A\ll (1 + 1/r + R)^K B.$$

For $w = \text{id}$ we have $A = I_n$, $U_w(\Bbb{R}) = \{I_n\}$ in Lemma \ref{kuz-formula}, and by \eqref{trivialweyl} we conclude that the contribution of $w = \text{id}$ is 
\begin{equation}\label{trivw}
\preccurlyeq   \mathcal{N}_q Z^{2\eta_1}. 
\end{equation}
Assume from now on $w \not= \text{id}$ and $w$ as in \eqref{w}. 

We  quote \cite[Lemma 1]{Bl}:   for   $w\in W$, $x \in U_w(\Bbb{R})$, $y, c, B \in \Bbb{R}_{>0}^{n-1}$ write ${\rm y}\big(\iota(B)c^{\ast}w x\iota(y) \big) = Y \in \Bbb{R}_{>0}^{n-1}$ and $A = \iota(B) c^{\ast}$.  Then
\begin{equation*}%\label{first}
\begin{split}
&c_j  \ll_{y, Y} %c_j \Delta_j(wx) = 
\prod_{i=1}^{n-1}  B_i^{s(i,j)}\end{split}
\end{equation*}
for $1 \leq j \leq n-1$ if the right hand side of \eqref{formula} is nonzero. From \cite[(3.3), (3.5), (3.6)]{Bl} it is obvious that the implied constant is of the form $(1/y + 1/Y + y + Y)^K$. 

As in \cite[(7.1)]{Bl} we apply this with $B = X \cdot M \cdot {}^w(X \cdot N)$ to see that the right hand side of \eqref{formula} vanishes  unless $c_j \preccurlyeq mZ$ for $1 \leq j \leq n-1$. We now apply Lemma \ref{lem41} and note that by \eqref{comp-equiv} the moduli for $w_1$ must be of the form as in part (b) of that lemma. By \eqref{mz}  (with $K$ sufficiently large) we conclude that the contribution of all $w$ except $w_{\ast}$ vanishes. Our next aim is to estimate  the $x, y$-integral in \eqref{formula} for $w = w_{\ast}$ which we estimate trivially by 
\begin{displaymath}
\begin{split}
  \int_{\tilde{T}(\Bbb{R})} \int_{U_{w_{\ast}}(\Bbb{R}) } |E({\rm y}(A w_{\ast} xy) )  E( {\rm y}( y))| \dd x\, \dd^{\ast}y.
\end{split}
\end{displaymath}
Here we argue verbatim as in \cite[Section 7]{Bl} and invoke \cite[Lemma 1]{Bl} and \cite[Lemma 3]{Bl} (with implied constants that are polynomial in the support of $E$) to conclude 
 $$ \int_{\tilde{T}(\Bbb{R})} \int_{U_{w_{\ast}}(\Bbb{R}) } F^{(X)}( \iota(X)^{-1} A w_{\ast} xy ) \theta^v(-x) \overline{E( {\rm y}( y))} \dd x\, \dd^{\ast}y \preccurlyeq {\rm y}(A)^{\eta(1+\varepsilon)}. $$     
Thus by Theorem \ref{thm32}  we see that the contribution of $w_{\ast}$ in \eqref{formula} is
\begin{displaymath}
    \begin{split}
    & \preccurlyeq q^{\varepsilon} \frac{\mathcal{N}_q  } {q^{n-1}} Z^{2\eta_1}  \sum_{q^n \mid c_1, \ldots, c_{n-1}\ll q^{n+1}} \Big( \frac{c_1}{q^n} \cdots \frac{c_{n-1}}{q^n}, q^{\infty}\Big)% \frac{\mathcal{N}_q q^{(n-1)^2} (c_1/q^n) \cdots (c_{n-1}/q^n)(c_1/q^n, q^{\infty}) \cdots (c_{n-1}/q^n,  q^{\infty})}{c_1\cdots c_{n-1}} \\
 %   & =    q^{\varepsilon} \frac{\mathcal{N}_q  } {q^{n-1}}  \sum_{\gamma_1, \ldots, \gamma_{n-1}\ll q}(\gamma_1, q^{\infty}) \cdots (\gamma_{n-1},  q^{\infty}) Z^{2\eta_1}
  \leq q^{\varepsilon}  \mathcal{N}_q  Z^{2\eta_1} \Big(\sum_{\substack{\gamma < q \\ \gamma \mid q^{\infty}}} 1 \Big)^{n-1}   \preccurlyeq  \mathcal{N}_q  Z^{2\eta_1}  \end{split}
\end{displaymath}
by a standard application of Rankin's trick in the last step. This bound majorizes \eqref{trivw} and bounds the right hand side of \eqref{formula} under the current hypotheses on $N, M, X, E$. From \eqref{hecke} we therefore obtain
$$\int_{\Gamma(q)^{\natural}} |A_{\varpi}(\textbf{1})|^2 |\lambda_{\varpi}(m)|^2   |\langle W_{\mu_{\varpi}}, E^{(Z, 1, \ldots, 1))}\rangle|^2 \dd \varpi  \preccurlyeq q^{\varepsilon}  \mathcal{N}_q Z^{2\eta_1},$$ provided that $(m, q) = 1$ and \eqref{mz} holds where $r, R$ are determined by the support of $E$. Combining this with \eqref{finite} and Lemma \ref{whit}, we conclude the following.
\begin{prop}\label{density-kuz} There exists an absolute constant $K > 0$ with the following property. Fix $M, T > 1$, $q$ squarefree,  and suppose that $T \leq M^{-K} q^{n+1}$. Fix a place $v$ of $\Bbb{Q}$.  If $v = p$ is finite, assume that $p \nmid q$. 
Then %\marginpar{\tiny{throughout the paper we use simultaneously $\sigma_{\pi}$ and $\sigma_{\varpi}$ \\ I have tried to make this more consistent.}}%\marginpar{\tiny{We could include more representations, e.g. non-residual Eisenstein series
$$\underset{ \| \mu_{\varpi}\| \leq M}{ \int_{\Gamma(q)^{\natural}}} % \sum_{\substack{\pi \text{\rm { cusp.}}\\ \| \mu_{\pi}(\infty) \| \leq M}} 
|A_{\varpi}(\textbf{{\rm 1}})|^2 T^{2 \sigma_{\varpi,v}} \dd \varpi   \ll_{v, \varepsilon} M^{K} q^{\varepsilon}  \mathcal{N}_q .$$   
\end{prop} 
%(Note that without loss  of generality we can assume $v \not= q$, otherwise there is nothing to show.)

\section{Fourier coefficients}\label{sec6}
   
In this section we prove Theorem \ref{thm15} and Theorem \ref{thm1}. 

\subsection{Towards local computations} Recall %\marginpar{\tiny{I re-arranged this slightly}} 
that given a cuspidal automorphic representation $\pi\mid L^2(X_q)$ as in \eqref{pimidxq} we defined in Section \ref{sec:adelising}  the associated subspace $V_{\pi}$ in $L^2(X_q)$ by de-adelisation. The first step in proving Theorem \ref{thm15} is to reduce it to certain local computations. We factor $\pi=\otimes_v \pi_v$. For finite places $v=p$ we write $\pi_p^{K_p(p^m)}$ for the space of $K_p(p^m)$-invariant vectors in $\pi_p$ with $K_p(q)$ as in \eqref{cpq}. We also fix an invariant inner product $\langle \cdot,\cdot\rangle_{\pi_p}$ on $\pi_p$. Recall, from \cite[Definition~2.1]{LM}, that the stable integral $\int^{\text{st}}_{U(\Bbb{Q}_p)}$ is defined by requiring that the corresponding integrals over sufficiently large open compact subgroups of $U(\Bbb{Q}_p)$ stabilizes.  Define the local quantity % \marginpar{\tiny{Why do you divide by $\langle v, v\rangle_{\pi_p}$. Isn't that 1? \\ I removed this but added a remark below that the quantity is independent of the choice of inner product.}} 
\begin{equation}
	\mathcal{S}_{\pi_p,p^m}(y) = \sum_{v\in {\rm ONB}(\pi_p^{K_p(p^m)})} \int_{U(\Bbb{Q}_p)}^{{\rm st}} \langle  \pi_p(u)v,v\rangle_{\pi_p} \boldsymbol{\psi}_{\Bbb{Q}_p}(yuy^{-1})^{-1}\dd u. \nonumber
\end{equation}
Note that this quantity is independent of the choice of $\langle \cdot,\cdot\rangle_{\pi_p}$. Let ${\rm Wh}(\pi_p)$ be the local Whittaker model for $\pi_p$ and fix an isomorphism $\pi_p\ni F_p \mapsto W_{F_p}\in {\rm Wh}(\pi_p).$ This amounts to choosing (the normalization of) the Whittaker functional. Recall the definition \eqref{dq} of $D_q$ and set $\boldsymbol{\psi}_{\Bbb{Q}_p}^{\natural}(u)=\boldsymbol{\psi}_{\Bbb{Q}_p}(D_quD_q^{-1})$. 

We start by inserting \eqref{eq:single_Fourier_coeff} to obtain
\begin{displaymath}
\begin{split}
	 \sum_{\varpi\in {\rm ONB}(V_{\pi})} \vert A_{\varpi}(1)\vert^2 &= \frac{C(\pi)}{\mathcal{V}_q} \sum_{F\in {\rm ONB}(\pi^{K(q)^{\natural}})} \prod_{p\mid q}I_p(W_{F,p})% \nonumber
	 \\
	& = \frac{C(\pi)}{\mathcal{V}_q} \prod_{p\mid q}\sum_{F_p\in {\rm ONB}(\pi_p^{K_p(q)})} I_p(W_{\pi_p(D_q)F_p}).
	 \end{split}
\end{displaymath}
We compute 
\begin{displaymath}
   \begin{split}
	   I_p(W_{\pi_p(D_q)F_p}) &= \int_{U(\Bbb{Q}_p)}^{{\rm st}} \frac{\langle  \pi_v(uD_q)F_p,\pi_p(D_q)F_p\rangle_{\pi_p}}{\langle \pi_p(D_q)F_p,\pi_p(D_q)F_p\rangle_{\pi_p}} \boldsymbol{\psi}_{\Bbb{Q}_p}(u)^{-1}\dd u \\ &= p^{v_p(q)\cdot \frac{1}{6} (n+1)n(n-1)}\int_{U(\Bbb{Q}_p)}^{{\rm st}} \frac{\langle  \pi_p(u)F_p,F_p\rangle_{\pi_p}}{\langle F_p, F_p\rangle_{\pi_p}} \boldsymbol{\psi}_{\Bbb{Q}_p}^{\natural}(u)^{-1}\dd u. 
	\end{split}
\end{displaymath}
Thus we have
\begin{equation}
	  \sum_{\varpi\in {\rm ONB}(V_{\pi})} \vert A_{\varpi}(1)\vert^2 = C(\pi)\cdot \frac{\mathcal{N}_q}{\mathcal{V}_q} q^{\frac{n(n-1)}{2}} \cdot \prod_{p\mid q} \mathcal{S}_{\pi_p,p^{v_p(q)}}(D_q). \nonumber
\end{equation}
Recall the lower bound $C(\pi) \gg (\Vert \mu_{\pi}\Vert q)^{-\varepsilon}$ from \eqref{eq:lower_bound_cpi}. Since $\dim_{\Bbb{C}}(V_{\pi}) = \prod_{p\mid q} \dim_{\Bbb{C}}(\pi_p^{K_p(p^{v_p(q)})})$, we conclude  the following. 
\begin{prop}\label{prop:global_to_local}
Let $q\in\Bbb{N}$ be an integer. Assume that 
\begin{equation}
	\mathcal{S}_{\pi_p,p^m}(D_q) \gg 1.\label{local_assumption_1}
\end{equation}
and
\begin{equation}
	\dim_{\Bbb{C}}(\pi_p^{K_p(p^m)})\ll  p^{m\cdot \frac{n(n-1)}{2}}\label{local_assumption_2}
\end{equation}
for all $p^m\Vert q$. For any cuspidal automorphic representation $\pi\mid L^2(X_q)$ we have
\begin{equation}
	\sum_{\varpi\in {\rm ONB}(V_{\pi})} \vert A_{\varpi}(1)\vert^2\gg (\Vert \mu_{\pi}\Vert q)^{-\varepsilon}\dim(V_{\pi})\frac{q^{n(n-1)(n-2)/6}}{q^{n^2-1}}.\nonumber
\end{equation}
\end{prop}

We will now proceed by showing \eqref{local_assumption_1} and \eqref{local_assumption_2} for all primes $p$ and $m=1$. Thus the conclusion of the proposition above holds for all squarefree $q$; which is precisely what is claimed in Theorem \ref{thm15}. We expect that \eqref{local_assumption_1} and \eqref{local_assumption_2} hold for general $m$ and most likely the bounds are essentially sharp for generic irreducible representations $\pi_p$. In the toy case $n=2$ computations verifying these assertions can be found in \cite{As} and \cite{miyauchi-yamauchi_remark}. We also refer to \cite{La2} for some results on the support of matrix coefficients. %\marginpar{\tiny{We can remove the reference to \cite{As} if this is inappropriate.}}

From now on we work locally at a place $v=p$. Throughout $\pi_p$ will be an irreducible generic smooth admissible (unitary) representation of ${\rm GL}_n(\Bbb{Q}_p)$. We write $c(\pi_p)$ for the exponent-conductor. This is the smallest non-negative integer $m$ such that $\pi_p$ has non-trivial $K_{1,p}(m)$-fixed vectors where $K_{1,p}(m)$ is the standard congruence subgroup consisting of matrices in ${\rm GL}_n(\Bbb{Z}_p)$ with bottom row congruent to $(0,\ldots,0,1)$ modulo $p^m$. We write $\rho(\pi_p)$ for the depth of $\pi_p$. Finally, given any open compact subgroup $C\subset {\rm GL}_n(\Bbb{Q}_p)$ we slightly abuse notation and write $\pi_p^C$ for the space of $C$-fixed vectors in $\pi_p$. Let $\omega_{\pi_p}$ be the central character of $\pi_p$. Without loss of generality we can assume that $\omega_{\pi_p}(p)=1$, as this property can be achieved by twisting $\pi_p$ with an unramified character if necessary.

First we recall that there is a standard parabolic subgroup $P=MN\subset {\rm GL}_n$ such that 
\begin{equation}
	\pi_p = \text{Ind}_P^G(\sigma), \nonumber
\end{equation}
where $\sigma$ is an irreducible essentially square integrable representation of $M$. By the classification of standard parabolic subgroups we find a partition $(n_1,\ldots,n_k)$ of $n$ such that $M \cong {\rm GL}_{n_1}(\Bbb{Q}_p)\times \ldots \times {\rm GL}_{n_k}(\Bbb{Q}_p)$. Thus we can write $\sigma = \tau_1\otimes \ldots\otimes \tau_k$, for irreducible essentially square integrable representations $\tau_i$ of ${\rm GL}_{n_i}(\Bbb{Q}_p)$ for $i=1,\ldots,k$ with the usual adjectives.

The following result classifies representations with $K_p(p^m)$-fixed vectors  in terms of conductors. This is certainly well-known to experts, but has been explicitly written down in general only recently in the preprint \cite[Theorem~1.2]{miyauchi-yamauchi_remark}. 

\begin{lemma}
Let $\pi$ be as above. Then $\pi_p^{K_p(p^m)}\neq \{0\}$ if and only if $c(\tau_i)\leq m\cdot n_i$.
\end{lemma}
\begin{proof}
For $\pi$ essentially square integrable (i.e.\ $k=1$) this was shown in \cite[Proposition~2]{bushnell-henniart_counting}.\footnote{Note that the authors in loc.\ cit.\ make other assumptions on the additive character $\psi'$. This implies that their conductor $f(\pi,\psi')$, which is defined via the epsilon factor, differs from ours, namely $f(\pi,\psi')=c(\pi)+n$.} The general situation ($k>1$) is easily reduced to this case by using \cite[Lemma 23.2]{NV}.
\end{proof}

Given an irreducible admissible representation $\tau_i$ of ${\rm GL}_{n_i}(\Bbb{Q}_p)$,   we can associate an $n_i$-dimensional Weil-Deligne representation $\rho_{i}=(r_i,N_i)$ via the local Langlands correspondence. Note that $\tau_i$ is essentially square integrable if and only if $\rho_i$ is indecomposable. Furthermore, $\tau_i$ is supercuspidal if and only if $\rho_i$ is irreducible; see \cite[Section~1.2.3]{wedhorn_llc}. According to \cite[Section~5]{rohrlich_elliptic},  any indecomposable Weil-Deligne representation is of the form $\rho_i\cong (r_i,0)\otimes \text{sp}(k_i)$, where $k_i\mid n_i$ and $r_i$ is an irreducible $\frac{n_i}{k_i}$-dimensional representation of the Weil-group (of $\Bbb{Q}_p$). We write $c(r_i)$ for the exponent conductor of $r_i$. This conductor can be defined on the Weil-Deligne side, but for our purposes it suffices to define $c(r_i)$ for irreducible $r_i$ via the local Langlands correspondence. The conductor of essentially square integrable representations has been computed in \cite[Section~10]{rohrlich_elliptic}. We have 
\begin{equation}
	c(\tau_i) = \begin{cases}
		k_i\cdot c(r_i) &\text{ if }r_i \text{ is ramified},\\
		k_i-1 &\text{ if }r_i \text{ is unramified}.
	\end{cases}
\end{equation}
If $k_i=n_i$, then $r_i$ is nothing but a character $r_i=\chi$ and $c(r_i)$ coincides with the usual exponent conductor of $\chi$. In this case $\tau_i=\chi\text{St}_{n_i}$ is a twist of the Steinberg representation. This gives us a good understanding of the representations we need to consider in view of resolving \eqref{local_assumption_1} and \eqref{local_assumption_2} case by case. The situations we need to consider are given in the following lemma, which follows directly from our discussion so far.

\begin{lemma}\label{lm:class}
Suppose $\pi_p$ is irreducible, generic and $\pi_p^{K_p(p)}$ is non-trivial. Then we must be in one of the following cases.
\begin{itemize}
	\item \textbf{Case~1:} $\pi_p$ is an (essentially) square integrable representation with $c(\pi_p)\leq n$. This includes the following two sub-cases:
	\begin{itemize}
		\item \textbf{Case~1.a:} $\pi_p$ is supercuspidal with $c(\pi_p)=n$ (i.e.\ depth-zero).
		\item \textbf{Case~1.b:} $\pi_p$ is a special representation with $c(\pi_p)\leq n$. Thus there is $d\mid n$ and $\tau$ supercuspidal with $c(\tau)\leq \frac{n}{d}$ so that $\pi_p$ is the special representation $\text{St}(\tau,d)$. (If $d=n$, then $\tau=\chi$ is a character with $c(\chi)\leq 1$ and we have $\pi_p =\chi\cdot \text{St}_n$, where we write $\text{St}_n$ for the Steinberg representation.)
	\end{itemize} 
	\item \textbf{Case~2:} There is a parabolic subgroup $P=MN$ and an essentially square integrable representation $\sigma=\tau_1\otimes\ldots\otimes\tau_k$ of $M$ with $c(\tau_i) \leq n_i$ such that ${\rm Ind}_{P(\Bbb{Q}_p)}^{{\rm GL}_n(\Bbb{Q}_p)}(\sigma)$ is irreducible and $\pi_p={\rm Ind}_{P(\Bbb{Q}_p)}^{{\rm GL}_n(\Bbb{Q}_p)}(\sigma)$.
\end{itemize}
\end{lemma}

\subsection{The case of  depth-zero supercuspidal representations (Case 1.a)}   Depth-zero supercuspidal representations of ${\rm GL}_n(\Bbb{Q}_p)$  can be constructed as follows. We take a cuspidal representation $(\mu,W)$ of ${\rm GL}_n(\Bbb{Z}/p\Bbb{Z})$, inflate it to a representation of $K_p$ and extend it to ${\rm Z}(\Bbb{Q}_p)K_p$ (taking the central character into account). The so obtained representation will also be denoted by $\mu$. We can now consider the compact induction  $\pi_p=\text{c-Ind}_{ZK_p}^{\rm GL_n(\Bbb{Q}_p)}(\mu)$. This representation acts by right translation on the space of functions
\begin{displaymath}
\begin{split}
	V_c(\mu) = \{f\colon & {\rm GL}_n(\Bbb{Q}_p)\to W\colon \text{compactly supported mod }Z, \\ & f(hg) = \mu(h)f(g)\text{ for all }g\in {\rm GL}_n(\Bbb{Q}_p) \text{ and }h\in {\rm Z}(\Bbb{Q}_p)K_p\}. \nonumber
	\end{split}
\end{displaymath}
The inner product is given by
\begin{equation}
	\langle f,g\rangle_{\pi_p} = \sum_{x\in {\rm Z}K_p\backslash {\rm GL}_n(\Bbb{Q}_p)} \langle f(x),g(x)\rangle_W.\nonumber 
\end{equation}
Note that the sum-notation is justified, since ${\rm supp}(f)\cup{\rm supp}(g)$ is compact so that we are summing a finite number of terms. For $w\in W$ and $y\in {\rm GL}_n(\Bbb{Q}_p)$ we can define %\marginpar{\tiny{why parentheses about $ZK_p$? \\ They were unnecessary.}} 
\begin{equation}
	f_{y,w}(g) =\begin{cases}
		\mu(h)w &\text{ if }g=hy\in {\rm Z}(\Bbb{Q}_p)K_p\cdot y, \\
		0 &\text{ else.}
	\end{cases} \nonumber
\end{equation}
The family $\{f_{y,w}\colon y\in  {\rm GL}_n(\Bbb{Q}_p), w\in W \}$ spans $V_c(\mu)$ and we have
\begin{equation}
\langle f_{y,w},f_{x,v} \rangle_{\pi} = \langle w,v \rangle_W\cdot \delta_{x\in {\rm Z}(\Bbb{Q}_p)K_py}.\nonumber
\end{equation}
With these preliminaries set up we can construct the following basis.

\begin{lemma}
Let $v_1,\ldots,v_l$ with $l=\dim_{\Bbb{C}}W$ be an orthonormal basis of $W$. Then $(f_{1,v_i})_{i=1}^{l}$ is an orthonormal basis of $V_c(\mu)^{K_p(p)}$.
\end{lemma}
\begin{proof}
In view of the above inner product relation it is clear that the functions $f_{1,v_i}$ are orthonormal. It is also clear that they are $K_p(p)$-invariant. Thus we only need to show that they exhaust the space $\pi_p^{K_p(p)}$. However, to see this it suffices to show that $f_{y,w}$ is not $K_p(p)$-invariant if $y\not\in {\rm Z}(\Bbb{Q}_p)K_p$. To this end write $g=zky\in {\rm Z}(\Bbb{Q}_p)K_p\cdot y$ for $k\in K_p$ and take $\gamma\in K_p(p)$. Then, writing $g\gamma = zk(y\gamma y^{-1})y$ yields the condition
\begin{equation}
	y\gamma y^{-1}\in K_p(p),\nonumber
\end{equation}
since otherwise we can choose $k$ so that $f_{y,v_i}(g) \neq f_{y,v_i}(g\gamma)$. We are done since the normalizer of $K_p(p)$ in ${\rm GL}_n(\Bbb{Q}_p)$ is ${\rm Z}(\Bbb{Q}_p)K_p$.
\end{proof}

\begin{lemma} \label{lm:sc_all_n}
For any $n\geq 2$ and $\pi_p$ depth-zero supercuspidal, we have
\begin{equation}
	\mathcal{S}_{\pi_p}(D_q) = 1.
\end{equation}
Furthermore, $$\dim \pi_p^{K_p(p)} = \prod_{i=1}^{n-1}(p^i-1) = p^{\frac{n(n-1)}{2}}(1+O(p^{-1})).$$ In particular, \eqref{local_assumption_1} and \eqref{local_assumption_2} hold for depth-zero supercuspidal representations $\pi_p$ and $m=1$. % (i.e. Case~1.a).
\end{lemma}
\begin{proof}
First we reduce the situation to a problem in the realm of representations of ${\rm GL}_n$ over finite fields. To do so we note that since $\pi_p$ is supercuspidal the matrix coefficients have compact support and we do not need to worry about the stabilization of the integral. We simply have
\begin{equation}
	\mathcal{S}_{\pi_p,p}(D_q) = \sum_{i=1}^l\int_{U(\Bbb{Q}_p)}\langle \pi_p(n)f_{1,v_i},f_{1,v_i}\rangle_{\pi} \boldsymbol{\psi}_{\Bbb{Q}_p}^{\natural}(u)^{-1}\dd u.\nonumber
\end{equation}

Note that $\boldsymbol{\psi}_{\Bbb{Q}_p}^{\natural}(u)$ is trivial for $u\in U(p\Bbb{Z}_p)$. Thus the character $\widetilde{\boldsymbol{\psi}} \colon U(\Bbb{Z}/p\Bbb{Z})\to \Bbb{C}^{\times}$ defined by
\begin{equation}
	\widetilde{\boldsymbol{\psi}}(\tilde{u}) = \boldsymbol{\psi}_{\Bbb{Q}_p}^{\natural}(u) \text{ for } \tilde{u}=[u]\in U(\Bbb{Z}_p)/U(p\Bbb{Z}_p) \cong U(\Bbb{Z}/p\Bbb{Z}) \nonumber
\end{equation}
is well defined and non-degenerate. Note that $$\langle \pi_p(u)f_{1,v_i},f_{1,v_i}\rangle_{\pi_p} =\langle \mu(u)v_i,v_i\rangle_W\delta_{u\in K_p}.$$ Thus we get
\begin{equation}
	\mathcal{S}_{\pi_p,p}(D_q) = \int_{U(\Bbb{Q}_p)\cap K_p}\sum_{i=1}^l\langle \mu(u)v_i,v_i\rangle_{W}\cdot \boldsymbol{\psi}_{\Bbb{Q}_p}^{\natural}(u)^{-1}\dd u.\nonumber
\end{equation}
Since $(v_i)_{i=1}^l$ is an orthonormal basis of $W$, we get $\sum_i\langle \mu(u)v_i,v_i \rangle_{W} = \text{tr}(\mu(u))$. Note that  $U(\Bbb{Q}_p) \cap K_p = U(\Bbb{Z}_p)$, the integrand is $U(p\Bbb{Z}_p) = U(\Bbb{Q}_p)\cap K_p(p)$ invariant and ${\rm Vol}(U(p\Bbb{Z}_p),\dd u) = p^{-\frac{n(n-1)}{2}}$, so we obtain
\begin{equation}
	\mathcal{S}_{\pi_p,p}(D_q) = p^{-\frac{n(n-1)}{2}} \sum_{u\in U(\Bbb{Z}/p\Bbb{Z})} {\rm tr}(\mu(u)) \widetilde{\boldsymbol{\psi}}(u). \label{finite_char_average}
\end{equation}

We start by observing that ${\rm tr}\mu(g) =\chi_{\mu}(g)$ is the character of $\mu$ at $g$. We get that
\begin{equation}
	\dim \pi_p^{K_p(p)} = \dim(W) = \chi_{\mu}(1) = \prod_{i=1}^{n-1}(p^i-1).\nonumber
\end{equation} 
This proves the last part of our statement. (For the evaluation of the character at $1$ see for example \cite[Theorem~2]{prasad_finite}). To deduce the rest we need some more character theory. Given any finite group $H$ we define the inner product
\begin{equation}
	\langle f,g\rangle_H = \frac{1}{\# H}\sum_{h\in H} f(h)\overline{g(h)}.\nonumber 
\end{equation}
With this at hand we can rephrase the statement of \eqref{finite_char_average} as   
\begin{equation}
	\mathcal{S}_{\pi}(a) = \langle \chi_{\mu}\vert_{U(\Bbb{Z}/p\Bbb{Z})}, \widetilde{\boldsymbol{\psi}} \rangle_{U(\Bbb{Z}/p\Bbb{Z})}.\nonumber
\end{equation}
We write $\widetilde{\boldsymbol{\psi}}^G$ for the character of the representation obtained by inducing $\widetilde{\boldsymbol{\psi}}$ from $U(\Bbb{Z}/p\Bbb{Z})$ to ${\rm GL}_n(\Bbb{Z}/p\Bbb{Z})$. Let $\mathcal{G}_{\widetilde{\boldsymbol{\psi}}}$ denote the space of this induced representation. By Frobenius reciprocity we have
\begin{equation}
	\langle \chi_{\mu}\vert_{U(\Bbb{Z}/p\Bbb{Z})}, \widetilde{\boldsymbol{\psi}} \rangle_{U(\Bbb{Z}/p\Bbb{Z})} = \langle \chi_{\mu}, \widetilde{\boldsymbol{\psi}}^G\rangle_{{\rm GL}_n(\Bbb{Z}/p\Bbb{Z})} = {\rm Hom}_{{\rm GL}_n(\Bbb{Z}/p\Bbb{Z})}(W,\mathcal{G}_{\widetilde{\boldsymbol{\psi}}}).\nonumber
\end{equation}
	
Recall that $\widetilde{\boldsymbol{\psi}}$ is non-degenerate. In this case it is a well known theorem (multiplicity one for $\mathcal{G}_{\widetilde{\boldsymbol{\psi}}}$) due to Gelfand-Graev that $$\dim_{\Bbb{C}}{\rm Hom}_{{\rm GL}_n(\Bbb{Z}/p\Bbb{Z})}(W,\mathcal{G}_{\widetilde{\boldsymbol{\psi}}})\leq 1.$$ However, by \cite[Corollary~5.4]{silberger_zink-characters} the dimension is precisely one for cuspidal representations $(\mu,W)$ (i.e.\ cuspidal representations are $\widetilde{\boldsymbol{\psi}}$-generic). This completes the proof.
\end{proof}

Thus for depth-zero supercuspidal representations $\pi_p$ we could compute $\mathcal{S}_{\pi_p}(D_q)$ and $\dim_{\Bbb{C}}(\pi_p^{K_p(p)})$ explicitly showing that \eqref{local_assumption_1} and \eqref{local_assumption_2} hold on the nose. Here we crucially used the construction of $\pi_p$ from certain representations of ${\rm GL}_n(\Bbb{Z}/p\Bbb{Z})$ and the representation theory thereof. To treat general supercuspidal representations $\pi_p$ a different argument is needed.

\subsection{Parabolically induced representations.} Suppose we have a partition $n=n_1+\ldots+n_k$ and (essentially) square integrable representations $\tau_1,\ldots,\tau_k$. Then we consider the representation $\sigma=\tau_1\otimes\ldots \otimes \tau_k$ of $M\cong {\rm GL}_{n_1}(\Bbb{Q}_p)
\times \ldots  \times {\rm GL}_{n_k}(\Bbb{Q}_p)$. If $P=MN$ is the standard parabolic subgroup with Levi component $M$, we consider the parabolically induced representation 
\begin{equation}
	\xi={\rm Ind}_{P(\Bbb{Q}_p)}^{{\rm GL}_n(\Bbb{Q}_p)}(\sigma).\nonumber
\end{equation} 
Note that at the moment we are not assuming that this representation is unitary or irreducible. 
\begin{lemma}\label{lm:dimensions_ps}
For $m\geq 1$ we have 
\begin{equation}
	\dim_{\Bbb{C}} (\xi^{K_p(p^m)}) = \# P(\Bbb{Q}_p)\backslash {\rm GL}_n(\Bbb{Q}_p)/K_p(p^m)\cdot \prod_{i=1}^k \dim \tau_i^{K_p(p^m)}. \nonumber
\end{equation}
In particular, \eqref{local_assumption_2} holds for $m=1$ and all smooth admissible representations $\pi_p$ of ${\rm GL}_n(\Bbb{Q}_p)$.
\end{lemma}
\begin{proof}
The dimension formula is given in \cite[Lemma~3.1]{miyauchi-yamauchi_remark} or \cite[Lemma 23.2]{NV}. The rest is a direct consequence of Lemma~\ref{lm:sc_all_n} and the Bernstein-Zelevinsky classification.	
\end{proof}

To make things more precise, we fix representatives $\gamma_1,\ldots, \gamma_l$ of the double quotient $P(\Bbb{Q}_p)\backslash {\rm GL}_n(\Bbb{Q}_p)/K_p(p^m)$ and an orthonormal basis $(v_i)_{i=1}^r$ of $\sigma^{K_p(p^m)}$. We can parametrize a basis for $\xi^{K_p(p^m)}$ by tuples $(\gamma_i,v_j)$ with $1\leq i\leq l$ and $1\leq j\leq r$. In particular $\dim \xi^{K_p(p^m)}=lr$. We model $\xi$ in the induced picture as follows. Suppose the representation $\sigma$ is modeled on a space $W$. Then we consider the space $V(\sigma)$ of smooth functions $f\colon {\rm GL}_n(\Bbb{Q}_p)\to W$ such that\footnote{The double use of $m\in\Bbb{N}$ and $m\in M(\Bbb{Q}_p)$ should not lead to confusion.} $$f(mng) = \delta_P(m)^{\frac{1}{2}}\sigma(m)f(g)$$ for all $g\in {\rm GL}_n(\Bbb{Q}_p)$, $n\in N(\Bbb{Q}_p)$ and $m\in M(\Bbb{Q}_p)$ where $\delta_P$ is the modular function of $P$. 

\begin{lemma}
A basis for $\xi^{K_p(p^m)}$ in the induced model is given by %\marginpar{\tiny{If we want representatives $\gamma_i$ in $P\backslash GL_n/C_p(p^m)$, then we need $\xi(\gamma_i^{-1})f_{1,j}$ instead of $\xi(\gamma_i)f_{1,j}$. This has some consequences below which I have fixed.}} 
$${\rm B}(\xi^{K_p(p^m)}) = \{ \xi(\gamma_i^{-1})f_{1,j}\colon 1\leq j\leq r \text{ and }1\leq i\leq l\},$$ with 
\begin{multline}
	f_{1,j}(g) = {\rm Vol}(P(\Bbb{Z}_p)K_p(p^m), \dd k)^{-\frac{1}{2}} \\\quad \cdot \begin{cases}
	\delta_P(m)^{\frac{1}{2}}\sigma(m)v_j &\text{ if }g=nmk \text{ for }k\in K_p(p^m),\, n\in N(\Bbb{Q}_p),\, m\in M(\Bbb{Q}_p),\\
	0 &\text{ else.}
	\end{cases} \nonumber
\end{multline}
\end{lemma}
\begin{proof}
Showing that $f_{1,j}$ is well defined follows since $K_p(p^m)$ has an Iwahori factorization with respect to $P$. Verifying the rest of the statement is similarly straightforward. %We omit the details.
\end{proof}

Note that $\tilde{\xi} \cong {\rm Ind}_{P(\Bbb{Q}_p)}^{{\rm GL}_n(\Bbb{Q}_p)}(\tilde{\sigma})$. If we write $(\cdot, \cdot)\colon W\times \tilde{W}\to \Bbb{C}$ for the canonical pairing between $(\sigma,W)$ and its contragredient $(\tilde{\sigma}, \tilde{W})$, we obtain a canonical pairing between the induced representations given by
\begin{equation}
	(f,g) = \int_{K_p} (f(k),g(k))\dd k \nonumber
\end{equation}
for $f\in V(\sigma)$ and $g\in V(\tilde{\sigma})$. Let $\hat{v}_1,\ldots,\hat{v}_r$ be a basis of $\tilde{\sigma}^{K_p(p^m)}$,  dual to $v_1,\ldots, v_r$ with respect to $(\cdot,\cdot)$. Then we obtain a basis $\hat{f}_{i,j}$ with $1\leq i\leq l$ and $1\leq j\leq r$ of $V(\tilde{\sigma})^{K_p(p^m)}$ with
\begin{equation}
	(f_{i_1,j_1},\hat{f}_{i_2,j_2}) = \delta_{\substack{i_1=i_2, \,  j_1 = j_2}}.\nonumber
\end{equation}
If $\xi$ is tempered and irreducible, then we can replace the pairing $(\cdot,\cdot)\colon \sigma\times\tilde{\sigma}\to\Bbb{C}$ by an inner product on $W$. We then obtain an invariant inner product on $V(\sigma)$. In this case we can naturally identify $f_{i,j}$ with $\hat{f}_{i,j}$, so that the basis we constructed is orthonormal with respect to this inner product.

\begin{lemma} \label{lemma_mat_coeff}
Let $\xi$ be as above and consider the basis $(f_{i,j})_{\substack{1\leq j\leq r,\\1\leq i\leq l}}$ of $\xi^{K_p(p^m)}$ in the induced model. Then we have
\begin{equation}
	(\xi(u)f_{i,j},\hat{f}_{i,j}) = \frac{\#[P(\Bbb{Z}_p)/(K_p(p^m)\cap P(\Bbb{Q}_p))]^{\frac{1}{2}}}{{\rm Vol}(K_p(p^m),\dd k)^{\frac{1}{2}}}\int_{K_p(p^m)}(f_{1,j}( k'\gamma_iu\gamma_i^{-1}),\hat{v}_j)\dd k',\nonumber
\end{equation}
for $u\in U(\Bbb{Q}_p)$.
\end{lemma}
\begin{proof}
By the definition of the pairing $(\cdot,\cdot)$ in the induced picture we have
\begin{equation}
	(\xi(u)f_{i,j},\hat{f}_{i,j}) = \int_{K_p} ( f_{i,j}(ku),\hat{f}_{i,j}(k)) \dd k = \int_{K_p} ( f_{1,j}(k\gamma_iu\gamma_i^{-1}),\hat{f}_{1,j}(k)) \dd k.\nonumber
\end{equation}
We can write $k\in K_p$ as $k=\delta k'$ with $k'\in K_p(p^m)$ and a representative $\delta$ in $K_p/K_p(p^m)$. By the definition of $\hat{f}_1$ we find that $f_{1,j}(k) = f_{1,j}(\delta k')\neq 0$ unless $\delta\in P(\Bbb{Q}_p)K_p(p^m)$. Thus we get
\begin{align}
	(\xi(u)f_{i,j},\hat{f}_{i,j}) &= \sum_{\delta\in P(\Bbb{Z}_p)/(K_p(p^m)\cap P)}\int_{K_p(p^m)}( f_{1,j}(\delta k'\gamma_iu\gamma_i^{-1}),\hat{f}_{1,j}(\delta k'))\dd k' \nonumber \\ 
	&= \frac{\#[P(\Bbb{Z}_p)/(K_p(p^m)\cap P)]}{{\rm Vol}(P(\Bbb{Z}_p)K_p(p^m),\dd k)^{\frac{1}{2}}}\int_{K_p(p^m)}( f_{1,j}( k'\gamma_iu\gamma_i^{-1}),\hat{v}_j)\dd k' \nonumber \\
	&= \frac{\#[P(\Bbb{Z}_p)/(K_p(p^m)\cap P)]^{\frac{1}{2}}}{{\rm Vol}(K_p(p^m),\dd k)^{\frac{1}{2}}}\int_{K_p(p^m)}( f_{1,j}( k'\gamma_iu\gamma_i^{-1}),\hat{v}_j) \dd k'\nonumber
\end{align} 
and the proof is complete.
\end{proof}

By \cite[Proposition~2.3]{LM} the function $f(u)=(\xi(u)f_{i,j},\hat{f}_{i,j})\boldsymbol{\psi}_{\Bbb{Q}_p}^{\natural}(u)$ has a stable integral, so that we can define
\begin{equation}
	\mathcal{S}_{\xi, p^m}(D_q) = \sum_{\substack{1\leq i\leq l,\\ 1\leq j\leq r}} \int_{U(\Bbb{Q}_p)}^{{\rm st}}(\xi(u)f_{i,j},\hat{f}_{i,j})\boldsymbol{\psi}_{\Bbb{Q}_p}^{\natural}(u) du, \nonumber
\end{equation} 
which extends the definition of $\mathcal{S}_{\pi_p, p^m}(D_q)$ made for irreducible unitary representations $\pi_p$. Note that \cite[Proposition~2.3]{LM} is stated for irreducible representations $\xi$, however the proof does not require irreducibility.

We are now going to reduce the computation of $\mathcal{S}_{\xi,p^m}(D_q)$ to the corresponding averages $\mathcal{S}_{\tau_i,p^m}(D_q)$. For convenience we set $$\mathcal{S}_{\sigma,p^m}(D_q) =\mathcal{S}_{\tau_1,p^m}(D_q^{(n_1)})\cdot\ldots \cdot \mathcal{S}_{\tau_k,p^m}(D_q^{(n_k)}),$$ where $D_q=\text{diag}(D_q^{(n_1)},\ldots,D_q^{(n_k)})$ and $S_{\tau_i,p^m}(D_q^{(n_i)})=1$ if $\tau_i$ is a character (i.e.\ $n_i=1$). 

\begin{lemma}\label{lm:reduction}
In the notation above with $p^m\mid\mid q$ we have
\begin{equation}
	\mathcal{S}_{\xi,p^m}(D_q) =  \mathcal{S}_{\sigma, p^m}(D_q). \nonumber
\end{equation}
\end{lemma}

\begin{proof}
Let $P^{\op}=M^{\op}N^{\op}$ denote the parabolic opposite to  $P=MN$, and set $T=T_P = P\cap P^{\op}$. We have $P^{\top}=MN^{\top}$.  Given a representation $\sigma=\tau_1\otimes \ldots\otimes \tau_k$ we write $\sigma^{\op}=\tau_k\otimes \ldots\otimes \tau_1$ for the natural representation on $P^{\op}$. There is a Weyl element $w_P$ such that $w_P^{-1} Mw_P = M^{\op}$ and $w_P^{-1}Nw_P=(N^{\op})^{\top}$. The long Weyl element would do the job but we can even require that
\begin{equation}
	\sigma(m) = \sigma^{\op}(w_P^{-1}mw_P) \text{ for all }m\in M(\Bbb{Q}_p).\nonumber
\end{equation}

Our first goal is to show that the $U$-integral vanishes for many $i$. In the end only $1\leq i\leq l$ where $\gamma_i$ is of a very nice shape will contribute. Note that this behavior is special to $D_q$. Indeed we will use the following fact. For each simple root $\alpha$ we have $$\boldsymbol{\psi}_{\Bbb{Q}_p}^{\natural}(\, \cdot \,)\vert_{U_{\alpha}(p^{m-1}\Bbb{Z}_p)}\not \equiv 1\quad  \text{ and } \quad \boldsymbol{\psi}_{\Bbb{Q}_p}^{\natural}(\, \cdot \,)\vert_{U_{\alpha}(p^{m}\Bbb{Z}_p)}\equiv 1.$$
In particular, if there is a simple root $\alpha$ so that the matrix coefficient $(\pi(\cdot )f_{i,j},\hat{f}_{i,j})$ is right (or left) $(u')^{-1}U_{\alpha}(p^{m-1}\Bbb{Z}_p)u'$-invariant for some $u'\in U$, then the $U$-integral will vanish for this $i$. We will show this invariance property for some of the $i$'s. Recall that $\gamma_1,\ldots,\gamma_l$ are representatives for $P(\Bbb{Q}_p)\backslash {\rm GL}_n(\Bbb{Q}_p)/K_p(p^m)$, which we can choose to lie inside  $K_p$.

Note that $f_{1,j}$ is $N(\Bbb{Z}_p)$-invariant. Turning to the more general case, we take $\gamma_i\in K_p$ and consider the map 
\begin{equation}
	K_p \to K_p/K_p(p) = {\rm GL}_n(\Bbb{Z}/p\Bbb{Z}) = \bigcup_{w\in W^M\backslash W}P(\Bbb{Z}/p\Bbb{Z})wU(\Bbb{Z}/p\Bbb{Z})\nonumber
\end{equation}
and write $\overline{\gamma}_i$ for the image of $\gamma_i$. Suppose $\overline{\gamma}_i\in P(\Bbb{Z}/p\Bbb{Z})wU(\Bbb{Z}/p\Bbb{Z})$ with $w\neq w_M$, so that we can assume that $\gamma_i = wu_i+pA$  for $u_i\in U(\Bbb{Z}_p)$ and $A\in M_{n}(\Bbb{Z}_p)$. There is a simple root $\alpha$ with $U_{\alpha}\subset U\cap w^{-1} Nw$. We take $u_iu_{\alpha}u_i^{-1}\in U_{\alpha}(p^{m-1}\Bbb{Z}_p)$. Recall that $\pi(u_{\alpha})f_{i,j} = \pi(u_{\alpha}\gamma_i^{-1})f_{1,j}$. Thus in order to see that $f_{i,j}$ is $U_{\alpha}(p^{m-1}\Bbb{Z}_p)$-invariant it is enough to show that $\gamma_iu_{\alpha}\gamma_{i}^{-1}\in N(\Bbb{Z}_p)+p^mM_n(\Bbb{Z}_p)$. To see this we compute
\begin{align}
	\gamma_iu_{\alpha}\gamma_i^{-1} &= wu_iu_{\alpha}u_i^{-1}w^{-1} + pAu_{\alpha}u_i^{-1}w^{-1}+pwu_iu_{\alpha}A+p^2Au_{\alpha}A' \nonumber \\
	&\in wu_iu_{\alpha}u_i^{-1}w^{-1} + pAu_i^{-1}w^{-1}+pwu_iA+p^2AA'+p^mM_n(\Bbb{Z}_p)\nonumber \\
	&\subseteq wu_iu_{\alpha}u_i^{-1}w^{-1}+p^m M_n(\Bbb{Z}_p). \nonumber
\end{align}
Here we used $u_{\alpha} \in 1+p^{m-1}M_n(\Bbb{Z}_p)$ and $$1 =(wu_i+pA)(u_i^{-1}w^{-1}+pA')  = 1+pAu_i^{-1}w^{-1}+pwu_iA'+p^2AA'.$$ However, our choice of $u_{\alpha}$ ensures that $wu_iu_{\alpha}u_i^{-1}w^{-1}\in N(\Bbb{Z}_p)$. Thus we have seen that $u_{\alpha}\gamma_i^{-1} \in \gamma_i^{-1} N(\Bbb{Z}_p)K_p(p^m)$,  so that $f_{i,j}$ is $u_i^{-1}U_{\alpha}(p^{m-1}\Bbb{Z}_p)u_i$-invariant and the corresponding $U$-integral vanishes. In other words $\gamma_i$ does not contribute to $\mathcal{S}_{\pi,p^m}(D_q)$.

We have seen that only representatives $\gamma_i$ of $P(\Bbb{Q}_p)\backslash {\rm GL}_n(\Bbb{Q}_p)/K_p(p^m)$ of the shape $\gamma_i\in w_MN(\Bbb{Z}_p)K_p(p)$ contribute to the $i$-sum in $\mathcal{S}_{\pi,p^m}(D_q)$. However, due to the Iwahori factorisation of $K_p(p)$ we can write them as $\gamma_i\in P(\Bbb{Z}_p)w_MN(\Bbb{Z}_p)$. Thus we have seen that only $i$ with $\gamma_i\in P(\Bbb{Z}_p) w_MU(\Bbb{Z}_p)/K_p(p^m)$ contribute. Therefore we can put $j_0 = \# N^{\op}(\Bbb{Z}_p/p^m\Bbb{Z}_p)$ and consider only the representatives $\gamma_1,\ldots,\gamma_{j_0}$ of the form
\begin{equation}
	\gamma_i = w_M n_i \text{ with } n_i\in N^{\op}(\Bbb{Z}_p/p^m\Bbb{Z}_p). \nonumber
\end{equation}
Here we used the identification 
\begin{displaymath}
\begin{split}
P(\Bbb{Z}_p)\backslash P(\Bbb{Z}_p)w_MU(\Bbb{Z}_p)/K_p(p^m) &= P(\Bbb{Z}_p)\backslash P(\Bbb{Z}_p)w_MP^{\op}(\Bbb{Z}_p)/K_p(p^m)\\
&\cong w_MN^{\op}(\Bbb{Z}_p/p^m\Bbb{Z}_p).
\end{split}
\end{displaymath}
	
For these representatives we can compute the $U$-integral starting from the expression for the matrix coefficients given in Lemma~\ref{lemma_mat_coeff}. Note that a change of variables $u\mapsto n_i^{-1}un_i$ leaves the character invariant.  Thus we can assume without loss of generality that $n_i=n_1=1$ getting
\begin{align}
	&\int_{U(\Bbb{Q}_p)}^{{\rm st}}(\xi(u)f_{i,j},\hat{f}_{i,j})\boldsymbol{\psi}_{\Bbb{Q}_p}^{\natural}(u)^{-1}\dd u \nonumber \\
	&  = \int_{U(\Bbb{Q}_p)}^{{\rm st}} \hspace{-0.3cm} \frac{\#[P(\Bbb{Z}_p)/(K_p(p^m)\cap P(\Bbb{Q}_p))]^{\frac{1}{2}}}{{\rm Vol}(K_p(p^m),\dd k)^{\frac{1}{2}}}\int_{K_p(p^m)}\hspace{-0.3cm}(f_{1,j}( k'\gamma_iu\gamma_i^{-1}),\hat{v}_j)\dd k'\,\boldsymbol{\psi}_{\Bbb{Q}_p}^{\natural}(u)^{-1}\dd u \nonumber \\
	&  =\int_{U(\Bbb{Q}_p)}^{{\rm st}}\hspace{-0.3cm}\frac{\#[P(\Bbb{Z}_p)/(K_p(p^m)\cap P(\Bbb{Q}_p))]^{\frac{1}{2}}}{{\rm Vol}(K_p(p^m),\dd k)^{\frac{1}{2}}}\int_{K_p(p^m)}\hspace{-0.3cm}(f_{1,j}( k'w_Muw_M^{-1}),\hat{v}_j)\dd k'\,\boldsymbol{\psi}_{\Bbb{Q}_p}^{\natural}(u)^{-1}\dd u \nonumber \\ 
	& = \int_{U(\Bbb{Q}_p)}^{{\rm st}}(\xi(u)f_{1,j},\hat{f}_{1,j})\boldsymbol{\psi}_{\Bbb{Q}_p}^{\natural}(u)^{-1}\dd u.  \nonumber
\end{align}
Thus the integral is independent of $1\leq i\leq j_0$. In summary, 
\begin{align}
&	\mathcal{S}_{\xi, p^m}(D_q) = \sum_{\substack{1\leq i\leq l\\ 1\leq j\leq r}} \int_{U(\Bbb{Q}_p)}^{{\rm st}}(\xi(u)f_{i,j},\hat{f}_{i,j})\boldsymbol{\psi}_{\Bbb{Q}_p}^{\natural}(u) \dd u \nonumber\\
	&= \sum_{\substack{1\leq i\leq j_0\\ 1\leq j\leq r}} \int_{U(\Bbb{Q}_p)}^{{\rm st}}(\xi(u)f_{i,j},\hat{f}_{i,j})\boldsymbol{\psi}_{\Bbb{Q}_p}^{\natural}(u) \dd u= j_0  \sum_{1\leq j\leq r} \int_{U(\Bbb{Q}_p)}^{{\rm st}}(\xi(u)f_{1,j},\hat{f}_{1,j})\boldsymbol{\psi}_{\Bbb{Q}_p}^{\natural}(u) \dd u. \nonumber
\end{align}

In order to compute the remaining integrals we consider the condition $\gamma_1k'u\gamma_1^{-1} = pk\in P(\Bbb{Q}_p)K_p(p^m)$, which can be rewritten as	
\begin{equation}	
	k'u = q\tilde{k}\in (P^{\op})^{\top}(\Bbb{Q}_p)K_p(p^m), \,\, q=w_M^{-1}pw_M\in (P^{\op})^{\top}\Bbb{Q}_p, \,\,\tilde{k}=w_M^{-1} kw_M\in K_p(p^m).\nonumber
\end{equation}
At this point we consider the Iwahori decomposition $k'=n_-tn_+$ with $n_-\in (N^{\op})^{\top}(\Bbb{Q}_p)\cap K_p(p^m)$, $t\in M^{\op}(\Bbb{Q}_p)\cap K_p(p^m)$ and $n_+\in N^{\op}(\Bbb{Q}_p)\cap K_p(p^m)$. Inserting this in our condition we obtain $n_+u \in (P^{\op})^{\top}(\Bbb{Q}_p)K_p(p^m)$. We can now suppose without loss of generality that $u=u_{N^{\op}}\cdot  u_{M^{\op}}$ with $u_{N^{\op}}\in N^{\op}$ and $u_{M^{\op}}\in U\cap M^{\op}$. We obtain the condition $u_{N^{\op}}\in K_p(p^m)$. This leads to
\begin{equation}
	f_{1,j}( \gamma_1^{-1}k'u_{N^{\op}}u_{M^{\op}}\gamma_1) = {\rm Vol}(P(\Bbb{Z}_p)K_p(p^m),\dd k)^{-\frac{1}{2}}\delta_{u_{N^{\op}}\in K_p(p^m)} \sigma(w_Mu_{M^{\op}}w_M^{-1})v_j.\nonumber
\end{equation}
	
We now insert this observation in our formula for the matrix coefficient and get
\begin{equation}
	(\xi(u_{N^{\op}}u_{M^{\op}})f_{1,j},\hat{f}_{1,j}) = \delta_{u_{N^{\op}}\in K_p(p^m)}(\sigma(w_Mu_{M^{\op}}w_M^{-1})v_j,\hat{v}_j). \nonumber 
\end{equation}
Factor $U=(U\cap N^{\op})\cdot (U\cap M^{\op})$ and note that this factorization induces a factorization of the measures. We thus have
\begin{displaymath}
\begin{split}
&	\int_{U(\Bbb{Q}_p)}^{{\rm st}} (\xi(u)f_{1,j},\hat{f}_{1,j})\boldsymbol{\psi}_{\Bbb{Q}_p}^{\natural}(u)\dd u \\ &= {\rm Vol}(N^{\op}\cap K_p(p^m), \dd n^{\op})\int_{U(\Bbb{Q}_p)\cap M^{\op}(\Bbb{Q}_p)}^{{\rm st}} (\sigma(w_Pu'w_P^{-1})v_j,\hat{v}_j)\boldsymbol{\psi}_{\Bbb{Q}_p}^{\natural}(u')\dd u'.\nonumber
	\end{split}
\end{displaymath}
	
We now compute
\begin{equation}
	S_{\xi,p^m}(D_q) = j_0\cdot{\rm Vol}(N^{\op}\cap K_p(p^m), \dd n^{\op})\sum_{j=1}^r\int_{U\cap M^{\op}}^{{\rm st}} (\sigma^{\op}(u')v_j,\hat{v}_j)\boldsymbol{\psi}_{\Bbb{Q}_p}^{\natural}(u')\dd u' . \nonumber
\end{equation}
Since the definition of $\mathcal{S}_{\sigma,p^m}(D_q)$ is independent of the bases of $\tau_i^{K_p(p^m)}$ we recognize this as
\begin{equation}
	\sum_{j=1}^r\int_U^{{\rm st}}(\pi(u)f_{1,j},\hat{f}_{1,j})\boldsymbol{\psi}_{\Bbb{Q}_p}^{\natural}(u)\dd u \\ = {\rm Vol}(N^{\op}\cap K_p(p^m), \dd n^{\op})\mathcal{S}_{\sigma^{\op},p^m}(D_q).\nonumber
\end{equation}
Note that $\mathcal{S}_{\sigma^{\op},p^m}(D_q) = \mathcal{S}_{\sigma,p^m}(D_q)$. Since the same formula holds for any $i$ in place of $1$, we conclude by summing over $1\leq i\leq j_0$ and realising that $j_0={\rm Vol}(N^{\op}(\Bbb{Q}_p)\cap K_p(p^m), \dd n^{\op})^{-1}.$
\end{proof}

\emph{Remark:} The previous  lemma  allows us to reduce the computation of $\mathcal{S}_{\pi_p,p^m}$ to the case of supercuspidal representations.  Note that if $\xi$ is unitary, then the integrals $$I_{i,j}=\int_{U(\Bbb{Q}_p)}^{{\rm st}}(\xi(u)f_{i,j},f_{i,j})\boldsymbol{\psi}_{\Bbb{Q}_p}^{\natural}(u)\dd u$$ are all positive. Thus in this case one can simply drop all the tuples $(i,j)$ from the definition of $\mathcal{S}_{\xi,p^m}$, where $\gamma_i$ does not lie in $P(\Bbb{Q}_p)w_MN(\Bbb{Z}_p)K_p(p^m)$. The same argument as above then yields the lower bound
\begin{equation} \label{example}
	\mathcal{S}_{\xi,p^m}(D_q) \geq S_{\sigma,p^m}(D_q). 
\end{equation}
However, for our treatment of (non-supercuspidal) square integrable representations $\pi_p$ we need to consider non-unitary $\xi$. In this case the positivity of the integrals $I_{i,j}$ is not guaranteed, so that it is impossible to simply drop unwanted $i$'s from the sum. This motivates the additional work in the proof that went into actually proving equality in \eqref{example}, which is essential. Note that this equality will not hold for $S_{\xi,p^m}(y)$ in general. Indeed it is important in the proof that $\boldsymbol{\psi}_{\Bbb{Q}_p}^{\natural}$ is trivial on $U(\Bbb{Q}_p)\cap K_p(p^m)$ but non-trivial on $U(\Bbb{Q}_p)\cap K_p(p^{m-1})$.

\begin{lemma}
Let $q$ be a positive integer with $p^m\mid\mid q$. Suppose \eqref{local_assumption_1} holds for $m\in \Bbb{N}$ and for all supercuspidal representations of ${\rm GL}_l(\Bbb{Q}_p)$ with $2\leq l\leq n$. Then  we have 
\begin{equation}
	\mathcal{S}_{\pi_p,p^m}(D_q)\gg 1,
\end{equation}
for all smooth irreducible unitary representations $\pi_p$ of ${\rm GL}_n(\Bbb{Q}_p)$. In particular \eqref{local_assumption_1} holds unconditionally for $m=1$.
\end{lemma}
\begin{proof}
We start with the case of essentially square integrable representations. In this case there is $d\mid n$ and a supercuspidal representation $\tau$ of ${\rm GL}_{\frac{n}{d}}(\Bbb{Q}_p)$ so that $\pi_p = {\rm St}_n(\tau)$. In particular $\pi_p$ is the unique irreducible submodule of $\xi = {\rm Ind}_P^{{\rm GL}_n}(\vert \cdot\vert^{\frac{d-1}{2}}\cdot \tau\otimes \ldots \otimes \vert \cdot\vert^{-\frac{d-1}{2}}\cdot \tau)$, where $P$ is the parabolic subgroup associated to the partition $(\frac{n}{d},\ldots, \frac{n}{d})$. We write $\sigma = \vert \cdot\vert^{\frac{d-1}{2}}\cdot \tau\otimes \ldots \otimes \vert \cdot\vert^{-\frac{d-1}{2}}\cdot \tau$. Note that $\xi$ is not unitary itself. The lower bound for $\mathcal{S}_{\pi_p,p^m}(D_q)$ will follow from the assumption and Lemma~\ref{lm:reduction} once we have seen that
\begin{equation}
	\mathcal{S}_{\pi_p,p^m}(D_q) = \mathcal{S}_{\xi,p^m}(D_q).\nonumber
\end{equation}
Note that $\pi$ (resp.\ $\tilde{\pi}$) is the unique generic sub-quotient of $\xi$ (resp $\tilde{\xi}$). Thus the claim follows since $\mathcal{S}_{\xi,p^m}(D_q)$ is independent of the choice of basis and the pairing  $\xi\times \tilde{\xi}\to \Bbb{C}$ given by $$ (f,f^{\vee})\mapsto \int_{U(\Bbb{Q}_p)}^{{\rm st}}(\pi(u)f,f^{\vee})\boldsymbol{\psi}_{\Bbb{Q}_p}^{\natural}(u)du$$ descends to a pairing between the twisted Jacquet modules $\mathbf{j}_{\boldsymbol{\psi}_{\Bbb{Q}_p}^{\natural}}(\xi)$ and $\mathbf{j}_{(\boldsymbol{\psi}_{\Bbb{Q}_p}^{\natural})^{-1}}(\tilde{\xi})$. This is given in \cite[Proposition~2.3]{LM} and made more precise in \cite[Proposition~2.8]{LM}.

According to the Langlands classification it remains to consider irreducible representations (parabolically) induced from square-integrable representations. More precisely we can assume that $\pi_p = {\rm Ind}_{P(\Bbb{Q}_p)}^{{\rm GL}_n(\Bbb{Q}_p)}(\sigma)$, where $P =MN$ is a standard parabolic subgroup and $\sigma$ is an (essentially) square integrable representation of $M(\Bbb{Q}_p)$. However, since we have established the result for square integrable representations, it holds for $\sigma$. We conclude by applying Lemma~\ref{lm:dimensions_ps} and Lemma~\ref{lm:reduction} respectively.
\end{proof} 
   
This completes the proof of Theorem \ref{thm15}, since we have established \eqref{local_assumption_1} and \eqref{local_assumption_2} in all relevant cases and can therefore apply Proposition~\ref{prop:global_to_local}. Note that the assumption $q$ squarefree is only used in the treatment of supercuspidal representations. Indeed we crucially use that if $q$ is squarefree only depth-zero supercuspidal representations are relevant. On the other hand the argument reducing to the case of supercuspidal representations is completely general.

\subsection{The final results}
    
Dropping everything but the cuspidal part from Proposition \ref{density-kuz} and combining the resulting inequality with Theorem \ref{thm15} and \eqref{vqnq}, we obtain an unweighted version of Proposition \ref{density-kuz} for the cuspidal spectrum. At this point there is no difference between the groups $\Gamma(q)$ and $\Gamma(q)^{\natural}$.
\begin{cor}\label{cor-kuz} There exists an absolute constant $K > 0$ with the following property. Let $M, T > 1$, $q$ squarefree and suppose that $T \leq M^{-K} q^{n+1}$. Fix a place $v$ of $\Bbb{Q}$.  If $v = p$ is finite, assume that $p \nmid q$. 
Then  
$$ \sum_{\substack{\varpi  \in {\rm ONB}(L^2_{{\rm cusp}}(X(q))\\ \| \mu_{\varpi} \| \leq M}} T^{2 \sigma_{\varpi,v}}    \ll_{v, \varepsilon} M^{K} q^{\varepsilon}  \mathcal{V}_q,$$ where ${\rm ONB}$ stands for any orthonormal basis of eigenforms for the ring of invariant differential operators and the Hecke algebra at $v$ (if $v$ is finite).
\end{cor}   
   
Theorem \ref{thm1} is now an easy consequence.  For $T = M^{-K} q^{n+1}$   we have (the first inequality is a version of Rankin's trick) 
$$\mathcal{N}(\sigma, \mathcal{F}) \leq \sum_{\substack{\varpi  \in {\rm ONB}(L^2_{{\rm cusp}}(X(q)))\\ \| \mu_{\varpi} \| \leq M}} T^{2 \sigma_{\varpi, v} - 2\sigma}   \ll M^K  \mathcal{V}_q q^{-2\sigma(n+1)+\varepsilon},$$ and Theorem \ref{thm1} follows from \eqref{vqnq}. 

\begin{remark}\label{rem:aut_rep_pers}
This directly yields the following density estimate on the level of cuspidal automorphic representations
\begin{equation}
	\sum_{\substack{\pi\mid L^2_{\rm cusp}(X(q))\\ \Vert\mu_{\pi_{\infty}}\Vert\leq M,  \sigma_{\pi_{v}}>\sigma}}\dim_{\Bbb{C}}(\pi^{K(q)}) \ll M^K\cdot q^{1+\epsilon}\cdot\mathcal{V}_q^{1-\frac{2\sigma}{n-1}}.\nonumber
\end{equation}
Note that the additional factor of $q$ is natural, since twisting $\pi$ by a character with conductor $q$ changes the isomorphism class of $\pi$ but leaves $\mu_{\pi_{\infty}}$ and $\sigma_{\pi_v}$ invariant. For comparison one can look at the main term in the Weyl law for principal congrunce subgroup given in \cite[Theorem~0.2]{Mu}.

The weights $\dim_{\Bbb{C}}(\pi^{K(q)})$ can be removed with a little more work. Indeed, we would expect that for cuspidal representations one has $q^{\frac{n(n-1)}{2}-\epsilon}\ll \dim_{\Bbb{C}}(\pi^{K(q)}) \ll q^{\frac{n(n-1)}{2}+\epsilon}.$ This follows from the discussion below Proposition~\ref{prop:global_to_local} and the fact that cuspidal representations are generic. More precisely, assuming \eqref{local_assumption_2} is sharp for all generic representations, one obtains the weight-free density estimate
\begin{displaymath}
\begin{split}	\sharp \{\pi\subset & L_{\rm cusp}^2({\rm GL}_n(\Bbb{Q})\backslash {\rm GL}_n(\Bbb{A})^1)\colon \pi^{K(q)}\neq\{0\},\,  \Vert\mu_{\pi_{\infty}}\Vert\leq M,\, \sigma_{\pi_{v}}>\sigma\}\\
	&\ll M^K[q^{\frac{n(n+1)}{2}}]^{1-\frac{4\sigma}{n}+\epsilon}.%\nonumber 
\end{split}
\end{displaymath}
\end{remark}
   
\section{The non-cuspidal spectrum}\label{sec-eis}
   
For our applications, we need an upgrade of Corollary \ref{cor-kuz} which includes Eisenstein series and the residual spectrum. The aim of this section is a proof of the following bound. 
   
\begin{theorem}\label{eisen} There exists an absolute constant $K > 0$ with the following property. Let $M, T > 1$, $q$ squarefree,  and suppose that $T \leq M^{-K} q^{n+1}$. Fix a place $v$ of $\Bbb{Q}$. If $v = p$ is finite, assume that $p \nmid q$. 
Then $$ \underset{\| \mu_{\varpi} \| \leq M}{\int_{\Gamma(q)}}  T^{2 \sigma_{\varpi}(v)}  \dd\varpi  \ll_{v, \varepsilon} M^{K} q^{\varepsilon}  \mathcal{V}_q .$$   
\end{theorem}   

The proof of this is an inductive argument using Corollary~\ref{cor-kuz}, Langlands' spectral decomposition of $L^2({\rm GL}_n(\Bbb{Q})\backslash {\rm GL}_n(\Bbb{A})^1)$ and the description of the residual spectrum for ${\rm GL}_n$ due to M\oe glin and Waldspurger. A  concise summary of the relevant theory can be found in \cite[Chapter~10]{Ge} %\marginpar{\tiny{Can we give a more precise reference in [GeH], e.g. a chapter?}}
or in \cite{Ar}. Full proofs are given in \cite{MW1} and \cite{MW2}.\\

\textbf{Proof.} 
Let us denote the $n$-dependence explicitly by writing  $$X(q) =  X^{(n)}(q) = \Gamma(q)\backslash {\rm GL}_n(\Bbb{R})/{\rm SO}_n(\Bbb{R}).$$  The basic strategy is to use Langlands' spectral decomposition for $L^2({\rm GL}_n(\Bbb{Q})\backslash {\rm GL}_n(\Bbb{A})^1)$ to decompose $$L^2(X^{(n)}(q)) = \bigoplus_{P/\sim } L^2_P(X^{(n)}(q)),$$ where $P$ runs through all standard parabolic subgroups of ${\rm GL}_n$ up to association, including  $P={\rm GL}_n$ itself. % and we write $X^{(n)}_{\rm disc}(q)= X^{(n)}_{{\rm GL}_n}(q)$.
Furthermore each standard parabolic subgroup of ${\rm GL}_n$ is determined by a partition $n=n_1+\ldots+n_k$, and we can describe each of the spaces $L^2_P(X^{(n)}(q))$ in terms of $$L^2_{\rm disc}(X^{(n_i)}(q)) = L^2_{{\rm GL}_{n_i}}(X^{(n_i)}(q))$$ with $1\leq i\leq k$. The nature of the spaces $L^2_P(X^{(n)}(q))$ will be described  below in our proof, but we will use the notation $\int_{\Gamma(q)}^{(P)}$ to denote the $P$-part of the spectral decomposition of $L^2(X(q))$. Of course it is enough to prove $$\underset{\| \mu_{\varpi} \| \leq M}{\int_{\Gamma(q)}^{(P)}}  T^{2 \sigma_{\varpi,v}}  \dd\varpi  \ll_{v, \varepsilon} M^{K} q^{\varepsilon}  \mathcal{V}_q,$$ for each standard parabolic subgroup $P$. 

We start by considering the contribution from $P={\rm GL}_n$, i.e.\   the discrete spectrum. %, and we write $X^{(n)}_{\rm disc}(q)= X^{(n)}_{{\rm GL}_n}(q)$. 
 To treat this piece we recall the parametrisation of  the discrete spectrum given by M\oe glin and Waldspurger. Given $d\mid n$ and a cuspidal automorphic representation $\pi$ of ${\rm GL}_{\frac{n}{d}}(\Bbb{A})^1$, we define the Speh representation ${\rm Speh}(\pi,d)$ of ${\rm GL}_n(\Bbb{A})^1$ as the unique irreducible subrepresentation of 
\begin{equation}
	{\rm Ind}_{P_d(\Bbb{A})}^{{\rm GL}_n(\Bbb{A})} (\vert \cdot\vert_{\Bbb{A}}^{-\frac{d-1}{2}}\pi \otimes\ldots \vert \cdot\vert_{\Bbb{A}}^{\frac{d-1}{2}}\pi), \nonumber
\end{equation}
where $P_d$ is the standard parabolic subgroup associated to the partition $\frac{n}{d}+\ldots +\frac{n}{d} = n$. It turns out that for each Speh representation (i.e.\ each tuple $(\pi,d)$) there is a unique irreducible subrepresentation of $L^2({\rm GL}_n(\Bbb{Q})\backslash {\rm GL}_n(\Bbb{A})^1)$ which is isomorphic to ${\rm Speh}(\pi,d)$. For convenience we will denote this automorphic representation also by ${\rm Speh}(\pi,d)$. One obtains the spectral decomposition
\begin{equation}
	L^2_{\rm disc}({\rm GL}_n(\Bbb{Q})\backslash {\rm GL}_n(\Bbb{A})^1) = \bigoplus_{d\mid n}\bigoplus_{\pi  \subset L^2_{\rm cusp}({\rm GL}_{\frac{n}{d}}(\Bbb{Q})\backslash {\rm GL}_{\frac{n}{d}}(\Bbb{A})^1)} {\rm Speh}(\pi,d) \nonumber
\end{equation}
of the discrete spectrum of $L^2({\rm GL}_n(\Bbb{Q})\backslash {\rm GL}_n(\Bbb{A})^1)$. By de-adelisation we obtain the decomposition
\begin{equation}
	L^2_{\rm disc}(X^{(n)}(q)) = \bigoplus_{d\mid n}\sideset{}{'}\bigoplus_{\pi\mid L^2_{\rm cusp}(X^{n/d}(q)) } V_{{\rm Speh}(\pi,d)}.\nonumber
\end{equation}
Note that the piece given by $d=1$ is nothing but the cuspidal part of the spectrum. Furthermore we have $\dim V_{{\rm Speh}(\pi,d)} = \dim {\rm Speh}(\pi,d)^{K(q)}$ and the factorization ${\rm Speh}(\pi,d) = \bigotimes_v {\rm Speh}(\pi_v,d)$, where ${\rm Speh}(\pi_v,d)$ is the Langlands quotient of $\xi_v={\rm Ind}_{P_d(\Bbb{Q}_v)}^{{\rm GL}_n(\Bbb{Q}_v)} (\vert \cdot\vert_v^{\frac{d-1}{2}}\pi_v \otimes\ldots \vert \cdot\vert_v^{-\frac{d-1}{2}}\pi_v)$. Applying Lemma~\ref{lm:dimensions_ps} locally we obtain 
\begin{equation}
	\dim V_{{\rm Speh}(\pi,d)} \leq \dim(V_{\pi})^d q^{\frac{n(n-1)}{2}-\frac{n(n/d-1)}{2}}. \label{upper_boubnd_speh}
\end{equation}
Probably \eqref{upper_boubnd_speh} is far from optimal but for what follows it is sufficient. Finally, observe that $$\sigma_{{\rm Speh}(\pi,d)}(v) = \sigma_{\pi}(v) + \frac{d-1}{2}.$$ Note that for a Dirichlet character $\chi$, the representation ${\rm Speh}(\omega_{\chi},n)$ is precisely the one-dimensional representation given by $\omega_{\chi}\circ\det$, and we have $\sigma_{{\rm Speh}(\omega_{{\chi_0},v},n)} = (n-1)/2$. 
 We compute (recall the notation \eqref{sum-prime}) 
\begin{align}
	\underset{\| \mu_{\varpi} \| \leq M}{\int_{\Gamma(q)}^{({\rm GL}_n)}}  &T^{2 \sigma_{\varpi,v}}  \dd\varpi = T^{2 \sigma_{{\rm Speh}(\omega_{{\chi_0},v},n)}}+\sum_{\substack{d\mid n \\ d\neq n}}\sideset{}{'}\sum_{\substack{\pi\mid L^2_{\rm cusp}(X^{(n/d)}(q))\\ \| \mu_{\pi_{\infty}} \| \leq M} }\dim (V_{{\rm Speh}(\pi,d)})T^{2 \sigma_{{\rm Speh}(\pi_v,d)}} \nonumber\\
	&\leq T^{n-1}+\sum_{\substack{d\mid n \\ d\neq n}}T^{d-1}q^{\frac{n(n-1)}{2}-\frac{n(n/d-1)}{2}}\sideset{}{'}\sum_{\substack{\pi\mid L^2_{\rm cusp}(X^{(n/d)}(q)) \\ \|\mu_{\pi_{\infty}} \|\leq M}}\dim( V_{\pi})^d\cdot T^{2 \sigma_{\pi_v}} \nonumber \\
	&\leq T^{n-1}+\sum_{\substack{d\mid n \\ d\neq n}}T^{d-1}q^{\frac{n(n-1)}{2}-\frac{n(n/d-1)}{2}}\Bigg(\sideset{}{'}\sum_{\substack{\pi\mid L^2_{\rm cusp}(X^{(n/d)}(q)) \\ \| \mu_{\pi_{\infty}} \| \leq M}}\dim( V_{\pi})\cdot T^{\frac{2}{d} \sigma_{\pi_v}}\Bigg)^d. \nonumber
\end{align}
We recall our assumption $T\leq M^{-K} q^{n+1}$ where $K = K_n$ is a sufficiently large constant. We apply Corollary~\ref{cor-kuz} to each term in the last parenthesis for  degree $n/d$ and $T^{1/d}$ in place of $T$. As long as $K_n \geq d K_{n/d}$, this application of Corollary~\ref{cor-kuz}  is admissible. By making $K_n$ larger if necessary, the assumption $K_n \geq d K_{n/d}$ can always be ensured. We thus obtain
 \begin{align}
	\underset{\| \mu_{\varpi} \| \leq M}{\int_{\Gamma(q)}^{({\rm GL}_n)}}  T^{2 \sigma_{\varpi,v}}  \dd\varpi &\ll  q^{n^2-1}+\sum_{\substack{d\mid n \\ d\neq n}}  M^{dK_{\frac{n}{d}}-(d-1)K_n}q^{(d-1)(n+1)+\frac{n(n-1)}{2}-\frac{n(n/d-1)}{2}+\frac{n^2}{d}-d + \varepsilon} \nonumber\\
	&\ll M^{K_n} q^{\varepsilon} \mathcal{V}_q \cdot \Big(1+ \sum_{\substack{d\mid n,d\neq n}} q^{\alpha_d}\Big), \quad \alpha_d = \frac{(d-1)n(2d-n)}{2d}.\nonumber
\end{align}
%for 
%\begin{equation}
%	\alpha_d = \frac{(d-1)n(2d-n)}{2d}. \nonumber
%\end{equation}
Note that for $1 \leq d \leq   n/2$  we have $\alpha_d\leq 0$. This completes the contribution the discrete spectrum (i.e.\ $P={\rm GL}_n$). To summarise, for $T\leq M^{-K} q^{n+1}$ we have seen that
\begin{equation}
	\sideset{}{'}\sum_{\substack{\pi\mid L^2_{\rm disc}(X^{(n)}(q)) \\ \| \mu_{\pi_{\infty}} \| \leq M}}\dim(V_{\pi}) T^{2 \sigma_{\pi_v}} = \underset{\| \mu_{\varpi} \| \leq M}{\int_{\Gamma(q)}^{({\rm GL}_n)}}  T^{2 \sigma_{\varpi,v}}  \dd\varpi \ll M^{K_n}q^{n^2 - 1 + \varepsilon}. \label{bound_disc}
\end{equation}
 We will need this bound for $n'<n$ in the process of treating the Eisenstein contribution.

Before we treat the remaining contributions we study a certain globally induced representation. %\marginpar{\tiny{Most likely this can be shortened considerably.}} 
Given a standard parabolic subgroup $P=MN$ let ${Z}_M^+$ denote the intersection $\tilde{T}(\Bbb{R})$ and the center of $M(\Bbb{A})$. We write $R_{M,{\rm disc}}$ for the right regular representation on $L^2_{\rm disc}({Z}_M^+M(\Bbb{Q})\backslash M(\Bbb{A}))$. We define the Hilbert space
\begin{multline}
	\mathcal{H}_P =\{ \phi\colon {Z}^+_MN(\Bbb{A})M(\Bbb{Q})\backslash {\rm GL}_n(\Bbb{A})\to \Bbb{C} \text{ measurable}\colon\\
 [m\mapsto\phi(mx)]\in L^2_{\rm disc}({Z}^+_MM(\Bbb{Q})\backslash M(\Bbb{A})), \\ \Vert \phi\Vert^2 = \int_K \int_{{Z}^+_MM(\Bbb{Q})\backslash M(\Bbb{A})} \vert \phi(mk)\vert^2\dd m\, \dd k<\infty \}.\nonumber
\end{multline}
For $\lambda\in \mathfrak{a}_{P,\Bbb{C}}^{\star}$ let $\mathcal{I}_P(\lambda)$ act on $\mathcal{H}_P$ by
\begin{equation}
	[\mathcal{I}_P(\lambda,g)\phi](x) = \phi(xg)e^{(\lambda+\rho_P)(H_P(xg))}e^{-(\lambda+\rho_P)(H_P(x))},\nonumber
\end{equation}
where $\rho_P\in\mathfrak{a}_{P,\Bbb{C}}^{\star}$ is the half sum of positive roots of a maximal torus of ${\rm GL}_n$ contained in $P$ and $H_P\colon {\rm GL}_n(\Bbb{A})\to \mathfrak{a}_{P,\Bbb{C}}$ satisfies $\delta_P(p) = e^{2\rho_P(H_P(p))}$ and $H_P(pk)=H_P(p)$ for all $p\in P(\Bbb{A})$ and all $k\in K$. Write $\mathcal{H}^{\circ}_P$ for the subspace of $K$-finite elements. For $\lambda=0$, this is nothing but the representation obtained by lifting $R_{M,{\rm disc}}$ to a representation of $P=MN$ and parabolically inducing to ${\rm GL}_n$ from there. We have the decomposition
\begin{equation}
	R_{M,\rm{disc}} = \bigoplus_{\tau_1 \subset L^2_{\rm disc}({\rm GL}_{n_1} (\Bbb{Q})\backslash{\rm GL}_{n_1}(\Bbb{A})^1)} \ldots \bigoplus_{\tau_k \subset L^2_{\rm disc}({\rm GL}_{n_k} (\Bbb{Q})\backslash{\rm GL}_{n_k}(\Bbb{A})^1)}\tau_1\otimes \ldots\otimes \tau_k.\nonumber
\end{equation}
To shorten notation we write $\boldsymbol{\tau}=(\tau_1,\ldots,\tau_k)$ for a tuple of irreducible automorphic representations as in the display and we set $\sigma(\boldsymbol{\tau}) = \tau_1\otimes \ldots\otimes \tau_k$. The latter defines an irreducible automorphic representation of $Z_M^+\backslash M(\Bbb{A})$, which factors $\sigma(\boldsymbol{\tau}) = \otimes _v\sigma(\boldsymbol{\tau})_v$. We define
\begin{equation}
	\sigma(\boldsymbol{\tau})_{v,\lambda}(m) = \sigma(\boldsymbol{\tau})_{v}(m) e^{\lambda(H_P(m))} \text{ for }\lambda\in\mathfrak{a}_{P,\Bbb{C}}\text { and } m\in M(\Bbb{Q}_v).\nonumber
\end{equation}
We then have
\begin{equation}
	\mathcal{I}_P(\lambda) =\bigoplus_{\boldsymbol{\tau}}{\rm Ind}_P^{{\rm GL}_n}(\sigma(\boldsymbol{\tau})_{\lambda}) = \bigoplus_{\boldsymbol{\tau}} \bigotimes_v {\rm Ind}_P^{{\rm GL}_n}(\sigma(\boldsymbol{\tau})_{v,\lambda}).\nonumber
\end{equation}
By Lemma~\ref{lm:dimensions_ps} we get
\begin{equation}
\begin{split}
	\dim {\rm Ind}_P^{{\rm GL}_n}&(\sigma(\boldsymbol{\tau})_{\lambda})^{K(q)} = \Big(\prod_{p\mid q}\# P(\Bbb{Q}_p)\backslash {\rm GL}_n(\Bbb{Q}_p)/K_p(q) \Big) \prod_{i=1}^k \dim(V_{\tau_i})\\
	& = q^{\frac{n(n-1)}{2}-\sum_{i=1}^k \frac{n_i(n_i-1)}{2}}\prod_{i=1}^k \dim(V_{\tau_i}) = q^{\frac{1}{2}\sum_{1 \leq i \not= j \leq n} n_in_j}\prod_{i=1}^k \dim(V_{\tau_i}).\nonumber
	\end{split}
\end{equation}
Furthermore, if $\lambda\in i\mathfrak{a}_{P}$, then
\begin{equation}
	\sigma_{{\rm Ind}_P^{{\rm GL}_n}(\sigma(\boldsymbol{\tau})_{\lambda})}(v) = \max_{i=1,\ldots,k}\sigma_{\tau_i}(v).\nonumber
\end{equation}

Langlands' theory of Eisenstein series yields the orthogonal decomposition 
\begin{equation}
	L^2({\rm GL}_n(\Bbb{Q})\backslash {\rm GL}_n(\Bbb{A})^1) = \bigoplus_{P/\sim} L^2_P({\rm GL}_n(\Bbb{Q})\backslash {\rm GL}_n(\Bbb{A})^1),
\end{equation}
where 
\begin{equation}
	L^2_P({\rm GL}_n(\Bbb{Q})\backslash {\rm GL}_n(\Bbb{A})^1) \cong \int_{i\mathfrak{a}_P^+} I_P(\lambda)\dd \lambda.
\end{equation}
The intertwiner can be given more explicitly, but this is not necessary for our purposes. The space $X_P^{(n)}(q)$ is precisely the space arising by de-adelising $L^2_P({\rm GL}_n(\Bbb{Q})\backslash {\rm GL}_n(\Bbb{A})^1)$. With this at hand we can write
\begin{displaymath}
\begin{split}
	\underset{\| \mu_{\varpi} \| \leq M}{\int_{\Gamma(q)}^{(P)}}  T^{2 \sigma_{\varpi,v}}  \dd\varpi  =
	& \sum_{\tau_1\mid L^2_{\rm disc}(X^{(n_1)}(q))}\ldots \sideset{}{'}\sum_{\tau_k\mid L^2_{\rm disc}(X^{(n_k)}(q))}\dim {\rm Ind}_P^{{\rm GL}_n}(\sigma(\boldsymbol{\tau})_{\lambda})^{K(q)}\\
	 &\cdot  T^{2\sigma_{{\rm Ind}_P^{{\rm GL}_n}(\sigma(\boldsymbol{\tau})_{v,\lambda})}} \int_{\substack{\lambda\in i\mathfrak{a}_P^+\\ \Vert \mu_{{\rm Ind}_P^{{\rm GL}_n}(\sigma(\boldsymbol{\tau})_{\infty,\lambda})}\Vert \leq M}}\dd\lambda. 
	 \end{split}
\end{displaymath}
Since the integral can be treated trivially, we  proceed by handling all the $\tau_i$-sums. Note that even though adelisation works only up to twist (modulo $q$) we can use this twist redundancy only to reduce one of the smaller representations, say $\tau_k$, modulo twists. Note that it also suffices to treat the case $\max_{i=1,\ldots,k}\sigma_{\tau_i}(v)=k$. For $i=1,\ldots,k-1$ we treat the $\tau_i$-sum only using a suitable Weyl law (which is a standard application of the trace formula or alternatively a consequence of \eqref{bound_disc} with $T=1$):  %\marginpar{\tiny{We need some bound on the archimedean parameter in the sums.\\ We can use $\leq M$ as well because the spectral parameter in the induced representation can be taken in $\lambda\in i\mathfrak{a}_P^+$. So instead of bounding the spectral parameter of the induced representation by $M$ we take $\Vert \lambda\Vert\leq M$ and $\Vert \mu_{\tau_{i,\infty}}\Vert\leq M$ for all $1\leq i\leq k$. This makes things only larger.}}
%Thus we will estimate
\begin{equation}
	\sum_{\substack{\tau_i\mid L^2_{\rm disc}(X^{(n_i)}(q)) \\ \Vert \mu_{\tau_{i,\infty}}\Vert\leq M}} \dim V_{\tau_i} \ll q\sideset{}{'}\sum_{\substack{\tau_i\mid L^2_{\rm disc}(X^{(n_i)}(q)) \\ \Vert \mu_{\tau_{i,\infty}}\Vert\leq M}}\dim V_{\tau_i} \ll M^{K_{n_i}}q^{n_i^2+\varepsilon}.\nonumber
\end{equation}
We can also estimate
\begin{equation}
	T^{2\sigma_{{\rm Ind}_P^{{\rm GL}_n}(\sigma(\boldsymbol{\tau})_{v,\lambda})}} \leq (T^{\frac{n_k}{n}})^{2\sigma_{\tau_{k,v}}} \cdot T^{\frac{(n-n_k)(n_k-1)}{n}},\nonumber
\end{equation}
where we use the trivial bound $\sigma_{\tau_{k,v}}\leq \frac{n_k-1}{2}$. With this at hand we obtain  
\begin{align}
\underset{\| \mu_{\varpi} \| \leq M}{\int_{\Gamma(q)}^{(P)}}  T^{2 \sigma_{\varpi,v}}  \dd\varpi  &\ll M^{\underset{i \leq k-1}{\sum}%_{i=1}^{k-1}
	K_{n_i}}q^{ \underset{ i \leq k-1}{\sum}%_{i=1}^{k-1}
	n_i^2+\frac{1}{2} \underset{1\leq i\neq j\leq k}{\sum} %_{1\leq i\neq j\leq k}
	n_in_j+\varepsilon } 
	T^{\frac{(n-n_k)(n_k-1)}{n}}\\[-0.5cm]
	&\quad\quad\quad\quad \cdot\sideset{}{'}\sum_{\substack{\tau_k\mid L^2_{\rm disc}(X^{(n_k)}(q)),\\ \Vert \mu_{\tau_{k,\infty}}\Vert\leq M  }}\dim V_{\tau_k}(T^{\frac{n_k}{n}})^{2\sigma_{\tau_{k,v}}}\int_{\substack{\lambda\in i\mathfrak{a}_P^+\\ \Vert \lambda\Vert \leq M}}\dd\lambda \nonumber \\
	&   \ll M^Kq^{\underset{i \leq k}{\sum}n_i^2+\frac{1}{2}\underset{1\leq i\neq j\leq k}{ \sum}n_in_j-1+\varepsilon } T^{\frac{(n-n_k)(n_k-1)}{n}}\\
	& \ll  M^Kq^{\varepsilon}\mathcal{V}_q q^{-\frac{1}{2}\underset{1\leq i\neq j\leq k}{\sum}n_in_j+\varepsilon } T^{\frac{(n-n_k)(n_k-1)}{n}}.\nonumber
\end{align}
Here  we have estimated the $\tau_k$-sum using \eqref{bound_disc}. Inserting the estimate $T\ll q^{n+1}$ gives us
\begin{equation}
	\underset{\| \mu_{\varpi} \| \leq M}{\int_{\Gamma(q)}^{(P)}}  T^{2 \sigma_{\varpi,v}}  \dd\varpi \ll M^Kq^{\alpha_P+\varepsilon}\mathcal{V}_q, \label{eq:better_on_para}
\end{equation}
for $\alpha_P = -\frac{1}{2}\sum_{1\leq i\neq j\leq k}n_in_j+(n+1)\frac{(n-n_k)(n_k-1)}{n}$. We are done as soon as we can show that $\alpha_P\leq 0$. To see this we recall $n=n_1+\ldots+n_k$ and write %\marginpar{\tiny{I have tried to explain this better. The best case is somehow when $P=B$ is the Borel. The worst case is $P$ associated to $(1,n-1)$, then $\alpha_P=-\frac{2}{n}$. This seemed somehow reasonable to me.}}
\begin{align}
	\alpha_P %&= (n-n_k)(n_k-1)-\frac{1}{2}\sum_{1\leq i\neq j\leq k}n_in_j + \frac{(n-n_k)(n_k-1)}{n}\nonumber \\
	%&= n_k^2-n-n_k(n_k-1) -\sum_{1\leq i<j<k}n_in_j +\frac{(n-n_k)(n_k-1)}{n} \nonumber \\
	&=  (n-n_k)\frac{(n_k-1-n)}{n}  - \sum_{1\leq i<j<k}n_in_j  \leq -\sum_{1\leq i<j<k}n_in_j \leq 0.\nonumber
\end{align}
%\end{proof}

\begin{remark}\label{rem_more_precise}
Note that the main contribution in our density estimate above comes from the discrete (even the cuspidal) part of the spectrum. The continuous contribution coming from some standard parabolic subgroup $P$ has been estimated quite wastefully in \eqref{eq:better_on_para}. Indeed our proof shows that $$\underset{\| \mu_{\varpi} \| \leq M}{\int_{\Gamma(q)}^{(P)}}  T^{2 \sigma_{\varpi,v}}  \dd\varpi \ll M^Kq^{\alpha_P+\varepsilon}\mathcal{V}_q,$$ where $\alpha_P<0$ for all proper parabolic subgroups $P$. One sees that $\alpha_P\leq -\frac{2}{n}$, the worst case being a maximal parabolic subgroup where this bound is attained. 
\end{remark}

We close this section by  recalling that congruence subgroups possess a spectral gap. For a cuspidal representation $\pi$ we have $\| \Re \mu_{\pi} \| \leq 1/2$ (in fact even $\leq \frac{1}{2} - \frac{1}{n^2 - 1}$). If we include the continuous and residual spectrum, then for $n \geq 3$ the worst case appears for the Eisenstein series associated to the $(n-1, 1)$ parabolic with the trivial representation on the $(n-1)$-block, in which case %$\| \Re \mu_{\pi} \| = (n-2)/2$. 
%each non-trivial $\pi$ (including the continuous and residual spectrum) occurring in the decomposition \eqref{spec} we have
\begin{equation}\label{gap}
  \| \Re \mu_{\pi} \| \leq \mu^{\ast} = \frac{n-2}{2}. 
\end{equation}

 \section{Applications}\label{sec8}
 
 In this final section we prove Theorems \ref{thm2} and \ref{thm3}. We start with an auxiliary result. 
 
 \subsection{Bounding eigenfunctions}

For each $\bar{g} \in {\rm SL}_n(\Bbb{Z}/q\Bbb{Z})$ choose $g \in {\rm SL}_n(\Bbb{Z})$ such that $g$ projects onto $\bar{g}$ modulo $q$. These matrices form a set of representatives for $\Gamma(1)/\Gamma(q)$. Note in particular that $\Gamma(q)$ is normal in $\Gamma(1)$. 

Besides $L^2$-eigenfunctions it is also necessary to deal with certain Eisenstein series that appear in the spectrum of $X_q$. We start by recalling their (adelic) construction. We use the same notation as in the proof of Theorem~\ref{eisen}. Let $P=MN$ be a standard parabolic subgroup of ${\rm GL}_n$. Given $x\in{\rm GL}_n(\Bbb{A})$, $\lambda\in \mathfrak{a}_{P,\Bbb{C}}^{\star}$ and $\phi\in \mathcal{H}_P^{\circ}$ we define the Eisenstein series by
\begin{equation}
	E(x,\phi,\lambda) = \sum_{\gamma\in P(\Bbb{Q})\backslash G(\Bbb{Q})}\phi(\gamma x)e^{(\lambda+\rho_P)(H_P(\gamma x))}.\nonumber
\end{equation}
Note that this series converges absolutely for $\Re(\lambda)\in \rho_P+\mathfrak{a}_P^+$ and is understood by analytic continuation otherwise. Recall the action of $\mathcal{I}_P(\lambda,\cdot)$ on $\mathcal{H}_P$ and observe that
\begin{equation}
	E(x,\mathcal{I}_P(\lambda,g)\phi,\lambda) = E(xg,\phi,\lambda).\nonumber
\end{equation}
In particular, the de-adelisation of $E(\cdot,\phi,\lambda)$ contributes to the spectral decomposition of $X_q$ precisely when $\mathcal{I}_P(\lambda,k)\phi = \phi$ for all $k\in K(q)$. This property is independent of $\lambda$ and we write $\mathcal{H}_P^{K(q)}$ for the space of $\phi\in\mathcal{H}_P$ with this property. Recall that we can decompose
\begin{equation}
	\mathcal{H}_P^{\circ} = \bigoplus_{\sigma\in L^2_{\rm disc}({ Z}_M^+M(\Bbb{Q})\backslash M(\Bbb{A}))}	\mathcal{H}_P(\sigma)^{\circ} \nonumber
\end{equation}
and identify $\mathcal{H}_P(\sigma)^{\circ}$ with the induced representation  $\text{Ind}_P^{{\rm GL}_n}(\sigma)$.

Finally we recall the definition of the truncation operator $\Lambda^T$ from \cite[Section~1]{Ar1}, for some fixed $T\in \mathfrak{a}^+$ with $\alpha(T)\gg 1$ for all simple roots $\alpha$ of ${\rm GL}_n$. The latter condition is referred to as ``sufficiently regular''. Below we will need the following assumption which is a precise version of \eqref{11} and discussed at the end of Subsection \ref{appl}.

\begin{hyp}\label{hyp_A}
For all standard parabolic subgroups $P=MN$ of ${\rm GL}_n(\Bbb{A})$, all irreducible automorphic representations $\sigma\subset L^2_{\rm disc}({Z}_M^+M(\Bbb{Q})\backslash M(\Bbb{A}))$ and all $\phi\in\mathcal{H}_P(\sigma)^{K(q)}$ with $\|\phi\|=1$ we have
\begin{equation}
	\| \Lambda^T E(.,\phi,\lambda)\|^2 = \int_{{\rm GL}_n(\Bbb{Q})\backslash {\rm GL}_n(\Bbb{A})^1}| \Lambda^T E(x,\phi,\lambda)|^2\dd x \ll_{T, \varepsilon} q^{\varepsilon}\Vert \mu_{\sigma_{\infty,\lambda}}\Vert^K.\nonumber
\end{equation}
\end{hyp}

We proceed to prove the following lemma.
 
\begin{lemma}\label{local} As in \eqref{spec} let $\varpi$ be a member of an ONB of the spectral decomposition $X_q$ with spectral parameter $\mu$. If $\varpi$ is not discrete, assume that Hypothesis~\ref{hyp_A} holds. Then $$\sum_{\bar{g} \in  {\rm SL}_n(\Bbb{Z}/q\Bbb{Z})}|\varpi(g)|^2 \ll q^{\varepsilon}   \|  \mu \|^K.$$
\end{lemma}
 
\textbf{Proof.} This is a standard ``local'' argument based on properties of spherical functions; see e.g.\ \cite[pp.\ 11-12]{SaMo}.  Fix $\delta > 0$ and let $B_{\delta}(g)$ denote the ball about $g$ of radius $\delta$. For $k \in g^{-1} {\rm SO}_n(\Bbb{R}) g$ we have $$\int_{B_{\delta}(g)} \varpi(x) \overline{\phi_{\mu}(g^{-1} x)} \dd x  = \int_{B_{\delta}(g)} \varpi(kx) \overline{\phi_{\mu}(g^{-1} x)} \dd x$$ by changing variables and using the left-${\rm SO}_n(\Bbb{R})$-invariance of the spherical function $\phi_{\mu}$ defined in \eqref{spher-func}. Integrating both sides over $g^{-1} {\rm SO}_n(\Bbb{R}) g$ and using the mean value property of $\phi_{\mu}$, we obtain $$\int_{B_{\delta}(g)} \varpi(x) \overline{\phi_{\mu}(g^{-1} x)} \dd x  =  \varpi(g) \int_{B_{\delta}(g)} |\phi_{\mu}(g^{-1} x)|^2 \dd x.$$ By the Cauchy-Schwarz inequality we obtain 
\begin{displaymath}
\begin{split}
   | \varpi(g) |^2 &\leq  \Big( \int_{B_{\delta}(g)} |\phi_{\mu}(g^{-1} x)|^2 \dd x \Big)^{-1}  \int_{B_{\delta}(g)} |\varpi(x) |^2 \dd x \\
   &=\Big( \int_{B_{\delta}(I_n)} |\phi_{\mu}(  x)|^2 \dd x \Big)^{-1}  \int_{B_{\delta}(g)} |\varpi(x) |^2 \dd x. 
   \end{split}
   \end{displaymath}
    For fixed $0 < \delta < 1$ we have from the uniform asymptotic behavior of the spherical function \cite[(5.5)]{Va} that $$|\phi_{\mu}(x) |^2  \gg \| \mu \|^{-n(n-1)}, \quad x\in B_{\delta}(I_n)$$ if $\| \mu \|$ is larger than some absolute constant. For bounded $\mu$, this remains true after possibly reducing $\delta$ by a compactness argument since $\phi_{\pi}(I_n) = 1$, so that $$| \varpi(g) |^2 \ll \| \mu \|^K  \int_{B_{\delta}(g)} |\varpi(x) |^2 \dd x.$$ After possibly reducing $\delta$ further, we may assume that the  union of the $B_{\delta}(g)$ for $\bar{g} \in {\rm SL}_n(\Bbb{Z}/q\Bbb{Z})$ is disjoint and   a compact portion $\Omega$ of $\Gamma(q) \backslash {\rm SL}_n(\Bbb{R})$. We conclude 
\begin{equation}\label{eq:loc_bound_disc}
	\sum_{\bar{g} \in  {\rm SL}_n(\Bbb{Z}/q\Bbb{Z})} | \varpi(g) |^2 \ll   \| \mu \|^K \int_{\Omega} |\varpi(x) |^2 \dd x \leq \| \mu \|^K
\end{equation}
if $\varpi$ is square integrable (i.e. contributes to the discrete spectrum of $X_q$). 
 
Turning towards the case of non-discrete $\varpi$ we assume (without loss of generality) that $\varpi(x) = K_q^{-1}\cdot E(x,\phi,\lambda)$, for a standard parabolic subgroup $P$, an irreducible automorphic representation $\sigma\subset L^2_{\rm disc}({ Z}^+_M(\Bbb{Q})\backslash M(\Bbb{A})$ and  $\phi\in \mathcal{H}_P(\sigma)^{K(q)}$ with $\Vert\phi\Vert=1$. Here $K_q$ is a constant with $K_q\asymp \mathcal{V}_q^{1/2}$ coming from the different normalization of the inner products of $L^2({\rm GL}_n(\Bbb{Q})\backslash {\rm GL}_n(\Bbb{A})^1)$ and $X_q$; see \eqref{strongapprox}.

We return to proving the desired pointwise bound. Our starting point is the first inequality in \eqref{eq:loc_bound_disc}, which remains true in this case. Indeed we have
\begin{equation}
	\sum_{\bar{g} \in  {\rm SL}_n(\Bbb{Z}/q\Bbb{Z})} | \varpi(g) |^2 \ll \| \mu \|^K\int_{\Omega} |\varpi(x) |^2 \dd x \ll \frac{\| \mu \|^K}{K_q^2}\int_{\Omega} |E(x_{\infty},\phi,\lambda) |^2 \dd x_{\infty}, \nonumber 
\end{equation}
where we indicate that the remaining integral sees only the archimedean place. Without loss of generality we can assume that $\phi$ transforms under $K_1(q)$ with respect to some tuple of character $\bm{\zeta}$ as in Section~\ref{sec:adelising}. In particular $|E(x_{\infty},\phi,\lambda)|$ is invariant under the open compact subgroup $K(q)\subset K'(q)\subset K_1(q)$ where $K'(q)=\text{diag}(\widehat{\Bbb{Z}}^{\times},1,\ldots,1)K(q)$. We obtain
\begin{equation}
	\frac{1}{K_q^2}\int_{\Omega} |E(x_{\infty},\phi,\lambda) |^2 \dd x_{\infty} = \text{Vol}(K'(q))^{-1}K_q^{-2}\int_{\Omega\times K'(q)} |E(x,\phi,\lambda) |^2 \dd x_{\infty}.
\end{equation}
Note that $$\text{Vol}(K'(q))^{-1}K_q^{-2} \asymp \frac{[K_{\rm fin}\colon K'(q)]}{\mathcal{V}_q}\asymp 1.$$ Put $\Omega_{\Bbb{A}} = B_{\delta}(I_n)\times K_{\rm fin}$, recall $\Omega=\bigsqcup_{\overline{g}\in{\rm SL}_n(\Bbb{Z}/q\Bbb{Z})}gB_{\delta}(I_n)$ and observe $K_{\rm fin}= \bigsqcup_{\overline{g}\in{\rm SL}_n(\Bbb{Z}/q\Bbb{Z})}g^{-1}K'(q)$. By some changes of variables we obtain
\begin{equation}
	\sum_{\bar{g} \in  {\rm SL}_n(\Bbb{Z}/q\Bbb{Z})} | \varpi(g) |^2 \ll \frac{\| \mu \|^K}{K_q^2}\int_{\Omega} |E(x_{\infty},\phi,\lambda) |^2 \dd x_{\infty} \ll \| \mu \|^K\int_{\Omega_{\Bbb{A}}} |E(x,\phi,\lambda) |^2 \dd x. \nonumber 
\end{equation}
The point of this maneuvre is that we have removed any $q$-dependence from the domain of integration $\Omega_{\Bbb{A}}$.  Instead of extending the integral to all ${\rm GL}_n(\Bbb{Q})\backslash {\rm GL}_n(\Bbb{A})^1$, which is not possible since $E(\cdot,\phi,\lambda)$ is not square integrable, we apply the truncation operator $\Lambda^T$ first. Indeed, without loss of generality we can assume that $\Omega_{\Bbb{A}}$ is contained in a Siegel set $\mathfrak{S}$ so that after making $\delta$ smaller or $T$ bigger \cite[Lemma~2.6]{La} implies 
\begin{equation}
	\int_{\Omega_{\Bbb{A}}} |E(x,\phi,\lambda) |^2 \dd x = \int_{\Omega_{\Bbb{A}}} |\Lambda^T E(x,\phi,\lambda) |^2 \dd x.\nonumber
\end{equation}
At this point we can extend the integral and apply Hypothesis~\ref{hyp_A} to get
\begin{equation}
	\sum_{\bar{g} \in  {\rm SL}_n(\Bbb{Z}/q\Bbb{Z})} | \varpi(g) |^2 \ll \| \mu \|^K\int_{{\rm GL}_n(\Bbb{Q})\backslash {\rm GL}_n(\Bbb{A})^1} |\Lambda^TE(x,\phi,\lambda) |^2 \dd x \ll q^{\epsilon}\| \mu \|^K. \nonumber 
\end{equation}
as desired.

\subsection{Optimal lifting}
 
We are now ready to prove Theorem \ref{thm3}, following the strategy of \cite{Sa2}.  Let $T \geq 1$ be a parameter and let $f : {\rm SL}_n(\Bbb{R}) \rightarrow \Bbb{C} $ be as in \eqref{deff}. For $z, w\in \Gamma(q) \backslash {\rm SL}_n(\Bbb{R})$ and $\bar{g} \in {\rm SL}_n(\Bbb{Z}/q\Bbb{Z})$ define
$$F(z, w; \bar{g})  = \sum_{\gamma \in \Gamma(q)} f(w^{-1} g\gamma z)$$% = \sum_{G \equiv g\, (\text{mod } q)} k(w^{-1} $$
and 
$$W(z) = \sum_{\bar{g} \in  {\rm SL}_n(\Bbb{Z}/q\Bbb{Z})} \int_{\Gamma(1)\backslash {\rm SL}_n(\Bbb{R})} \Big|F(z, w; \bar{g}) - \frac{\tilde{f}(\rho)}{\text{vol}(\Gamma(q) \backslash \mathcal{H})}\Big|^2 \dd\mu(w)$$
   where $\dd\mu$ is the usual Haar measure on ${\rm SL}_n(\Bbb{R})$. We will estimate $W(I_n)$ from above and below. On the one hand, let $N(q, T)$ denote the number of $\bar{g} \in {\rm SL}_n(\Bbb{Z}/q\Bbb{Z})$ \emph{without} a lift of norm $\leq 4nT$. Then for $w$ in a small ball about $I_n$ we have $F(I_n, w; \bar{g}) = 0$, and hence by \eqref{sph-rho} and \eqref{vqnq} we conclude 
   \begin{equation}\label{lowerW}
   W(I_n) \gg N(q, T) \cdot \frac{T^{2n(n-1)}}{\mathcal{V}^2_q} .
   \end{equation}
   
  On the other hand, %,  the function $F$ lives on $\Gamma(q) \backslash {\rm SL}_n(\Bbb{R})/{\rm SO}_n(\Bbb{R}) = \Gamma(q) \backslash \mathcal{H}$ in the first two coordinates. 
  we apply the pretrace formula \eqref{pretrace} to $F(z, w; \bar{g})$, subtract the contribution of the trivial representation (constant function), open the square and combine the $\bar{g}$-sum with the integral to obtain
   $$W(z) = \int^{\ast}_{\Gamma(q)}    |\tilde{f}(\mu_{\varpi}) |^2 |\varpi(z)|^2 \dd\varpi$$
   where the $\ast$ indicates that the trivial representation is omitted. We let $P = (qT)^{\varepsilon}$ and distinguish the contribution $\| \mu_{\varpi} \| \leq P $ and $\| \mu_{\varpi} \| > P$. Correspondingly we write $W(z) = W(z)_{\leq P} + W(z)_{> P}$. Using \eqref{sph} with sufficiently large $B$ (in terms of $\varepsilon$) and the simple bound
   $$ \underset{\| \mu \| \leq R}{\int_{\Gamma(q)} }  |\varpi(I_n)|^2 \dd\varpi \ll R^{n(n+1)/2 - 1}$$
  for $R \geq 1$ (which is another application of the pretrace formula -- in fact any polynomial bound would suffice), we obtain
  \begin{equation}\label{largeP}
  W(I_n)_{> P} \ll T^{n(n-1) +\varepsilon}.
  \end{equation}
  Since $\Gamma(q)$ is normal in $\Gamma(1)$, we have $W(I_n) = W(g)$ for $g \in \Gamma(1) = {\rm SL}_n(\Bbb{Z})$, and so
   $$W(I_n)_{\leq P} =\frac{1}{\mathcal{V}_q} \sum_{\bar{g} \in {\rm SL}_n(\Bbb{Z}/q\Bbb{Z})} W(g)_{\leq P}  = \frac{1}{\mathcal{V}_q} \sum_{\bar{g} \in {\rm SL}_n(\Bbb{Z}/q\Bbb{Z})} \underset{\| \mu_{\varpi} \| \leq P}{ \int^{\ast }_{\Gamma(q)} }  |\tilde{f}(\mu_{\varpi}) |^2 |\varpi(g)|^2 \dd\varpi.  $$   
 By Lemma \ref{local} and  \eqref{sph} we conclude
 $$W(I_n)_{\leq P} \ll (Tq)^{\varepsilon} \frac{T^{n(n-1)  }}{\mathcal{V}_q}\underset{\| \mu_{\varpi} \| \leq P}{ \int^{\ast }_{\Gamma(q)} }   T^{2n \| \Re \mu_{\varpi} \|} \dd \varpi.$$
 In order to apply Theorem \ref{eisen}, we choose $T_0 = (P^{-K} q^{n+1})^{1/n} \gg P^{-K/n}\mathcal{V}_q^{1/n(n-1)} $ and write 
 $$T^{2n \| \Re \mu_{\varpi} \|} \leq T_0 ^{2n \| \Re \mu_{\varpi} \|} \Big(\frac{T}{T_0}\Big)^{2n  \mu^{\ast}}$$
 using the spectral gap \eqref{gap}. 
  Thus by Theorem \ref{eisen} we get
 $$W(I_n)_{\leq P} \ll (Tq)^{\varepsilon} \frac{T^{n(n-1)}}{\mathcal{V}_q} \mathcal{V}_q \Big(\frac{T}{T_0}\Big)^{2n \mu^{\ast}}.$$
 This dominates \eqref{largeP}. 
 Combining this with  \eqref{lowerW}, we obtain
 $$\frac{N(q, T)}{\mathcal{V}_q}   \leq (Tq)^{\varepsilon} \frac{\mathcal{V}_q}{T^{n(n-1)}} \Big(\frac{T}{T_0}\Big)^{2n \mu^{\ast}} \ll (Tq)^{\varepsilon} \Big(\frac{\mathcal{V}_q}{T^{n(n-1)}}  \Big)^{1/(n-1)}  \ll \mathcal{V}_q^{-\delta}$$
 if   $T \gg q^{1+ 1/n + \varepsilon}$ and $\delta = \delta(\varepsilon) > 0$ is sufficiently  small.

 \subsection{A counting problem}
 
The argument for the proof of Theorem \ref{thm2} is very similar to the preceding proof. With the same choice of $f$ as in \eqref{deff} we have
$$\sum_{\substack{\gamma \in \Gamma(q)\\ \| \gamma \| \leq T}} 1 \leq \sum_{\gamma \in \Gamma(q)} f(\gamma) = \frac{1}{\mathcal{V}_q} \sum_{\bar{g} \in \text{SL}_n(\Bbb{Z}/q\Bbb{Z})} \sum_{\gamma \in \Gamma(q)} f(g^{-1} \gamma g)$$%% = \int_(q) \sum_{\varpi \in \text{ONB}(\pi)} |\varpi($$   
where as before  we used that $\Gamma(q)$ is normal in ${\rm SL}_n(\Bbb{Z})$. Applying the pretrace formula \eqref{pretrace} and using also Lemma \ref{local} and \eqref{sph}, we obtain
$$\sum_{\substack{\gamma \in \Gamma(q)\\ \| \gamma \| \leq T}} 1 \leq \frac{1}{\mathcal{V}_q} \sum_{\bar{g} \in \text{SL}_n(\Bbb{Z}/q\Bbb{Z})} \int_{\Gamma(q)}  |\varpi(g)|^2\tilde{f}(\mu_{\varpi}) \dd \varpi \ll \frac{(Tq)^{\varepsilon}}{\mathcal{V}_q}   \int_{\Gamma(q)}  T^{n(\frac{n-1}{2} +  \| \Re \mu_{\varpi} \|)}\dd \varpi.$$
Again we choose $T_0 = (P^{-K} q^{n+1})^{2/n} \gg P^{-2K/n}\mathcal{V}_q^{2/n(n-1)} $ and apply Theorem \ref{eisen} in combination with \eqref{gap}. Then the above is
$$\ll(Tq)^{\varepsilon} \frac{T^{n(n-1)/2  }}{\mathcal{V}_q} \mathcal{V}_q\Big(1 + \frac{T}{T_0}\Big)^{n(n-1)/2} = (Tq)^{\varepsilon} \Big(T^{n(n-1)/2} + \frac{T^{n(n-1)}}{\mathcal{V}_q}\Big)$$
as claimed. \\

 \textbf{Acknowledgement:} The authors would like to thank S.\ Jana and A.\ Kamber for useful discussions on their work \cite{JK}. We thank the referees for a careful reading of the manuscript.

\end{document}